\documentclass[a4paper,fleqn]{cas-sc}

\usepackage[numbers]{natbib}
\usepackage{subcaption}
\usepackage{amsthm}
\usepackage{amsmath}
\usepackage{physics}
\usepackage{float}
\usepackage{graphicx}
\usepackage{ulem}
\usepackage{placeins}
\usepackage{algorithm}
\usepackage{algpseudocode}
\usepackage{tikz-cd}
\usepackage{tikz}
\usetikzlibrary{arrows.meta, positioning, calc}
\usepackage{subcaption} % works well with Elsevier templates

\theoremstyle{definition}
\newtheorem{definition}{Definition}

\def\tsc#1{\csdef{#1}{\textsc{\lowercase{#1}}\xspace}}
\tsc{WGM}
\tsc{QE}
\newtheorem{theorem}{Theorem}[section]

\usepackage{xcolor}

\begin{document}
\let\WriteBookmarks\relax
\def\floatpagepagefraction{1}
\def\textpagefraction{.001}

\shorttitle{Weighted Upwind Vector Kinetic Method}    

\shortauthors{Michael W. Brown, Jehanzeb Chaudhry, John N. Shadid}  

\title [mode = title]{A Weighted Upwind Vector Kinetic Lattice Boltzmann Method For Hyperbolic Conservation Laws}  




\author[1,2]{Michael W. Brown}%[<options>]

\fnmark[1]

\ead{mwbrow@sandia.gov}



\affiliation[1]{organization={Mathematics and Statistics Department, The University of New Mexico},
            addressline={311 Terrace Street NE}, 
            city={Albuquerque},
%          citysep={}, % Uncomment if no comma needed between city and postcode
            postcode={87106}, 
            state={NM},
            country={USA}}

\author[1]{Jehanzeb Chaudhry}%[<options>]

% Footnote of the second author
\fnmark[2]

% Email id of the second author
\ead{jehanzeb@unm.edu}



\author[1,2]{John N. Shadid}%[<options>]

% Footnote of the third author
\fnmark[3]

\cormark[1]

% Email id of the third author
\ead{jnshadi@sandia.gov}



\affiliation[2]{organization={Computational and Applied Math Department, Sandia National Laboratories},
            addressline={MSCO1 115}, 
            city={Albuquerque},
%          citysep={}, % Uncomment if no comma needed between city and postcode
            postcode={87185}, 
            state={NM},
            country={USA}}


\begin{abstract}
Vector kinetic lattice Boltzmann (VKLB) methods have recently emerged as a promising framework for solving hyperbolic partial differential equation (PDE) systems. VKLB discretizes Boltzmann-type equations using a discrete set of lattice velocities, enforces discrete moment constraints, and carefully defines equilibrium distribution functions. In this work, we introduce novel upwinded equilibrium distribution functions constructed from conservation variables and numerical fluxes derived via continuous flux vector splitting based on the eigen-decomposition of the flux Jacobian. This formulation enables the weighted-upwind VKLB equilibrium to be applied broadly to general hyperbolic systems. The method is  verified on a set of challenging hyperbolic systems that includes the shallow water, Euler and ideal magnetohydrodynamics (MHD) equations. The proposed method demonstrates improved stability, reduced error norms, and sharper shock resolution across increasingly complex verification and benchmark problems.
\end{abstract}

% Keywords
\begin{keywords}
 \sep lattice Boltzmann
 \sep flux splitting methods
 \sep conservations laws
 \sep vector kinetic methods
\end{keywords}

\maketitle

\section{Introduction } 
Numerous important physical systems of interest in science and technology are modeled by conservation law PDE systems, and consequently such equations occupy a central role in computational physics and mathematics. Examples include chemical species transport, shallow water equations, linear and nonlinear elasticity for solid dynamics, the Euler equations of gas dynamics, and magnetohydrodynamics (MHD) plasma models. 

Successful numerical methods for these systems are designed to preserve the conservative structure of such PDEs while also providing stable and accurate techniques for capturing the formation of shocks, rarefactions, contact discontinuities, and steep unresolved internal and boundary layers \cite{Leveqe_Finite_Volume_Book, Toro_2009_Riemann_Solvers, Shu_Wang_Weno_Schemes, Cockburn_Discontinuous_Galerkin_Method,ern2004theory,kuzmin2012flux}. 

The finite volume method and the finite element method are two widely used traditional numerical methods for hyperbolic PDEs.  Finite volume methods are natural for conservation laws and are constructed directly from cell interface fluxes and therefore enforce conservation within the computational domain by design \cite{Leveqe_Finite_Volume_Book, Toro_2009_Riemann_Solvers, Leer_1979_Second_Order}. Properly defined numerical fluxes lead to robust solvers for hyperbolic problems, as seen through the development of approximate Riemann solvers \cite{ROE_Approx_Riemann_Solver}, flux splitting strategies \cite{Steger_1981_Euler_Splitting, Borah_2016_Ideal_MHD_Flux_Splitting, Toro_2022_Shallow_Water_Flux_Splitting}, and high resolution reconstruction methods such as ENO and WENO schemes \cite{Shu_Wang_Weno_Schemes, Shu_eno_schemes}. Discontinuous Galerkin finite element methods provide another powerful framework, combining the flexibility of finite element approximations with element wise conservation and high order accuracy \cite{Cockburn_Discontinuous_Galerkin_Method}. Fintie element methods based on invariant domain preserving (IDP) methods with graph-based numerical viscosity techniques and algebraic flux correction (AFC) methods also offer a systematic approach to enforce monotonicity and positivity preservation, and nonlinear stability properties in finite element methods. \cite{guermond2018second,kuzmin2012flux,kuzmin2020monolithic}.

Recently, lattice Boltzmann (LB) schemes are receiving increasing attention as an additional framework for numerically approximating systems of conservation laws and general governing PDE balance law systems. Reasons include low dispersion / dissipation errors, inherent parallelism and simple data access / communication patterns allowing vectorization on accelerators for the streaming (advection) step, exploitation of data locality in the collisional step, and extensibility to models for complex multiphysics systems with the non-linear behavior confined to the local collision operator.  A traditional lattice Boltzmann (LB) method evolves scalar particle distribution functions on a uniform mesh with a set of discrete lattice velocities, and macroscopic fluid quantities are recovered through discrete moments of such distribution functions \cite{Succi_LBM_2001, chen_shiyi_doolan_LB_fluid_flow, Y_H_Qian_1992}. In the context of the Navier-Stokes (NS) system, a traditional scalar LB method is used to recover the NS system in a low Mach number limit through appropriately constructed polynomial equilibrium distribution functions \cite{Kruger_2017_LB_Graduate_Book}. 
%The stream and collide algorithm of the LB method is attractive because of the local nonlinear collision operator, exact %advection operator, and suitability for highly parallel computation. 
Additionally, the LB framework has been extended to include models for multiphase and multi-component flows \cite{Xiaowen_chen_multiphase_flows}, thermal flows \cite{He_Chen_LB_Thermal_Model}, and formulations for magnetohydrodynamics \cite{Dellar_LB_MHD}. The traditional scalar algorithm and the various model extensions demonstrate the extensibility of the lattice Boltzmann framework. 

Vector Kinetic Lattice Boltzmann (VKLB) schemes build on the LB framework and are tailored for the solution of conservation laws~\cite{Anandan_2024_VKLB_Upwinding_Source_Term}. The vector kinetic models use a finite set of discrete velocity vectors together with vector valued  distribution functions so that discrete moments recover the conserved variables and physical fluxes of the target hyperbolic system. VKLB preserves the stream and collide structure of LB methods while approximating systems such as shallow water flow, compressible hydrodynamics, and ideal MHD. The VKLB framework for hyperbolic systems is developed from \cite{Natalini_Multi_Dimensional_Systems_Discrete_Systems}, with corresponding discrete moment structures, Chapman Enskog analysis, and a subcharacteristic condition \cite{Anandan_2024_VKLB_Upwinding_Source_Term, baty2023robust, Wissocq_2024_Positive_Preserving_VKLB}. Recent work has further extended VKLB schemes to include source terms, upwinding, physical diffusion, and MHD \cite{Wissocq_2024_Positive_Preserving_VKLB,Wissoq_2024_VKLB_Diffusion, Anandan_2024_VKLB_Upwinding_Source_Term,baty2023robust}. There has also been progress in establishing important properties such as positivity preservation, entropy-compatible behavior, consistency with the macroscopic system, and total variation boundedness \cite{Wissocq_2024_Positive_Preserving_VKLB,Anandan_2024_VKLB_Upwinding_Source_Term, boolakee_geier_2024_linear_elastodynamiocs, guillon2024stabilityanalysisvectoriallatticeboltzmann, bellotti2025consistencystabilityboundaryconditions, dubois2014simulationstrongnonlinearwaves, aregbadriollet2025equilibriumboundaryconditionsvectorial}. Unlike LB schemes, in the context of hydrodynamics the VKLB framework is not restricted to the low-Mach-number limit \cite{baty2023robust}.

The purpose of this work is to introduce a weighted upwind vector kinetic lattice Boltzmann (WU-VKLB) scheme constructed from  the flux splitting distribution functions developed in \cite{Anandan_2024_VKLB_Upwinding_Source_Term, Natalini_Multi_Dimensional_Systems_Discrete_Systems}. Rather than constructing a flux splitting tailored to one specific system, we appropriate a flux splitting from the finite volume framework based solely on the eigenstructure of the underlying hyperbolic system. We then develop a continuous version of this splitting, which helps in generalizing the centered flux equilibrium distribution set and the upwind equilibrium distribution set found in \cite{Anandan_2024_VKLB_Upwinding_Source_Term}. We then demonstrate, through development of an equivalent finite difference formulation, that including the conserved vector in the equilibrium distribution functions introduces a useful additional dissipative term with respect to the conserved variables. We also establish an equivalent finite volume formulation, showing that the method satisfies a discrete conservation law in which internal numerical fluxes cancel telescopically \cite{Leveqe_Finite_Volume_Book}. The weighted upwind distribution function set is tested on verification and benchmark problems for shallow water, hydrodynamics, and ideal MHD systems. The numerical results show the weighted upwind construction improves stability relative to the discontinuous upwind distribution function set while also improving accuracy relative to both the centered flux and discontinuous upwind distribution function sets. 

The remainder of this paper is organized as follows. In \autoref{s:2}  we review developments of the VKLB method, understanding the kinetic formulation, the centered flux equilibrium distribution function set and using a discontinuous flux splitting to construct an upwind equilibrium distribution function set. In \autoref{s:3}, we present the weighted upwind distribution function set by constructing a smooth flux splitting. We examine boundary conditions, and CFL constraints for a VKLB scheme. In Section \autoref{s:4}, we analyze the weighted upwind scheme by looking at equivalent finite difference and finite volume formulations. We also examine entropy results to constrain parameters introduced in the weighted upwind scheme. In \autoref{s:5}, we present numerical results for shallow water, Euler system of gas dynamics, and ideal MHD. In  \autoref{s:7}, we summarize the main conclusions of the theoretical and numerical work, as well as discuss future work.

% Main text
\section{Vector Kinetic Lattice Boltzmann Framework}\label{s:2}
The vector kinetic lattice Boltzmann (VKLB) scheme used to approximate a system of conservation laws is constructed from a set of semi-discrete vector kinetic equations posed on a uniform grid. In this section, we review the developments of the VKLB framework, summarizing the governing kinetic equation, the discrete moment structure, and various equilibrium distribution function sets in one and two dimensions \cite{Anandan_2024_VKLB_Upwinding_Source_Term, Natalini_Multi_Dimensional_Systems_Discrete_Systems}. This provides the mathematical foundation for the weighted upwind VKLB scheme developed in \autoref{s:3}.
\subsection{Vector Kinetic Lattice Boltzmann Equation}
Let $\Omega\subset\mathbb{R}^D$ denote a domain with coordinates $\mathbf{x}=(x_1,\dots,x_D)$, $\partial_t$ denote the partial time derivative and $\partial_{x_d}$ denote the partial spatial derivative in direction $d \in \{1, 2, \ldots, D\}$. Consider the hyperbolic system
\begin{subequations}\label{eq: hyperbolic-system}
\begin{align}
  \partial_t U(\mathbf{x},t)
  + \sum_{d=1}^{D}\partial_{x_d}F^{(d)}\bigl(U\bigr) 
  &= 0, \qquad \mathbf{x} \in \Omega, \  t \in (0,T],
  \label{eq:hyperbolic-system-a}\\
  U(\mathbf{x},0) &= U_0(\mathbf{x}),  \qquad \mathbf{x} \in \Omega,
  \label{eq:hyperbolic-system-b}
\end{align}
\end{subequations}
where $T > 0$, $U:\Omega\times[0,T]\to\mathbb{R}^p$ is the vector of conserved quantities and $F^{(d)}(U):\mathbb{R}^p\to\mathbb{R}^p$ are flux functions in direction $d$. In a vector kinetic lattice Boltzmann (VKLB) framework, vector valued distribution functions \(f_q(\mathbf{x},t)\in\mathbb{R}^p\) and discrete velocity vectors \(\mathbf{v}_q \in \mathbb{R}^D, \, q \in \{1, 2, \ldots, Q\}\) are introduced, where each $f_q(\mathbf{x},t)$ streams along a corresponding velocity vector $\mathbf{v}_q$ \cite{Anandan_2024_VKLB_Upwinding_Source_Term}. For each velocity vector $\mathbf{v}_q$, the nonzero components are equal in magnitude to a speed $\xi \equiv \xi(t)$, where $\xi(t)$ represents the streaming speed along coordinate directions at a time $t \in (0,T]$ \cite{Anandan_2024_VKLB_Upwinding_Source_Term, Natalini_Multi_Dimensional_Systems_Discrete_Systems, Wissocq_2024_Positive_Preserving_VKLB}. The distribution function set $\{f_q(\vb{x}, t)\}_{q=1}^Q$ satisfies a system of semi-discrete approximate Boltzmann equations, where the collision operator is the Bhatnagar--Gross--Krook (BGK) collision approximation \cite{Bhatnager_1954_BGK_Approximation},
\begin{subequations}\label{eq: discrete-kinetic-equation}
\begin{align}
  \partial_t f_q
  + \sum_{d=1}^D \partial_{x_d}\bigl(v_q^d\,f_q\bigr)
  &= -\frac{1}{\varepsilon}\bigl(f_q - f_q^{eq}(U)\bigr),
  \label{eq:discrete-kinetic-equation-a}\\
  f_q(\mathbf{x},0) &= f_q^{eq}\bigl(U(\mathbf{x},0)\bigr),
  \quad q=1,\dots,Q.
  \label{eq:discrete-kinetic-equation-b}
\end{align}
\end{subequations}
Here, $f_q^{eq}\left(U\right) \in \mathbb{R}^p$ are equilibrium distribution functions, the parameter $\varepsilon>0$ is a relaxation time (in a non-dimensional context this parameter corresponds to the Knudsen number \cite{Wissoq_2024_VKLB_Diffusion}), and $v_q^d \in \mathbb{R}$ is the d\textsuperscript{th} component of the velocity vector $\vb{v}_q$. The relationship between the vector of conserved quantities $U$ and the distribution function sets $\{f_q\}_{q=1}^{Q}, \, \{f_q^{eq}\}_{q=1}^Q$ is established by defining the \textit{first discrete moment},
\begin{equation}\label{eq: discrete-first-moment}
  \sum_{q=1}^Qf_q
  \;=\;\sum_{q=1}^Qf_q^{eq}
  \;=\;U.
\end{equation}
Summing \eqref{eq:discrete-kinetic-equation-a} over all $q$ yields the first moment equation,
\begin{equation}\label{eq: kinetic-first-moment}
  \partial_t\!\sum_{q=1}^Q f_q
  + \sum_{d=1}^D\partial_{x_d}\!\sum_{q=1}^Q\bigl(v_q^d\,f_q\bigr)
  = -\frac{1}{\varepsilon}
    \Bigl(\sum_{q=1}^Qf_q - \sum_{q=1}^Qf_q^{eq}\Bigr).
\end{equation}
Note that \eqref{eq:discrete-kinetic-equation-a}  implies $f_q\to f_q^{eq}$ in the limit $\varepsilon\to 0$. Thus, if the \textit{second discrete moment} of the equilibrium distribution functions satisfies
\begin{equation}\label{eq: discrete-second-moment}
  \sum_{q=1}^Qv_q^d\,f_q^{eq}(U)
  = F^{(d)}(U),
  \quad
  d=1,\dots,D,
\end{equation}
then using  \eqref{eq: discrete-first-moment} and \eqref{eq: discrete-second-moment} in  \eqref{eq: kinetic-first-moment} recovers the hyperbolic system \eqref{eq: hyperbolic-system} in the limit $\varepsilon\to 0$ \cite{Anandan_2024_VKLB_Upwinding_Source_Term}. Discretizing the semi-discrete equation \eqref{eq:discrete-kinetic-equation-a} using the method of characteristics, the lattice Boltzmann equation,
\begin{equation}\label{lattice-boltzmann-algorithm}
     f_q(\mathbf{x}, t) = (1 - \omega)f_q(\mathbf{x} - \mathbf{v}_q \Delta t, t - \Delta t) + \omega f_q^{eq}\left(U(\mathbf{x} - \mathbf{v}_q \Delta t, t - \Delta t)  \right)
\end{equation}
is constructed, where $\omega = \Delta t/\varepsilon$ is a non-dimensional relaxation parameter. The parameter \(\omega\) has a direct influence on the stability of the VKLB scheme as discussed in \autoref{ss: properties}.
\subsection{Equilibrium Distribution Functions}\label{ss: equilibrium_distribution_functions}
The BGK collision approximation \cite{Bhatnager_1954_BGK_Approximation} introduces a set of equilibrium distribution functions, $\{f_q^{eq}\}_{q=1}^Q$. In order for the coupled semi--discrete kinetic equations \eqref{eq:discrete-kinetic-equation-a} to approximate the hyperbolic system \eqref{eq: hyperbolic-system}, the equilibrium distribution function set needs to satisfy the necessary discrete moments \eqref{eq: discrete-first-moment} and \eqref{eq: discrete-second-moment}. We illustrate the construction of equilibrium distribution sets in 1D for simplicity. Letting  \(D=1\),  \(Q=3\) and denoting \(F^{(1)}=F\), the hyperbolic system \eqref{eq: hyperbolic-system} reduces to
\begin{equation}\label{eq: one_dimension_hyperbolic_system}
  \partial_t U \;+\;\partial_{x} F(U)\;=\;0.
\end{equation}
Denote discrete lattice sites by $\{x_i\}_{i=1}^N$ with a fixed spacing $\Delta x$. Let \(\xi = \Delta x / \Delta t\) denote the streaming speed. We define the one--dimensional discrete velocities by \(\mathbf{v}_1 = (\xi)\), \(\mathbf{v}_2 = (0)\), and \(\mathbf{v}_3 = (-\xi)\), with corresponding distribution functions \(f_1\), \(f_2\), and \(f_3\), as shown in \autoref{fig: d1q3_lattice_structure}.
 \begin{center}
  \includegraphics[width=0.65\columnwidth]{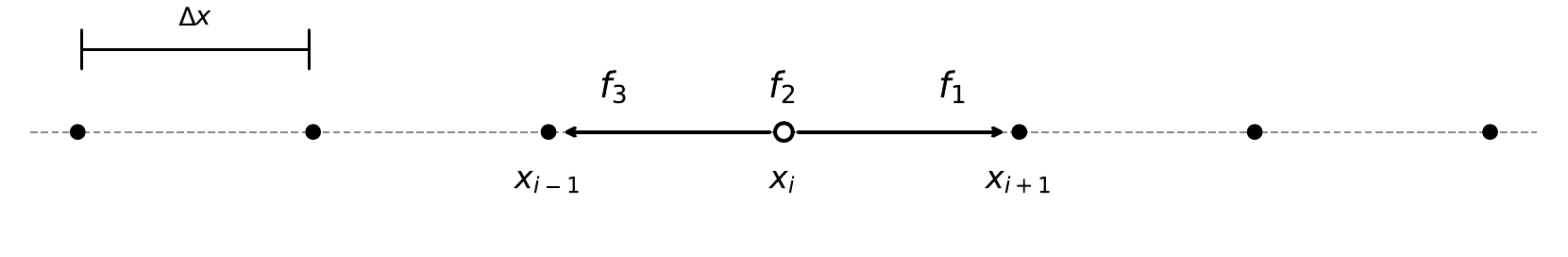}
  \captionsetup{hypcap=false}
  \captionof{figure}{One dimensional lattice (\(D=1\)) with three discrete velocities (\(Q=3\)). The distribution functions \(f_1\), \(f_2\), and \(f_3\) have corresponding velocity vectors \(\mathbf{v}_1 = (\xi)\), \(\mathbf{v}_2 = (0)\), and \(\mathbf{v}_3 = (-\xi)\). Here, \(\Delta x\) is the distance traversed by the flow of the distribution functions $f_1, \, f_3$ in the positive and negative x directions during a time increment \(\Delta t\).}
  \label{fig: d1q3_lattice_structure}
\end{center}
With only the constraints \eqref{eq: discrete-first-moment}, \eqref{eq: discrete-second-moment}, any number of equilibrium distribution functions can be constructed. A common distribution function set is the \textbf{\textit{centered flux distribution function set}} \cite{Anandan_2024_VKLB_Upwinding_Source_Term}, 
\begin{align}
  f^{eq}_1(U)
  &= \frac{1}{3}U + \frac{1}{2\xi}\,F\bigl(U\bigr),
  \label{eq: centered-f1}\\
  f^{eq}_2(U)
  &= \,\frac{1}{3}U,
  \label{eq: centered-f2}\\
  f^{eq}_3(U)
  &= \frac{1}{3}U - \frac{1}{2\xi}\,F\bigl(U\bigr).
  \label{eq: centered-f3}
\end{align}
Upwinding can be introduced into the VKLB scheme by using a flux splitting of the physical flux \(F(U)\) within the equilibrium distribution functions. A \textit{flux splitting} of \(F(U)\) is a decomposition
\begin{equation}\label{eq:flux-splitting-definition}
  F(U) = F^{(+)}(U) + F^{(-)}(U).
\end{equation}
 where  \(F^{(+)}(U), \; F^{(-)}(U) : \mathbb{R}^p \rightarrow \mathbb{R}^p \).
%
% The \textit{upwind distribution function set} \cite{Anandan_2024_VKLB_Upwinding_Source_Term} incorporates a flux splitting of $F(U)$ by,
An equilibrium distribution set based on the flux splitting \eqref{eq:flux-splitting-definition} is given by \cite{Anandan_2024_VKLB_Upwinding_Source_Term},
\begin{align}
  f^{eq}_1(U)
  &= \frac{1}{\xi}\,F^{(+)}\bigl(U\bigr),
  \label{eq: upwinding-f1}\\
  f^{eq}_2(U)
  &= \,U
    \;-\;\frac{1}{\xi}\bigl[F^{(+)}\bigl(U\bigr)
    -F^{(-)}\bigl(U\bigr)\bigr],
  \label{eq: upwinding-f2}\\
  f^{eq}_3(U)
  &= -\frac{1}{\xi}\,F^{(-)}\bigl(U\bigr).
  \label{eq: upwinding-f3}
\end{align}
The second discrete moment constraint \eqref{eq: discrete-second-moment} requires $F^{+}\bigl(U\bigr) + F^{-}\bigl(U\bigr) = F\bigl(U\bigr)$, which is naturally satisfied from the flux splitting definition \eqref{eq:flux-splitting-definition}. In the context of finite volume methods a variety of splitting methods have been proposed for specific hyperbolic systems such as shallow water \cite{Toro_2022_Shallow_Water_Flux_Splitting, Rebollo_2003_Shallow_Water_Flux_Splitting}, Euler equations \cite{Toro_2012_Euler_Flux_Splitting, Chu_2025_Euler_Flux_Splitting}, and ideal magnetohydrodynamics (MHD) \cite{Borah_2016_Ideal_MHD_Flux_Splitting, Zhang_2024_Ideal_MHD_Flux_Splitting}.

We use a general flux splitting method, seen in a finite volume setting \cite{Leveqe_Finite_Volume_Book}, that lends itself to any hyperbolic system. This splitting is based on the eigenvalues of the Jacobian matrix $\partial_UF(U)$. Since the system \eqref{eq: hyperbolic-system} is assumed to be hyperbolic, the Jacobian has an eigenvalue decomposition
\begin{equation}\label{eq: eigenvalue_decomposition}
  \partial_{U}{F}(U)
  =P(U)\,\Lambda(U)\,P(U)^{-1},
\end{equation}
where
\begin{itemize}
  \item \(P(U)=[\,r^{1},\dots,r^{p}\,]\) matrix of linearly independent right‐-eigenvectors,
  \item \(\Lambda(U) = \text{diag}(\lambda^{1},\dots,\lambda^{p})\) diagonal matrix of real valued eigenvalues (characteristic values).
\end{itemize}
Denote the evaluation of the flux $F(U)$ at a lattice site $x_i$ by $F_i$, that is \(F_i=F(U(x_i))\).  The flux vector $F_i$ is expanded with respect to the right eigenvectors of the decomposition \eqref{eq: eigenvalue_decomposition},
\begin{equation}\label{eq: expansion_coefficients}
    \Phi_i
  =\bigl(\phi^{1}_i,\dots,\phi^{p}_i\bigr)^{\!T}
  = P_i^{-1}F_i \; \Rightarrow \;F_i
  =\sum_{j=1}^p\phi^{j}_i\,r^{j}_i.
\end{equation}
Each coefficient $\phi_i^j$ is split based on the sign of the corresponding eigenvalue \(\lambda_i^j\),
\begin{equation}\label{eq: eigenvalue-split-discrete}
  \phi_i^{j,+}
  =\begin{cases}
    \phi_i^{j},&\lambda^{j}_i\geq0,\\
    0,&\lambda^{j}_i<0,
  \end{cases}
  \quad
  \phi_i^{j,-}
  =\begin{cases}
    0,&\lambda^{j}_i\geq0,\\
    \phi_i^{j},&\lambda^{j}_i<0.
  \end{cases}
\end{equation}
The \textit{partial fluxes} are constructed as
\begin{equation}\label{eq: discontinuous_partial_fluxes}
  F_i^{(+)}
  =\sum_{j=1}^p\phi_i^{j,+}\,r_i^{j},
  \quad
  F_i^{(-)}
  =\sum_{j=1}^p\phi_i^{j,-}\,r_i^{j}.
\end{equation}
At each lattice site $x_i$, the partial fluxes \eqref{eq: discontinuous_partial_fluxes} satisfy the flux splitting definition \eqref{eq:flux-splitting-definition}, \(F_i=F_i^{(+)}+F_i^{(-)}\). For hyperbolic systems which are \textit{homogeneous of degree one}, that is $F(U) = (\partial_UF)\cdot U$, the splitting \eqref{eq: eigenvalue-split-discrete} is equivalent to the Steger-Warming splitting \cite{Leveqe_Finite_Volume_Book, Steger_1981_Euler_Splitting}. 

We call the  distribution function set \eqref{eq: upwinding-f1} - \eqref{eq: upwinding-f3} using the partial fluxes \eqref{eq: discontinuous_partial_fluxes} a \textbf{\textit{discontinuous upwind distribution function set}}. Numerical results in \autoref{s:5} show instabilities occur for the discontinuous upwind distribution function set when an eigenvalue of the Jacobian matrix changes sign or if there is not sufficient numerical diffusion in the system.

\subsection{Stability Constraints for the Vector Kinetic Lattice Boltzmann Distribution Functions}\label{ss: properties}
A Chapman-Enskog multiscale analysis of \eqref{lattice-boltzmann-algorithm} recovers the the conservation laws \eqref{eq: hyperbolic-system} albeit with an introduction of a first order diffusive term~\cite{Anandan_2024_VKLB_Upwinding_Source_Term}. This term depends on both the relaxation parameter \( \omega \), the equilibrium distribution function set \(\{f_q^{eq}(U)\}_{q = 1}^{Q}\) and the streaming speed \( \xi \). For simplicity, let \(D = 1\) and \(Q = 3\) as illustrated in \autoref{fig: d1q3_lattice_structure}. The first order diffusive term takes the form
\begin{equation}\label{subcharacteristic_condition}
 E_T =
\Delta t\,
\left(\tfrac{1}{\omega}-\tfrac{1}{2}\right)\,
\partial_{x}\!\left(
\left[
\partial_U\sum_{q=1}^{3} \mathbf{v}_q^2 f_q^{eq}(U)
-\left(\partial_U F\right)^2
\right]
\,\partial_{x}U\right) + \mathcal{O}(\Delta t^2).
\end{equation}
For the diffusive operator to be well posed, the  corresponding coefficients of the $O(\Delta t)$ term need to be positive. The first coefficient, \(\left(\tfrac{1}{\omega} - \tfrac{1}{2}\right) \) constrains the relaxation parameter to \( \omega \in (0,2] \). Within this range, the choice of \( \omega \) determines the formal order of accuracy of the VKLB scheme. When \( \omega < 2 \), the method retains a nonzero first order diffusive contribution and is therefore first order accurate. In contrast, setting \( \omega = 2 \) eliminates this leading order diffusion term, yielding a second order accurate scheme. Thus, \( \omega = 2 \) is preferred when higher order accuracy is the objective (e.g. for smooth solutions), whereas \( \omega < 2 \) intentionally retains numerical diffusion, providing stability in the presence of shocks and discontinuities.

As demonstrated in \cite{Anandan_2024_VKLB_Upwinding_Source_Term} restricting $0 < \omega \leq 1$ associates an H--inequality with the lattice Boltzmann equation \eqref{lattice-boltzmann-algorithm}, which is a stronger constraint than the positivity of the coefficient \(\left(\tfrac{1}{\omega} - \tfrac{1}{2}\right) \). In addition, the distribution function set \eqref{eq: upwinding-f1}--\eqref{eq: upwinding-f3} has several desirable properties when \( \omega \leq 1 \), including positivity, consistency, and total variation boundedness \cite{Anandan_2024_VKLB_Upwinding_Source_Term}. For shock-dominated solutions, we set \(\omega = 1\) to preserve the desired theoretical properties while minimizing numerical diffusion.

The streaming speed \( \xi \) is constrained through the remaining diffusive coefficient, commonly referred to as the \textit{subcharacteristic condition},
\begin{equation}\label{eq: subcharacteristic_condition}
\partial_U\sum_{q=1}^{3} \mathbf{v}_q^2 f_q^{eq}(U)
-\left(\partial_U F\right)^2 > 0.
\end{equation}
For the centered flux distribution function set \eqref{eq: centered-f1}--\eqref{eq: centered-f3}, this condition can be evaluated explicitly and reduces to \cite{Anandan_2024_VKLB_Upwinding_Source_Term}
\[
\xi > \sqrt{\frac{3}{2}} \max_j\{|\lambda^j|\}.
\]
This leads to a constraint on the streaming speed,
\begin{equation}
    \eta = \sqrt{\frac{3}{2}}\left(\frac{\max_j\{|\lambda^j|\}}{\xi}\right) < 1
\end{equation}
However, for the discontinuous upwind distribution function set \eqref{eq: upwinding-f1}--\eqref{eq: upwinding-f3}, the subcharacteristic condition \eqref{eq: subcharacteristic_condition} does not admit a straightforward closed form constraint. In this case we appeal to the classical hyperbolic CFL constraint  \cite{Leveqe_Finite_Volume_Book} in \autoref{ss: CFL}, which provides a practical stability restriction for the moment systems \eqref{eq: hyperbolic-system} corresponding to general hyperbolic formulations. This CFL constraint is used in place of the subcharacteristic condition for the discontinuous upwind distribution function set \eqref{eq: upwinding-f1}--\eqref{eq: upwinding-f3} and the proposed weighted upwind distribution function set \eqref{eq: adjusted-f1}--\eqref{eq: adjusted-f3} introduced in \autoref{s:3}. The effectiveness of this choice is considered in the challenging verification and benchmark problems presented in \autoref{s:5}.

\section{Weighted Upwind Vector Kinetic Lattice Boltzmann Method}\label{s:3}
In this section we propose a \textit{weighted upwind vector kinetic lattice Boltzmann} (WU-VKLB) scheme. This equilibrium distribution function set is designed to combine the favorable stability properties of the centered flux distribution function set with the improved resolution of the discontinuous upwind distribution function set. In particular, while the centered flux distribution function set \eqref{eq: centered-f1} - \eqref{eq: centered-f3} provides a stable discretization, it introduces excessive numerical diffusion near sharp gradients. By contrast, the discontinuous upwind distribution function set \eqref{eq: upwinding-f1} - \eqref{eq: upwinding-f3} improves the numerical approximation in smooth regions and near discontinuities, but suffers from instabilities when an eigenvalue changes sign within the computational domain. The weighted upwind distribution function set combines these two limiting constructions in a manner that keeps the desired stability and accuracy of each scheme. 
\subsection{Weighted Upwind Distribution Functions}\label{ss: weighted_upwinding_set}
For simplicity, we first consider the case of one dimension,  \(D=1\) and \(Q = 3\). Let \(\mathbf{v}_1 = (\xi)\), \(\mathbf{v}_2 = (0)\), and \(\mathbf{v}_3 = (-\xi)\) denote the discrete velocities, with corresponding distribution functions \(f_1\), \(f_2\), and \(f_3\), as illustrated in \autoref{fig: d1q3_lattice_structure}. Consider the distribution function set given by
\begin{align}
  f^{eq}_1(U)
  &= cU + \frac{1}{\xi}\,\mathcal{F}^{(+)}\bigl(U\bigr),
  \label{eq: adjusted-f1}\\
  f^{eq}_2(U)
  &= (1 - 2c)\,U
    \;-\;\frac{1}{\xi}\bigl[\mathcal{F}^{(+)}\bigl(U\bigr)
    -\mathcal{F}^{(-)}\bigl(U\bigr)\bigr],
  \label{eq: adjusted-f2}\\
  f^{eq}_3(U)
  &= cU - \frac{1}{\xi}\,\mathcal{F}^{(-)}\bigl(U\bigr).
  \label{eq: adjusted-f3}
\end{align}
The terms $\mathcal{F}^{(+)}(U), \, \mathcal{F}^{(-)}(U)$ are flux splitting components of $F(U)$, that is $F(U) = \mathcal{F}^{(+)}(U) + \mathcal{F}^{(-)}(U)$. The introduction of the parameter $c \in \mathbb{R}$ is used in \autoref{s:4} to ensure entropy results for \eqref{eq: discrete-kinetic-equation} \cite{Wissocq_2024_Positive_Preserving_VKLB}. 
 The distribution function set \eqref{eq: adjusted-f1}--\eqref{eq: adjusted-f3} is a generalization of the discontinuous upwind and centered flux distribution function sets. Setting $c = 1/3$ and $\mathcal{F}^{(+)}(U) = \mathcal{F}^{(-)}(U) = \tfrac{1}{2}F(U)$ recovers the centered flux distribution function set \eqref{eq: centered-f1} - \eqref{eq: centered-f3}, while setting $c = 0$ and $\mathcal{F}^{(\pm)}(U) = F^{(\pm)}(U)$ recovers the discontinuous upwind distribution function set \eqref{eq: upwinding-f1}--\eqref{eq: upwinding-f3}. 
 
 Employing \eqref{eq: adjusted-f1}--\eqref{eq: adjusted-f3} we pursue the development of a continuous VKLB flux vector splitting \eqref{eq:flux-splitting-definition} that provides the benefits of the increased accuracy of upwinding while mitigating the stability issues often observed in finite volume methods when an eigenvalue crosses zero \cite{TANG2005507, Toro_2009_Riemann_Solvers, Zha_Bilgen_Euler_Splitting, Leveqe_Finite_Volume_Book}. These types of eigenvalue sign transitions can often be associated with phenomena such as transonic flow in the case of the Euler system. \cite{Leveqe_Finite_Volume_Book, Toro_2009_Riemann_Solvers, 1982LNP...170..507V}. Hence for the VKLB flux splitting components $\mathcal{F}^{\pm}(U)$ in \eqref{eq: adjusted-f1} - \eqref{eq: adjusted-f3}, we propose a \textbf{\textit{weighted upwind splitting}}, where the discontinuous flux splitting \eqref{eq: eigenvalue-split-discrete} is made continuous by weighting the coefficients using a smooth function $\alpha$. We take the function $\alpha$ to be parameterized by the eigenvalues of the hyperbolic system \eqref{eq: hyperbolic-system}. The coefficients in \eqref{eq: expansion_coefficients} are split at a lattice site $x_i$ as 
\begin{equation}\label{eq: eigenvalue-split-smooth}
  \psi_i^{j,+} = \alpha(\lambda_i^j)\phi_i^{j},
  \quad
  \psi_i^{j,-} = (1 - \alpha(\lambda_i^j))\phi_i^{j}
\end{equation}
where $\alpha \in [0,1]$ and $\alpha(0) = \tfrac{1}{2}$. The partial fluxes are constructed with the coefficients \eqref{eq: eigenvalue-split-smooth},
\begin{equation}\label{eq: weighted-flux-splitting}
      \mathcal{F}_i^{(+)}
      =\sum_{j}\psi_i^{j,+}\,r_i^{j},
      \quad
      \mathcal{F}_i^{(-)}
      =\sum_{j}\psi_i^{j,-}\,r_i^{j}.
\end{equation}
The partial fluxes \eqref{eq: weighted-flux-splitting} satisfy the flux decomposition definition \eqref{eq:flux-splitting-definition}, $F(U) = \mathcal{F}^{(+)}(U) + \mathcal{F}^{(-)}(U)$. The function $\alpha$ we consider for the weighted upwind splitting \eqref{eq: eigenvalue-split-smooth} is the sigmoid function,
\begin{equation}\label{eq: parameterization-function}
\alpha\left(\frac{\lambda}{|\lambda_{max}|}; k\right) = \frac{1}{1 + e^{-k\lambda/\max_{j}(|\lambda_j|)}}.
\end{equation}
\begin{center}
  \includegraphics[width=0.4\columnwidth]{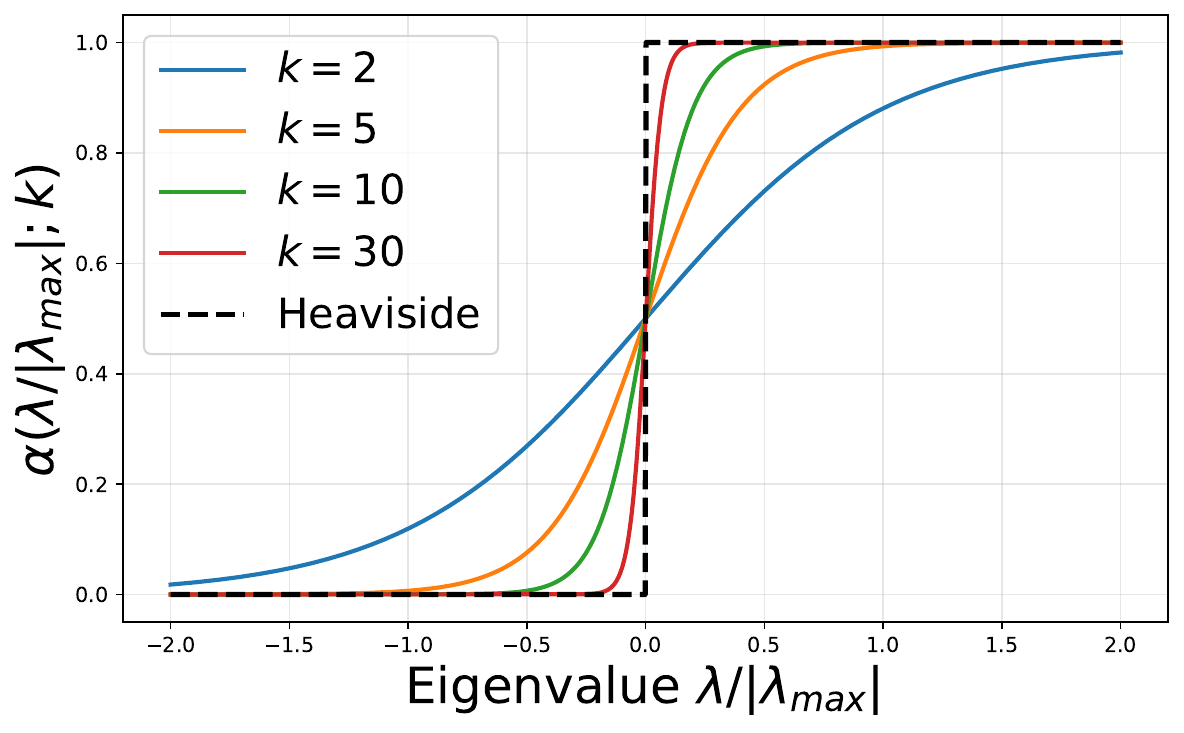}
  \captionsetup{hypcap=false}
  \captionof{figure}{Sigmoid transition function \(\alpha\left(\tfrac{\lambda}{|\lambda_{max}|};k\right)\) for $\lambda > 0$.}
  \label{fig:alpha-sigmoid}
\end{center}
We call the distribution function set \eqref{eq: adjusted-f1}--\eqref{eq: adjusted-f3} using the partial fluxes \eqref{eq: weighted-flux-splitting} a \textbf{\textit{weighted upwind distribution function set}}.

The eigenvalues in \eqref{eq: parameterization-function} are normalized so the resulting parameterization is scale invariant. The sigmoid function \eqref{eq: parameterization-function} allows the flux structure of the weighted upwind distribution function set \eqref{eq: adjusted-f1}--\eqref{eq: adjusted-f3} to vary continuously between the centered flux distribution function set \eqref{eq: centered-f1}--\eqref{eq: centered-f3} and the discontinuous upwind distribution set \eqref{eq: upwinding-f1}--\eqref{eq: upwinding-f3}. As an eigenvalue \(\lambda\) approaches zero, the partial fluxes \eqref{eq: weighted-flux-splitting} satisfies \(\mathcal{F}^{\pm}(U) \rightarrow \tfrac{1}{2}F(U)\). The partial fluxes \eqref{eq: weighted-flux-splitting} equal the discontinuous partial fluxes \eqref{eq: eigenvalue-split-discrete} as k goes to infinity, that is \(\mathcal{F}^{\pm}(U) \rightarrow F^{\pm}(U)\) as $k \rightarrow \infty$. 

\subsection{CFL Constraint}\label{ss: CFL}
Stability of any hyperbolic solver requires the numerical domain of dependence to include the physical one. Let \(\Delta x\) denote the spacing between lattice sites and \(\xi\) the nonzero streaming speed; both are illustrated in \autoref{fig: d1q3_lattice_structure}. For a system with \(p\) characteristic speeds \(\{\lambda^j\}_{j=1}^p\), the Courant–Friedrich–Lewy (CFL) condition reads \cite{Leveqe_Finite_Volume_Book}
\begin{equation}\label{eq: CFL-Leveque}
  \nu \;=\;\frac{\Delta t}{\Delta x}\,\max_{j}|\lambda^j|\;\le1\,.
\end{equation}
In the VKLB scheme, exact streaming from one lattice site to its neighbor in a single time‐step $\Delta t$ requires
\[
  \Delta x \;=\;\xi\,\Delta t
  \quad\Longrightarrow\quad
  \xi \;=\;\frac{\Delta x}{\Delta t}\,.
\]
Substituting this expression into \eqref{eq: CFL-Leveque} results in
\[
  \nu
  =\frac{\Delta t}{\Delta x}\,\max_j|\lambda^j|
  =\frac{\max_j|\lambda^j|}{\xi}
  \;\le1
  \quad\Longrightarrow\quad
  \xi\;\ge\;\max_j|\lambda^j|\,.
\]
Equivalently, for a prescribed CFL number \(0<\nu\le1\),
\begin{equation}\label{eq: lattice_velocity_CFL}
  \xi \;=\;\frac{\max_{j}|\lambda^j|}{\nu}\,.
\end{equation}
\subsection{Boundary Conditions}\label{ss:boundary_conditions}
We consider nonreflecting and periodic boundary conditions in this article.
\subsubsection*{Nonreflecting Boundary Conditions}
Shock problems are typically posed on non-periodic domains, where special care is needed to avoid spurious reflections from the boundaries back into the computational domain. To address this, we impose characteristic (nonreflecting) boundary conditions based on the hyperbolic structure \eqref{eq: eigenvalue_decomposition} of \eqref{eq: hyperbolic-system}. This construction controls incoming waves while allowing outgoing waves to pass freely through the domain boundaries. Such boundary treatments have been developed previously for polynomial distribution functions \cite{Kruger_2017_LB_Graduate_Book,Heubes_2014_Characteristic_BC,Klass_2024_Characteristic_BC} and we extend this construction to the vector kinetic setting. For simplicity, we construct the characteristic boundary conditions in one-dimension. We consider the case \(D = 1\) and \(Q = 3\), with \(x_b\) denoting a boundary point; see \autoref{fig: d1q3_lattice_nonreflecting}. Applying the chain rule to \eqref{eq: hyperbolic-system} yields the equivalent quasi-linear form
\begin{equation}\label{eq: quasilinear_form}
  \partial_t U + (\partial_U F)\,\partial_x U = 0.
\end{equation}
Substituting in the eigenvalue decomposition \eqref{eq: eigenvalue_decomposition} into \eqref{eq: quasilinear_form} motivates the introduction of \textit{characteristic variables} $W$
\begin{equation}\label{eq: characteristic_variables}
  W \;=\; P^{-1}U,
\end{equation}
so that the linearized system about the boundary point $x_b$ decouples \eqref{eq: quasilinear_form} into
\begin{equation}\label{eq: char-adv}
  \partial_t W + \Lambda\, \partial_x W = 0,
  \quad
  W = (W^1,\dots,W^p)^T,
  \quad
  \Lambda \partial_x W = (\lambda^{j}\,\partial_x W^{j})_{j=1}^p.
\end{equation}
\begin{figure}
  \centering

  \begin{subfigure}[t]{0.48\textwidth}
    \centering
    \includegraphics[width=\linewidth]{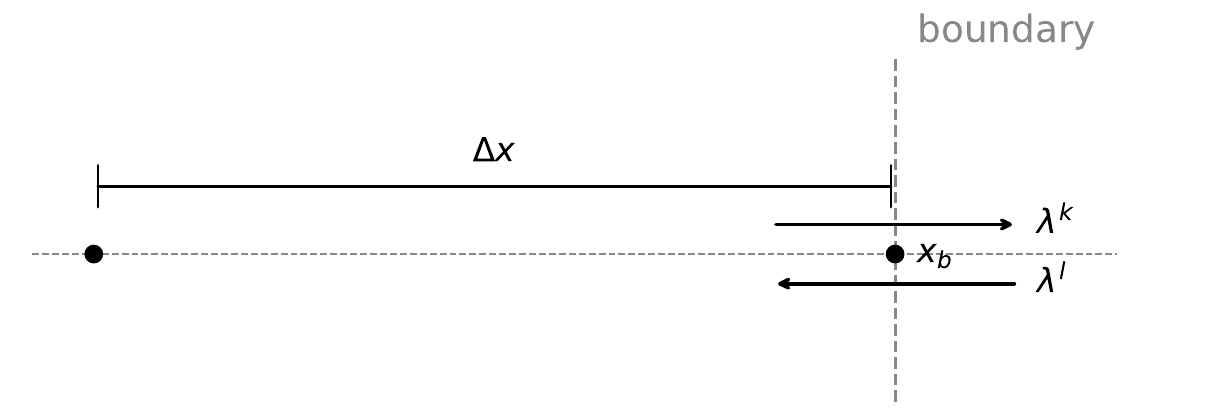}
    \caption{Nonreflecting boundary conditions.}
    \label{fig: d1q3_lattice_nonreflecting}
  \end{subfigure}
  \hfill
  \begin{subfigure}[t]{0.48\textwidth}
    \centering
    \includegraphics[width=\linewidth]{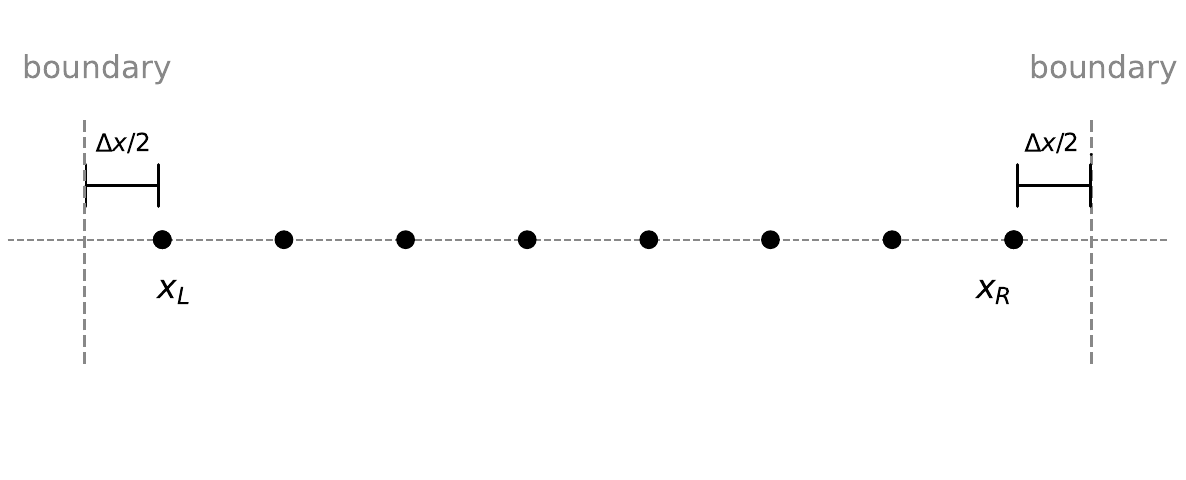}
    \caption{Periodic boundary conditions.}
    \label{fig: d1q3_lattice_periodic}
  \end{subfigure}

  \caption{One-dimensional lattice, \(D=1\). (a)
  For non-reflecting boundary conditions, the lattice point \(x_b\) is on the computational/physical boundary. The eigenvalue \(\lambda^k > 0\) while \(\lambda^l < 0\), corresponding to outward and inward flow of the characteristic variables. (b) For periodic boundary conditions, the lattice nodes $x_L$ and $x_R$ are within the fluid.}
  \label{fig: d1q3_lattice_structure_boundary}
\end{figure}
The eigenvalues of the decoupled system \eqref{eq: char-adv} determine the propagation direction of the characteristic variables \eqref{eq: characteristic_variables} at a boundary point $x_b$. For example, take the right boundary point $x_b$ in \autoref{fig: d1q3_lattice_nonreflecting} where positive characteristic values \(\lambda^k>0\) carry waves out of the domain while negative characteristic values \(\lambda^l<0\) introduce/reflect waves back into the domain. In order to have a non‐reflecting condition, we set incoming waves to zero by setting the corresponding eigenvalue to zero. At the lattice boundary point $x_b$, we adjust the decoupled system \eqref{eq: char-adv} to be, 
\begin{equation}\label{eq:mod-chi-deriv}
  \left(\lambda^j \partial_x W^j\right)'
  = \begin{cases}
    \lambda^j\,\partial_x W^j, & \lambda^j>0\ (\text{outgoing}),\\[3pt]
    0, & \lambda^j<0\ (\text{incoming}),
  \end{cases}
  \qquad
  j = 1, 2, \ldots p.
\end{equation}
At the lattice boundary point $x_b$, replace \(\Lambda \partial_x W \) in the decoupled system \eqref{eq: char-adv} with the modification \eqref{eq:mod-chi-deriv}. The linearized system about $x_b$ reduces to an ODE that can be solved for the vector of conserved quantities,
\begin{equation}\label{eq:Ub-t}
  \partial_t U(x_b,t)
  = -\,P\;\left(\Lambda \partial_x W\right)'(x_b,t).
\end{equation}
Using a first order approximation with a time step \(\Delta t\),
\begin{equation}\label{eq:Ub-update}
  U(x_b,t+\Delta t)
  =
  U(x_b,t)
  \;-\;
  \Delta t\;P\;(\Lambda \partial_x W)'(x_b,t).
\end{equation}
The distribution functions $\{f_q\}_{q=1}^3$ are then determined at the boundary node $x_b$ by setting them equal to the equilibrium distribution functions evaluated with \eqref{eq:Ub-update},
\begin{equation}\label{eq:bc-fq}
  f_q(x_b,t+\Delta t)
  \;=\;
  f_q^{eq}\!\bigl(U(x_b,t+\Delta t)\bigr),
  \quad
  q=1,\dots,Q.
\end{equation}
Equations \eqref{eq: char-adv}–\eqref{eq:bc-fq} provide a linearized procedure for imposing non‐reflecting conditions in the VKLB framework. 
\subsubsection*{Periodic Boundary Conditions}
For periodic boundary conditions, it is essential to ensure that the vector equilibrium distribution functions maintain their periodicity. For simplicity, consider a one-dimensional domain $\Omega = [0,L]$ and three discrete velocities, as illustrated in \autoref{fig: d1q3_lattice_structure}. Let $x_R$ and $x_L$ be nodes $\Delta x/2$ away from the computational boundary (see \autoref{fig: d1q3_lattice_periodic}). Under periodicity, distributions streaming out of one side of the domain are reintroduced through the opposite side. Thus, after the collision step, the post-collision distributions that leave the domain at one boundary are used as the incoming post-streaming distributions at the opposite boundary (see \autoref{fig: flow_chart}); see \cite{Kruger_2017_LB_Graduate_Book}. For the setup in \autoref{fig: d1q3_lattice_structure} and \autoref{fig: d1q3_lattice_periodic}, the periodicity of the distribution functions takes the form,
\begin{equation}\label{eq: periodic_boundary_conditions}
    f_1(\mathbf{x}_L, t + \Delta t) = f_1^*(\mathbf{x}_R - \mathbf{v}_1 \Delta t, t), \quad f_3(\mathbf{x}_R, t + \Delta t) = f_3^*(\mathbf{x}_L - \mathbf{v}_3 \Delta t, t)
\end{equation}

\subsection{Extension to Higher Dimensions}\label{ss: higher_dimensions}
The weighted upwind distribution function set \eqref{eq: adjusted-f1}--\eqref{eq: adjusted-f3} extends in a straightforward manner to higher dimensions. Here we briefly provide the changes in the formulations for the two dimensions. The lattice is composed of points $\{(x_i,y_i)\}_{i=1}^N$ and has uniform spacing in both the $x$ and $y$ directions,  $\Delta x = \Delta y$. Let $D=2$, $Q =5$ and let $\xi$ be the streaming speed along coordinate directions. The discrete velocity vectors $\mathbf{v}_q$ are given by,
\begin{equation}\label{eq: general_velocity_vectors}
\mathbf{v}_1 =
\begin{pmatrix}
\xi \\
0
\end{pmatrix}, \,
\mathbf{v}_2 =
\begin{pmatrix}
-\xi \\
\phantom{-} 0
\end{pmatrix}, \,
\mathbf{v}_3 =
\begin{pmatrix}
0\\
0
\end{pmatrix}, \,
\mathbf{v}_4 =
\begin{pmatrix}
0\\
\xi 
\end{pmatrix}, \,
\mathbf{v}_5 =
\begin{pmatrix}
\phantom{-}0 \\
-\xi 
\end{pmatrix}.
\end{equation}
\begin{center}
  \includegraphics[width=0.35\columnwidth]{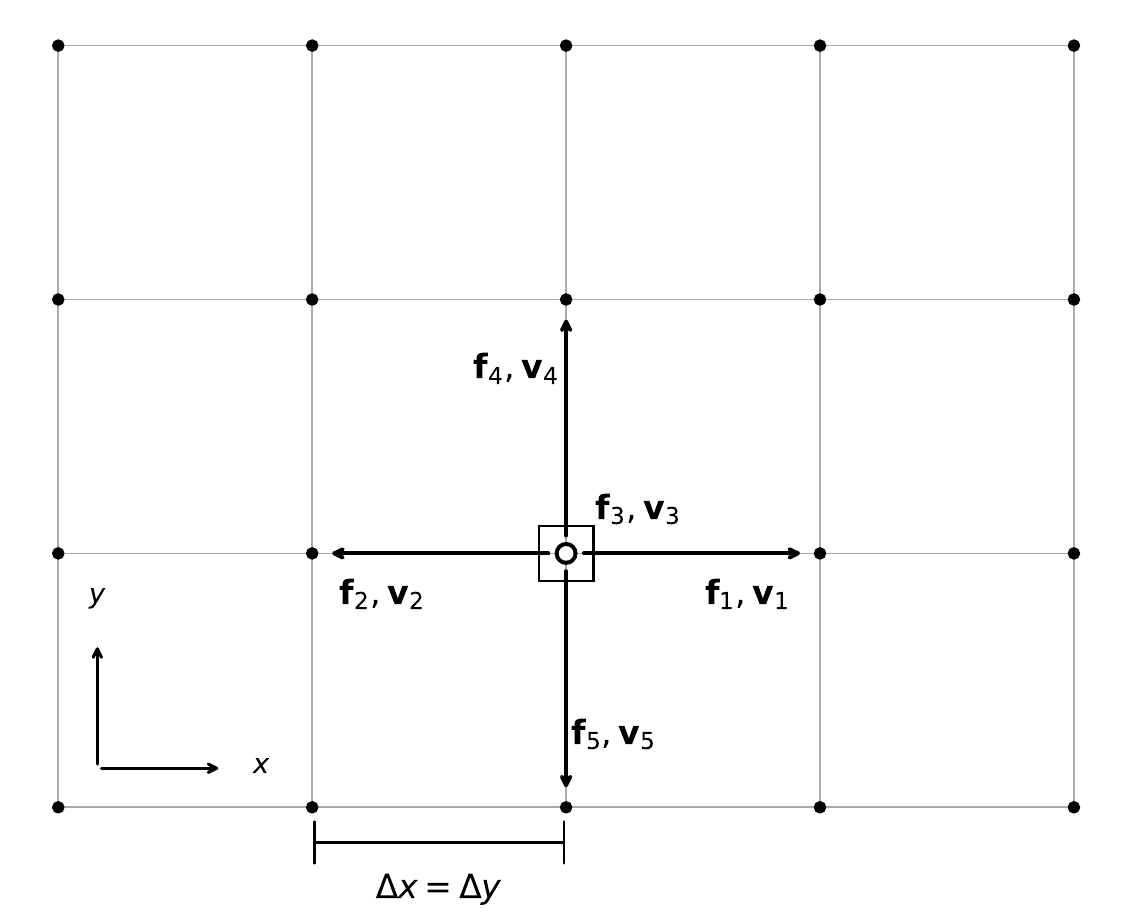}
  \captionsetup{hypcap=false}
  \captionof{figure}{Orthogonal vectors \eqref{eq: general_velocity_vectors} relative to the coordinate axis.}
  \label{fig: d2q5_lattice_structure}
\end{center}
Each velocity vector in \eqref{eq: general_velocity_vectors} has a corresponding equilibrium distribution function $f^{eq}_1, f^{eq}_2, f^{eq}_3, f^{eq}_4, f^{eq}_5$. The \textbf{\textit{2D centered flux distribution functions}} with respect to the velocity vectors \eqref{eq: general_velocity_vectors} are given by,
\begin{equation}\label{eq: 2d_centered_flux}
f_1^{eq} = \frac{1}{5}U + \frac{1}{2\xi}F^{(1)}(U),  \ 
f_2^{eq} = \frac{1}{5}U - \frac{1}{2\xi}F^{(1)}(U), \ 
f_3^{eq} = \frac{1}{5}U, \ 
f_4^{eq} = \frac{1}{5}U + \frac{1}{2\xi}F^{(2)}(U), \ 
f_5^{eq} = \frac{1}{5}U - \frac{1}{2\xi}F^{(2)}(U).
\end{equation}

We next formulate a two-dimensional discontinuous upwind distribution function set with respect to a discrete velocity vectors \eqref{eq: general_velocity_vectors}. Since the system \eqref{eq: hyperbolic-system} is assumed to be hyperbolic, the eigenvalue decomposition \eqref{eq: eigenvalue_decomposition} is well defined along any direction $d \in \{1, 2, \ldots, D\}$. Let $F^{(d,\pm)}(U) \in \mathbb{R}^p$ be the partial flux components \eqref{eq: discontinuous_partial_fluxes} of the physical flux $F^{d}(U)$ along the coordinate direction $d$. The \textbf{\textit{2D discontinuous upwind distribution function set}} with respect to the velocity vectors \eqref{eq: general_velocity_vectors} takes the form,
\begin{equation}\label{eq: 2d_discontinuous_upwind}
\begin{aligned}
&f_1^{eq} = \frac{1}{\xi}F^{(1,+)}(U), \ f_2^{eq} = -\frac{1}{\xi}F^{(1,-)}(U), \ f_4^{eq} = \frac{1}{\xi}F^{(2,+)}(U),  \ f_5^{eq} = -\frac{1}{\xi}F^{(2,-)}(U),\\
&f_3^{eq} = U - \frac{1}{\xi}\left[  F^{(1,+)}(U) + F^{(2,+)}(U) - F^{(1,-)}(U) - F^{(2,-)(U)}
\right] 
.
\end{aligned}
\end{equation}
% 
% \begin{align*}
% f_1^{eq} &= \phantom{-}\frac{1}{\xi}\left( \cos(\theta)F^{(1,+)} + \sin(\theta)F^{(2,+)}\right) \\
% f_2^{eq} &= -\frac{1}{\xi}\left( \cos(\theta)F^{(1,-)} + \sin(\theta)F^{(2,-)}\right) \\
% f_3^{eq} &= U - \frac{1}{\xi}\left[ (\cos(\theta) + \sin(\theta))\left( F^{(1,+)} + F^{(2,+)} - F^{(1,-)} - F^{(2,-)} \right)
% \right] \\
% f_4^{eq} &= \phantom{-}\frac{1}{\xi}\left( \cos(\theta)F^{(2,+)} - \sin(\theta)F^{(1,-)}\right) \\
% f_5^{eq} &= -\frac{1}{\xi}\left( \cos(\theta)F^{(2,-)} - \sin(\theta)F^{(1,+)}\right)
% \end{align*}
The construction the weighted upwind distribution function set is similar to that of discontinuous upwind distribution function set. Let $\mathcal{F}^{(d,\pm)}(U) \in \mathbb{R}^p$ be the partial flux components \eqref{eq: weighted-flux-splitting} of the physical flux $F^{(d)}(U)$ along the coordinate direction $d$. The \textbf{\textit{2D weighted upwind distribution function set}} takes the form,
\begin{equation}\label{eq: 2d_weighted_upwind}
\begin{aligned}
&f_1^{eq} = cU + \frac{1}{\xi}\mathcal{F}^{(1,+)}(U), \
f_2^{eq} = cU -\frac{1}{\xi}\mathcal{F}^{(1,-)}(U), \
f_4^{eq} = cU + \frac{1}{\xi}\mathcal{F}^{(2,+)}(U),  \
f_5^{eq} = cU  - \frac{1}{\xi}\mathcal{F}^{(2,-)}(U), \\
f_3^{eq} &= (1 - 4c)U - \frac{1}{\xi}\left[ \mathcal{F}^{(1,+)}(U) + \mathcal{F}^{(2,+)}(U) - \mathcal{F}^{(1,-)}(U) - \mathcal{F}^{(2,-)}(U) \right].
\end{aligned}
\end{equation}
% \begin{align*}
% f_1^{eq} &= cU + \frac{1}{\xi}\left( \cos(\theta)\mathcal{F}^{(1,+)} + \sin(\theta)\mathcal{F}^{(2,+)}\right) \\
% f_2^{eq} &= cU -\frac{1}{\xi}\left( \cos(\theta)\mathcal{F}^{(1,-)} + \sin(\theta)\mathcal{F}^{(2,-)}\right) \\
% f_3^{eq} &= (1 - 4c)U - \frac{1}{\xi}\left[ (\cos(\theta) + \sin(\theta))\left(\mathcal{F}^{(1,+)} + \mathcal{F}^{(2,+)} - \mathcal{F}^{(1,-)} - \mathcal{F}^{(2,-)} \right)
% \right] \\
% f_4^{eq} &= cU + \frac{1}{\xi}\left( \cos(\theta)\mathcal{F}^{(2,+)} - \sin(\theta)\mathcal{F}^{(1,-)}\right) \\
% f_5^{eq} &= cU  - \frac{1}{\xi}\left( \cos(\theta)\mathcal{F}^{(2,-)} - \sin(\theta)\mathcal{F}^{(1,+)}\right)
% \end{align*}
% %

The equilibrium distribution function sets constructed above correspond to the velocity vector set \eqref{eq: general_velocity_vectors}  aligned along the coordinate axis. In a similar fashion, equilibrium distribution function sets rotated at an angle $\theta$ relative to the x-axis can be constructed which are important for anisotropic discretizations \cite{Anandan_2024_VKLB_Upwinding_Source_Term}.

\subsection{Vector Kinetic Lattice Boltzmann Algorithm}\label{s: algorithm_VKLB}
The lattice Boltzmann equation \eqref{lattice-boltzmann-algorithm} leads to a simple collision streaming algorithm where the nonlinear collision operator is local to each lattice site \(x_i\), while the streaming operator exactly advects each distribution function \(f_q(\mathbf{x},t)\) along its corresponding discrete velocity \(\vb{v}_q\). First, the collision operator is,
\begin{equation}\label{eq: collision_operation}
f_q^{*}(\vb{x}, t) = f_q(\mathbf{x},t) - \omega\Bigl(f_q(\mathbf{x},t)-f_q^{eq}(U)\Bigr).
\end{equation}
The collision operator is then followed by the streaming operation,
\begin{equation}\label{eq: streaming_operation}
    f_q(\mathbf{x}+\mathbf{v}_q\Delta t,\; t+\Delta t) = f_q^{*}(\vb{x}, t)
\end{equation}
The conserved variables are then updated through the first discrete moment \eqref{eq: discrete-first-moment}. Boundary conditions are subsequently imposed (\autoref{ss:boundary_conditions}), and the streaming speed \(\xi\) is then updated according to either the subcharacteristic (\autoref{ss: properties}) or CFL (\autoref{ss: CFL}) constraint. The equilibrium distribution function set (\autoref{s:2} or \autoref{s:3}) is then evaluated for the next time step. This sequence defines one full time step of the vector kinetic lattice Boltzmann method and is repeated until a final time is reached. The flow chart \autoref{fig: flow_chart} summarizes the VKLB algorithm, which holds both one and two dimensions.
\begin{figure}
\begin{tikzpicture}[
  font=\scriptsize,
  box/.style={draw, rounded corners, align=left, text width=0.84\textwidth, inner sep=4pt},
  redbox/.style={box, draw=red, text=red},
  arr/.style={-Latex, thick}
]

% \node[box] (init)
% {\textbf{Initialization:} Initialize $\xi$ through CFL/subcharacteristic constraint using $U(\mathbf{x},0)$. Initialize \(f_q(\vb{x}, 0) = f^{eq}_q(U(\mathbf{x},0))\).};

\node[box] (init) {%
\textbf{Initialization:} Initialize $\xi$ through CFL/subcharacteristic
constraint using $U(\mathbf{x},0)$. (See \autoref{ss: CFL}, \autoref{ss: properties})\\
Initialize $f_q(\mathbf{x}, 0) = f_q^{eq}(U(\mathbf{x},0))$. (See \autoref{ss: equilibrium_distribution_functions}, \autoref{ss: weighted_upwinding_set})
};

\node[box, below=0.28cm of init] (time)
{\textbf{Time Step Update:} \(\Delta t = \Delta x/\xi \) .};

\node[box, below=0.28cm of time] (coll)
{\textbf{Collision:} \(f_q^{*}(\mathbf{x}_i, t^n) = f_q(\mathbf{x}_i, t^n) -\omega\left(f_q(\mathbf{x}_i,t^n) -f_q^{eq}(U_i^n)\right)\). (See Eq. \eqref{eq: collision_operation})};

\node[box, below=0.28cm of coll] (stream)
{\textbf{Streaming:} \(f_q^{n+1} \equiv f_q(\mathbf{x}_i+\mathbf{v}_q\Delta t,t^{n} + \Delta t)=f_q^*(\mathbf{x}_i,t^n)\) (See Eq. \eqref{eq: streaming_operation}). \\
Periodic boundary conditions implemented here. (See Eq.\eqref{eq: periodic_boundary_conditions})};

\node[box, below=0.28cm of stream] (mom)
{\textbf{Moment recovery:} \(U_i^{n+1}=\sum_q f_q^{n+1}\). (See Eq. \eqref{eq: discrete-first-moment})};

\node[box, below=0.28cm of mom] (bc)
{\textbf{Characteristic boundary conditions:} Impose boundary conditions for non-periodic shock problem. (See \autoref{ss:boundary_conditions})};

\node[box, below=0.28cm of bc] (sub)
{\textbf{CFL constraint:} Update $\xi$ through CFL/subcharacteristic constraint using $U_i^{n+1}$. (See  \autoref{ss: CFL}, \autoref{ss: properties})};

\node[box, below=0.28cm of sub] (eq)
{\textbf{Equilibrium update:} Compute \(f_q^{eq}(U_i^{n+1})\). (See \autoref{ss: equilibrium_distribution_functions}, \autoref{ss: weighted_upwinding_set})};

\draw[arr] (init) -- (time);
\draw[arr] (time) -- (coll);
\draw[arr] (coll) -- (stream);
\draw[arr] (stream) -- (mom);
\draw[arr] (mom) -- (bc);
\draw[arr] (bc) -- (sub);
\draw[arr] (sub) -- (eq);

% loop arrow on the right side
\draw[arr] 
  ([xshift=0.3cm]eq.east) -- 
  ([xshift=1.2cm]eq.east) --
  ([xshift=1.2cm]time.east) --
  ([xshift=0.3cm]time.east);
  
\end{tikzpicture}
\caption{Flow chart description of the VKLB algorithm.}
\label{fig: flow_chart}
\end{figure}
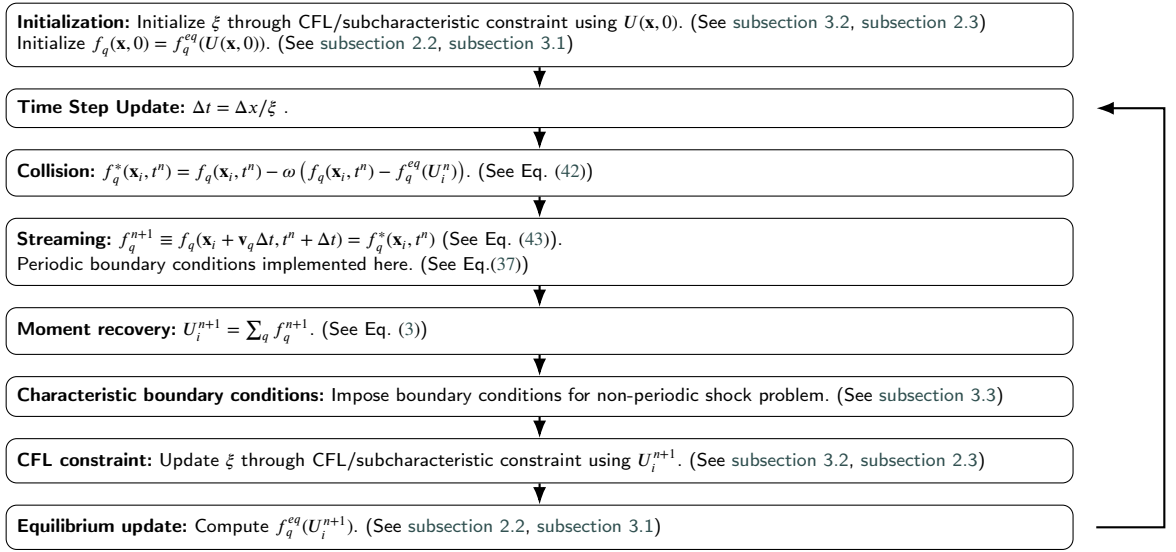
\section{Analysis of the Weighted Upwind Distribution Scheme}\label{s:4}
We analyze the WU-VKLB scheme associated with the equilibrium distribution functions \eqref{eq: adjusted-f1}--\eqref{eq: adjusted-f3}. We begin by examining an equivalent finite difference form of the scheme and showing consistency with the hyperbolic system \eqref{eq: hyperbolic-system}. We then turn to an entropy-based analysis to derive an upper bound for the parameter \(c\). The question of admissible constraints on \(k\) in \eqref{eq: parameterization-function} is left for future work.

\subsection{Equivalent Finite Difference and Finite Volume Formulations}
\label{sec:equivalence_fd_fv}
The lattice Boltzmann equation \eqref{lattice-boltzmann-algorithm} has an equivalent finite difference formulation, and we use this general finite difference formulation for the weighted upwind distribution function set \eqref{eq: adjusted-f1}--\eqref{eq: adjusted-f3} to interpret the role the coefficient $c$ of the conserved vector $U$ as the coefficient to a diffusive stencil. \autoref{thrm: fd_formulation} is the equivalent finite difference formulation for any $0 < \omega < 2$ (see Section 4.2 of \cite{Anandan_2024_VKLB_Upwinding_Source_Term} for the proof). 
\begin{theorem}\label{thrm: fd_formulation} (Macroscopic Finite Difference Form) Consider spacial dimension $D = 1$, let $n ,N \in \mathbb{N}$ and $0 < \omega < 2$. Assume $\xi^{n}$ satisfies a CFL constraint \eqref{eq: lattice_velocity_CFL} for each time step, $\mathbf{v}_{q}^{n}  = \left( m_q \xi^{n} \right)$ for $m_q \in \mathbb{Z}$, and $t^{n+1} - t^n \equiv \Delta t^n = \Delta x/\xi^{n}$. Let $f_{q_i}^{n+1} = f_q(x_i, t^n + \Delta t^n)$, $f_{q_{i - m_{q}}}^n = f_q(x_i - v_{q}^{1,n}\Delta t^n , t^n)$, and $f_{q_{i - m_{q}}}^{eq,n} = f_q\left(U(x_i - v_{q}^{1,n}\Delta t^n , t^n)\right)$. If $t^{n - N}$ is the initial time, then the general vector lattice Boltzmann scheme
\begin{equation}\label{general_LBM}
     f_{q_i}^{n+1} = (1 - \omega)f_{q_{i - {m_q}}}^n + \omega f_{q_{i - {m_q}}}^{eq,n}
\end{equation}
is equivalent to the following finite difference form
\begin{equation}\label{eq: general_equivalent_fd_formulation}
    f_{q_i}^{n+1} = \omega \left( \sum_{k = 0}^{N-1}(1 - \omega)^kf_{q_{i - (k+1)m_q}}^{eq, n-k}\right) + (1 - \omega)^Nf_{q_{i - (N+1)m_q}}^{eq, n - N}
\end{equation}
where $f_{q_{i - (N+1)m_q}}^{eq, n- N} = f_{q_{i - (N+1)m_q}}^{n- N}$. 
\end{theorem}
\begin{theorem}
\label{thm:equiv_fd_omega_1}    

Let $D = 1$ and $Q = 3$ as illustrated in \autoref{fig: d1q3_lattice_structure}. Let $\omega = 1.0$ and let $\{f_q^{eq}\}_{q=1}^3$ be the weighted upwind equilibrium distribution function set \eqref{eq: adjusted-f1}--\eqref{eq: adjusted-f3}. The WU-VKLB scheme is equivalent to the following finite difference formulation,
\begin{equation}\label{eq: weighted-upwinding_fd-equivalence}
\begin{aligned}
U_i^{n+1} &= U_i^n + c\!\left(U_{i-1}^{\,n} -2\,U_{i}^{\,n} + U_{i+1}^{\,n}\right)
- \tfrac{1}{\xi}\left(\mathcal{F}_{i}^{(+,n)} - \mathcal{F}_{i-1}^{(+,n)}\right) - \tfrac{1}{\xi}\left(\mathcal{F}_{i+1}^{(-,n)} - \mathcal{F}_{i}^{(-,n)}\right)
\end{aligned}
\end{equation}
\end{theorem}
\begin{proof} Since $\omega = 1.0$, the only nonzero term in the finite difference formulation \eqref{eq: general_equivalent_fd_formulation} is $k = 0$. That is,
\begin{equation}\label{eq: k_equals_zero_fd_formulation}
   f_{q,i}^{n+1} = f_{q,i - m_q}^{eq, n-1} 
\end{equation}
For the three discrete velocities $\mathbf{v}_q^n, \; q = 1,2,3$ in \autoref{fig: d1q3_lattice_structure}, it follows that $m_1 = 1, \, m_2 = 0, \,  m_3 = -1$. Taking the discrete first moment of \eqref{eq: k_equals_zero_fd_formulation}, and using the first discrete moment \eqref{eq: discrete-first-moment}, it follows that
\begin{equation}\label{eq: k_equals_zero_discrete formulation}
   U_i^{n+1} = f_{1, i-1}^{eq, n-1} + f_{2, i}^{eq, n-1} + f_{3, i+1}^{eq, n-1}
\end{equation}
Substituting in the equilibrium distribution function set \eqref{eq: adjusted-f1}--\eqref{eq: adjusted-f3} into \eqref{eq: k_equals_zero_discrete formulation} recovers \eqref{eq: weighted-upwinding_fd-equivalence}.
\end{proof}

The term $U_{i-1}^{\,n} -2\,U_{i}^{\,n} + U_{i+1}^{\,n}$ produces a diffusive stencil applied to the vector of conserved quantities, hence $c$ needs to be positive for a well posed diffusive term. 

In \cite{Anandan_2024_VKLB_Upwinding_Source_Term}, the finite difference formulation \eqref{eq: general_equivalent_fd_formulation} of the distribution function set \eqref{eq: upwinding-f1}--\eqref{eq: upwinding-f3} with $\omega = 1.0$ is used to show consistency of the scheme up to $\mathcal{O}(\Delta x)$ with respect to \eqref{eq: hyperbolic-system} (see Section 4.3 of \cite{Anandan_2024_VKLB_Upwinding_Source_Term} for details and proof). This consistency argument can be applied verbatim to the weighted upwind equilibrium distribution function set \eqref{eq: adjusted-f1}--\eqref{eq: adjusted-f3} when $c = 0$. To ensure consistency when $c > 0$, we only need to show the additional diffusive stencil is of order $\mathcal{O}(\Delta x^2)$. Let $U$ be sufficiently smooth and $c = \mathcal{O}(1)$. A Taylor series expansion of each term in $U(x_{i-1}) - 2 U(x_i) + U(x_{i+1})$ about a lattice point $x_i$ gives,
\[
\begin{aligned}
&c(U_{i-1}-2U_i+U_{i+1})  \\
&= c\Big(U_i-\Delta x\,U_x(x_i)+\frac{\Delta x^2}{2}U_{xx}(x_i)+\mathcal{O}(\Delta x^3)\Big)
-2cU_i + c\Big(U_i+\Delta x\,U_x(x_i)+\frac{\Delta x^2}{2}U_{xx}(x_i)+\mathcal{O}(\Delta x^3)\Big).\\
&=c(U_i-2U_i+U_i)
+c\big(-\Delta x\,U_x(x_i)+\Delta x\,U_x(x_i)\big)
+c\left(\frac{\Delta x^2}{2}U_{xx}(x_i)+\frac{\Delta x^2}{2}U_{xx}(x_i)\right)
+\mathcal{O}(\Delta x^3) \\
&=c\Delta x^2\,U_{xx}(x_i)+\mathcal{O}(\Delta x^3).
\end{aligned}
\]
Hence the corresponding finite difference formulation for the equilibrium distribution function set  is consistent with the hyperbolic system \eqref{eq: hyperbolic-system} up to $\mathcal{O}(\Delta x)$.

Robust numerical schemes for hyperbolic systems ensure that some type of discrete local conservation property holds within the spatial domain. In this context it is demonstrated that the finite difference discretization \eqref{eq: weighted-upwinding_fd-equivalence} can be written as an equivalent finite volume formulation with appropriately defined numerical fluxes.
\begin{theorem}\label{lem: FV}
Let $\omega = 1$ and consider the weighted upwind equilibrium distribution function set \eqref{eq: adjusted-f1}--\eqref{eq: adjusted-f3}. The finite difference formulation \eqref{eq: weighted-upwinding_fd-equivalence} can be written as a finite volume formulation in conservative form. That is,
\[
U_i^{n+1} = U_i^{n} - \tfrac{1}{\xi}\left(\mathscr{F}_{i + 1/2}^n - \mathscr{F}_{i - 1/2}^n\right),
\]
with interface flux between each cell given as
\begin{equation}\label{eq: flux_description}
\mathscr{F}_{i + 1/2}^n
= (\mathcal{F}^{(+, n)}_i + \mathcal{F}^{(-, n)}_{i+1}) - c\xi\left(U^n_{i+1} - U^n_i\right).
\end{equation}
\end{theorem}
\begin{proof} For the finite difference formulation \eqref{eq: k_equals_zero_fd_formulation}, the diffusive stencil can be written as,
\[
c(U_{i-1}^n - 2U_i^n + U_{i+1}^n) = c(U_{i+1}^n - U_i^n) - c(U_i^n - cU_{i-1}^n)
\]
Then, collecting appropriate flux terms in \eqref{eq: weighted-upwinding_fd-equivalence} gives the finite volume formulation.
\end{proof}
For the partial fluxes \eqref{eq: weighted-flux-splitting}, when an eigenvalue of the system \eqref{eq: hyperbolic-system} equals zero, the numerical fluxes \eqref{eq: flux_description} are similar to that of a local lax Friedrich scheme. As discussed in \autoref{ss: weighted_upwinding_set}, the partial fluxes \eqref{eq: weighted-flux-splitting} are constructed such that $\mathcal{F}^{(\pm)}(U) \rightarrow \tfrac{1}{2}F(U)$ when an eigenvalue goes to zero. Substituting this limit in \eqref{eq: flux_description},
\begin{equation}\label{eq: LLF-numerical-flux}
\mathscr{F}_{i + 1/2}^n
=\frac{1}{2}(F^{(n)}_i + F^{(n)}_{i+1}) - c\xi\left(U^n_{i+1} - U^n_i\right),
\end{equation}
where \(\xi\) denotes the maximum wave speed, as seen in \autoref{ss: CFL}. This flux has the structure of a local Lax--Friedrichs flux (LLF) with a diffusion coefficient \(c\). Generally, \(c>0\) controls the amount of additional artificial diffusion introduced into the VKLB framework beyond the term accounted for in \autoref{subcharacteristic_condition}. Letting $\nu$ be the CFL number from \autoref{ss: CFL}, the choice \(c = \nu/2\) substituted into \eqref{eq: LLF-numerical-flux} gives the standard local Lax--Friedrichs diffusion, while \(c<\tfrac{\nu}{2}\) yields a less diffusive LLF-type flux. Additionally, \autoref{lem: FV} establishes
a discrete conservation within the numerical domain from a telescoping cancellation of the internal numerical fluxes. This conservative property is formed for the discontinuous upwind distribution function set \eqref{eq: upwinding-f1}--\eqref{eq: upwinding-f3} and weighted upwind distribution function set \eqref{eq: adjusted-f1}--\eqref{eq: adjusted-f3}.

\subsection{Estimate for bounding c}
Entropy results for VKLB schemes have been proven to ensure stability of the numerical scheme \cite{Bouchut_1999_Kinetic_Entropies, Wissocq_2024_Positive_Preserving_VKLB, Natalini_Multi_Dimensional_Systems_Discrete_Systems}. In particular, the following criteria ensures that the discrete vector kinetic system \eqref{eq: discrete-kinetic-equation} admits an H-theorem if the hyperbolic system \eqref{eq: hyperbolic-system} has a convex entropy function.

\begin{definition}[Bouchut Criteria]\label{MMD}
A set of equilibrium distribution functions \(\{ f_q^{eq}(U) \}_{q=1}^Q\) is said to be \textit{monotone non-decreasing} if for all \(q \in \{1,\dots,Q\}\), \(\partial_U f_q^{eq}(U)\) is diagonalizable with nonnegative eigenvalues.
\end{definition}
Constructing the weighted upwind distribution function set \eqref{eq: adjusted-f1}--\eqref{eq: adjusted-f3} to satisfy Definition \ref{MMD} leads to an upper bound on $c$. 
\begin{theorem}\label{thrm: bounding_c} Let $D = 1$ and $Q = 3$. Consider the hyperbolic system 
\[ \partial_tU + \partial_xF(U) = 0.\]
Let $\{f_q^{eq}\}_{q=1}^3$ be the weighted upwind equilibrium distribution function set \eqref{eq: adjusted-f1}--\eqref{eq: adjusted-f3}. Consider $\alpha(\lambda)$ and $\nu$ as defined in equations \eqref{fig:alpha-sigmoid} and \eqref{eq: lattice_velocity_CFL} respectively. Assume $F(U)$ is homogeneous of degree one and the Jacobian matrix $\partial_{U}F$ is constant. Then the set  $\{f_q^{eq}\}$ is monotone non-decreasing in the limit $k \rightarrow \infty$ only if
\begin{equation}\label{eq: bounded_c_value}
0 \leq c \leq \tfrac{1}{2}\left(1 - \nu \right).
\end{equation}
\end{theorem}
\begin{proof}
The lower bound of \eqref{eq: bounded_c_value} is from the equivalent finite difference formulation \eqref{eq: weighted-upwinding_fd-equivalence} for the weighted upwind vector kinetic equilibrium distribution functions, where $c \geq 0$ for the diffusive stencil to be well-posed. For the upper bound, since $F$ is homogeneous of degree one, \(F = (J_UF) \cdot U = (P \Lambda P^{-1}) \cdot U \). Taking the limit as $k \rightarrow \infty$, \eqref{eq: eigenvalue-split-smooth} implies that the partial fluxes \eqref{eq: weighted-flux-splitting} coincide with the partial fluxes for the  discontinuous splitting \eqref{eq: eigenvalue-split-discrete}, i.e. script $\mathcal{F}^{(\pm)} = F^{(\pm)}$. The discontinuous splitting, along with the Jacobian being homogeneous of degree one, ensures that the flux splitting \eqref{eq: eigenvalue-split-discrete} (which is the same as \eqref{eq: weighted-flux-splitting} in the limit $k \rightarrow \infty)$) is equivalent to a Steger-Warming splitting. That is, if \ $\Lambda^{(\pm)}$ is a diagonal matrix of positive/negative eigenvalues from the Jacobian decomposition \eqref{eq: eigenvalue_decomposition}, then the partial fluxes \eqref{eq: discontinuous_partial_fluxes} are equivalent to $F^{(+)}(U) = (P \Lambda^{(+)} P^{-1}) \cdot U \;, F^{(-)}(U) = (P \Lambda^{(-)} P^{-1})\cdot U$. Taking the derivative with respect to $U$,  
\[
\pdv{F^{(+)}}{U} = P\Lambda^{(+)}P^{-1}, \quad \pdv{F^{(-)}}{U} = P\Lambda^{(-)}P^{-1}.
\]
For the weighted upwind distribution functions, the first and third distribution functions $f_1^{eq}, \, f_3^{eq}$ in \eqref{eq: adjusted-f1}--\eqref{eq: adjusted-f3} satisfy Definition \ref{MMD} directly. Taking the derivative of $f_2^{eq}(U)$,
\begin{align*}
    \tfrac{\partial f_2^{eq}}{\partial U} &= (1 - 2c)\vb{I} - \tfrac{1}{\xi}\left(\tfrac{\partial F^{(+)}}{\partial U} - \tfrac{\partial F^{(-)}}{\partial U} \right) \\
    &= (1 - 2c)\vb{I} - \tfrac{1}{\xi}\left(P\Lambda^{(+)}P^{-1} - P\Lambda^{(-)}P^{-1} \right) \\
    &= P\left( (1 - 2c)\vb{I} - \tfrac{|\Lambda|}{\xi} \right)P^{-1}.
\end{align*}
Written out for each component, it follows that
\[
1 - 2c - \tfrac{\lambda^j}{\xi} \geq 0, \:  \Rightarrow  \:c \leq \tfrac{1}{2}\left(1 - \tfrac{|\lambda^j|}{\xi}\right) \quad \forall j = 1, 2, \ldots, p.
\]
The minimum of the right hand side is equivalent to taking the maximum of the eigenvalues, and using \eqref{eq: lattice_velocity_CFL}, it follows that,
\[ 0 \leq c \leq  \tfrac{1}{2}\left(1 - \nu\right)\]
\end{proof}
It should be noted that the assumption that  $F(U)$ is homogeneous of degree one is not satisfied for either the shallow water system \eqref{eq: shallow_water_system} or ideal mhd system \eqref{eq: ideal_mhd_system}. However the relationship \eqref{eq: bounded_c_value} is used for all of our studies, with the heuristic reasoning that this relationship holds approximately to highest order by a Taylor series type expansion argument.

\section{Numerical Results}\label{s:5}
In this section, we assess the stability and accuracy of the weighted upwind vector distribution function set \eqref{eq: adjusted-f1}--\eqref{eq: adjusted-f3} on a sequence of hyperbolic systems of increasing complexity, from the shallow water equations, the compressible Euler equations, and ideal magnetohydrodynamics (MHD). For one-dimensional problems (\(D=1\)), we compare the weighted upwind set \eqref{eq: adjusted-f1}--\eqref{eq: adjusted-f3} against the centered flux distribution function set \eqref{eq: centered-f1}--\eqref{eq: centered-f3} and the discontinuous upwind distribution function set \eqref{eq: upwinding-f1}--\eqref{eq: upwinding-f3}. For two-dimensional problems, we compare the weighted upwind distribution function set \eqref{eq: 2d_weighted_upwind} against the centered flux distribution function set \eqref{eq: 2d_centered_flux}. Stability and accuracy is evaluated using both profile comparisons, to assess qualitative solution behavior, and verification studies, to measure quantitative order-of-accuracy or convergence rates for the VKLB methods.

All one-dimensional problems use characteristic boundary conditions, whereas all two-dimensional problems use periodic boundary conditions (\autoref{ss:boundary_conditions}). Unless otherwise stated, one-dimensional centered flux distribution function set \eqref{eq: centered-f1} -- \eqref{eq: centered-f3} uses the subcharacteristic constraint (see \autoref{ss: properties}) \(\eta = 0.90\) while the discontinuous upwind set \eqref{eq: upwinding-f1}--\eqref{eq: upwinding-f3} and weighted upwind set \eqref{eq: adjusted-f1}--\eqref{eq: adjusted-f3} use the CFL constraint (see \autoref{ss: CFL}) \(\nu = 0.90\). Hereafter, CFL constraint refers to the appropriate condition for the distribution function set under consideration. 

The weighted upwind distribution function set \eqref{eq: adjusted-f1}--\eqref{eq: adjusted-f3} also requires the diffusive parameter \(c\). Taking \(\nu = 0.90\) in \eqref{thrm: bounding_c} gives the bound \(0 \leq c \leq 0.05\), and we therefore use \(c = 0.05\) for both the one and two dimensional weighted upwind sets. This choice adds numerical diffusion to improve stability near steep gradients and discontinuities. 

For shock-dominated problems, we set the relaxation parameter \(\omega = 1.0\) for all distribution function sets, This maintains a first order behavior while providing both numerical diffusion for stability needed near discontinuities and desired theoretical properties \eqref{subcharacteristic_condition}. For smooth problems, we set \(\omega = 2.0\), which theoretically yields second order accuracy \eqref{subcharacteristic_condition}. In all cases, we take \(k = 2.0\) in \eqref{fig:alpha-sigmoid} to ensure a smooth transition when an eigenvalue changes sign. 
\subsection{Shallow Water}
We begin with the one-dimensional shallow water equations, providing a simple nonlinear vector conservation law system $(p=2)$ where Riemann problems develop shocks and rarefactions. In conservative form, the components of the system are given by
\begin{equation}\label{eq: shallow_water_system}
U=
\begin{pmatrix}
h\\
hu
\end{pmatrix},
\qquad
F(U)=
\begin{pmatrix}
hu\\[0.2em]
\displaystyle \frac{(hu)^2}{h} + \tfrac12 g h^2
\end{pmatrix}.
\end{equation}
where $h$ is the water height, $hu$ is the discharge, and $g$ is the gravitational constant. The shallow water system is equivalent to the isentropic compressible Euler equations with a particular pressure law \cite{Leveqe_Finite_Volume_Book}. The characteristic wave speeds are $u \pm \sqrt{gh}$, and the system is strictly hyperbolic whenever \(h>0\) for all time.

We first test the characteristic boundary conditions (\autoref{ss:boundary_conditions}) with the centered flux distribution function set \eqref{eq: centered-f1}--\eqref{eq: centered-f3}. We use the following smooth initial condition \eqref{fig: shallow_water_cbc} for the shallow water system \eqref{eq: shallow_water_system}. Let $\Omega = [-5, 5]$, a final time $T = 8$, an initial velocity $u = 0$, and a smooth initial height
\begin{equation}\label{eq: shallow_water_smooth_ic}
h(x)=
\begin{cases}
1 + 0.3 \left( \dfrac{1-\cos\!\bigl(\pi(x+1)\bigr)}{2} \right)^2, & x \in \Omega\cap[-1,1], \\[8pt]
1, & x \in \Omega\setminus[-1,1].
\end{cases}
\end{equation}
\begin{center}
  \includegraphics[width=0.5\columnwidth]{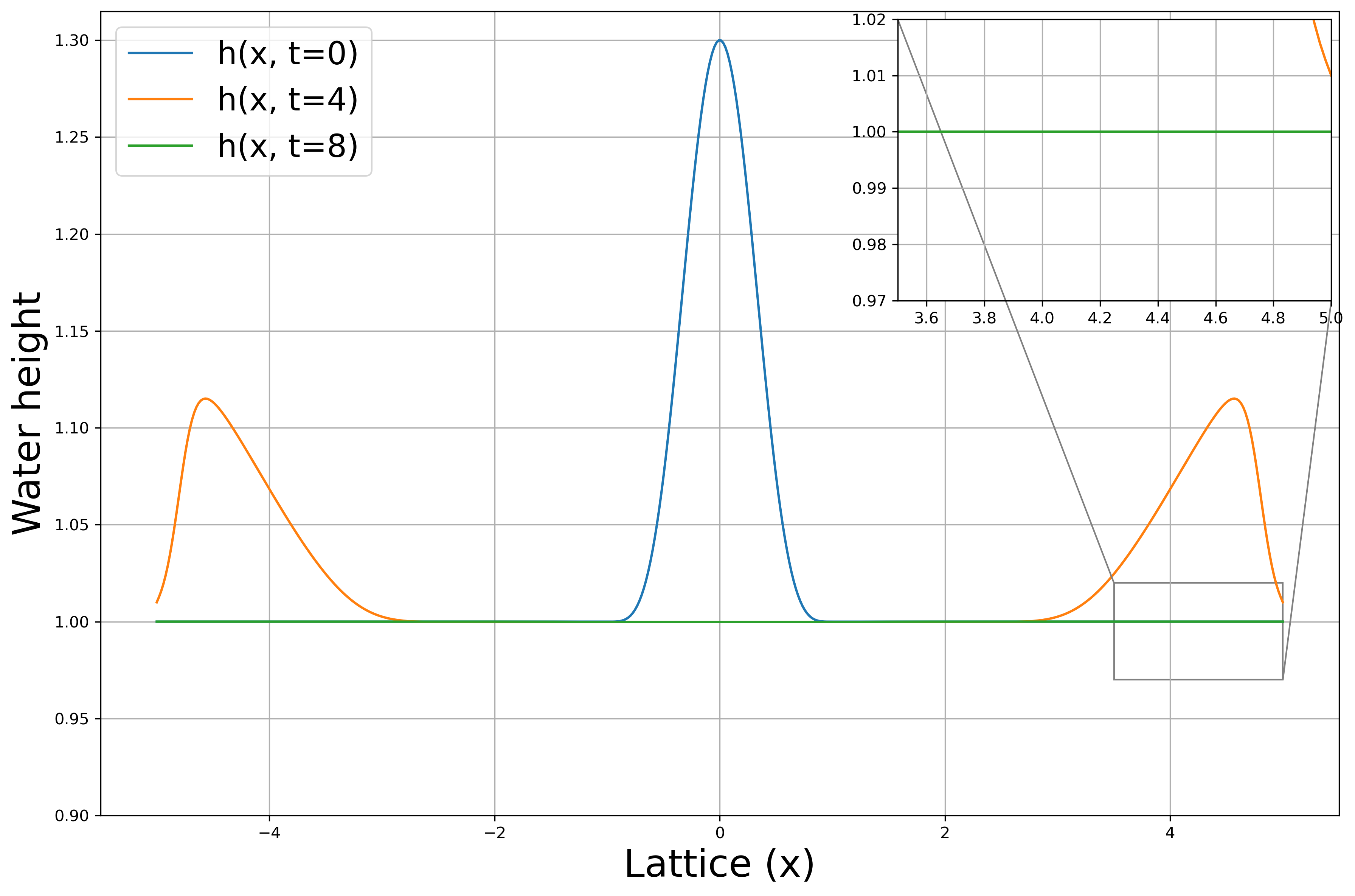}
  \captionsetup{hypcap=false}
  \captionof{figure}{Outflow boundary conditions for the shallow water system with a smooth initial condition \eqref{eq: shallow_water_smooth_ic}.}
  \label{fig: shallow_water_cbc}
\end{center}
We use $N = 512$ lattice points and a relaxation parameter $\omega = 1.0$. \autoref{fig: shallow_water_cbc} shows the numerical solution transmits through the computational boundary without any visibly reflecting waves back into the computational domain. This is important in the current setting, as Riemann problems require a non-periodic domain, and even small boundary reflections can pollute convergence measurements.

Next, we examine three dam-break Riemann problems whose initial conditions are given in \autoref{tab: shallow_water_riemann_problems}. Each  problem has distinct eigenvalue behavior. In the \textit{standard dam-break problem}, both eigenvalues retain their sign over all time. This is a baseline shock problem in which all three distribution function sets can be compared under relatively favorable conditions. In the \textit{hydraulic flow} problem, the Froude number, $Fr = u/\sqrt{gh}$, passes through unity, hence one eigenvalue changes sign within the computational domain. This transition is significant because it challenges the flux splitting \eqref{eq: eigenvalue-split-discrete} that relies on a eigenvalue decomposition \eqref{eq: eigenvalue_decomposition}. In the \textit{compression flow} problem, one eigenvalue is initialized zero, placing the solution at the transition point between positive and negative components in \eqref{eq: eigenvalue-split-discrete}.

For the standard dam-break problem, all methods exhibit the expected convergence rate, which is slightly below first order, as seen in \autoref{fig: standard-flow-shallow-water}. Furthermore, he weighted upwind set \eqref{eq: adjusted-f1}--\eqref{eq: adjusted-f3} is the most accurate, capturing the right edge of the rarefaction more precisely than both the centered flux set \eqref{eq: centered-f1}--\eqref{eq: centered-f3} and the discontinuous upwind set \eqref{eq: upwinding-f1}--\eqref{eq: upwinding-f3}.

For the hydraulic flow problem, \autoref{fig: hydraulic-flow-shallow-water} demonstrates the convergence rates are again slightly below first order. \autoref{fig: height-sw-hydraulic-flow-quantity} shows the discontinuous upwind set \eqref{eq: upwinding-f1}--\eqref{eq: upwinding-f3} develops instabilities along the rarefaction as the Froude number crosses unity, and these instabilities persist for all tested CFL numbers, \(\nu \in \{0.1, 0.2, \ldots, 0.9\}\). In contrast, the weighted upwind set \eqref{eq: adjusted-f1}--\eqref{eq: adjusted-f3} suppresses these instabilities and more accurately resolves both the rarefaction and the shock than the centered flux set \eqref{eq: centered-f1}--\eqref{eq: centered-f3}. Overall, the weighted upwind set provides the best accuracy for the hydraulic flow case (\autoref{fig: hydraulic-flow-shallow-water}).

The compression flow problem involves two waves colliding. The initial conditions in \autoref{tab: shallow_water_riemann_problems} ensure one of the eigenvalues is initialized to zero. \autoref{fig: compression-flow-shallow-water} shows the discontinuous upwind set \eqref{eq: upwinding-f1}--\eqref{eq: upwinding-f3} develops instabilities along the right shock for all tested CFL numbers, \(\nu \in \{0.1, 0.2, \ldots, 0.9\}\). The weighted upwind set \eqref{eq: adjusted-f1}--\eqref{eq: adjusted-f3} again gives higher accuracy than the centered flux set \eqref{eq: centered-f1}--\eqref{eq: centered-f3} while smoothing the right-shock instabilities developed with the discontinuous upwind set.
\begin{table}
    \renewcommand{\arraystretch}{1.15}
    \setlength{\tabcolsep}{6pt}
    \begin{tabular}{>{\raggedright}p{4.0cm}cccc}
        \toprule
        \textbf{Parameter} 
        & \textbf{Standard Dam Break}
        & \textbf{Hydraulic Flow} 
        & \textbf{Compression Flow}  \\
        \midrule
        Spatial domain 
        & $[-5.0, 5.0]$ 
        & $[-20.0,\,20.0]$ 
        & $[-2.0,\,2.0]$ \\

        Time interval [s] 
        & $[0.0,\,2.0]$
        & $[0.0,\,4.0]$ 
        & $[0.0,\,0.5]$ \\

        Initial discontinuity $x_0$ 
        & $0.0$
        & $0.0$ 
        & $0.0$ \\

        Height $(h_l,\,h_r)$ 
        & $(3.0, \; 1.0)$
        & $(10.0,\;1.0)$ 
        & $(1.0,\;1.0)$ \\

        Discharge $((hu)_l,\,(hu)_r)$ 
        & $(0.0, \; 0.0)$
        & $(0.0,\;0.0)$ 
        & $(1.0,\;-1.0)$ \\

        Gravitational Constant $g$ 
        & $1.0$
        & $1.0$ 
        & $1.0$ \\
        \bottomrule
    \end{tabular}
    \caption{Initial conditions for shallow water Riemann problems.}
    \label{tab: shallow_water_riemann_problems}
\end{table}

\begin{figure}
  \centering
  \begin{subfigure}[b]{0.30\textwidth}
    \includegraphics[width=\linewidth]{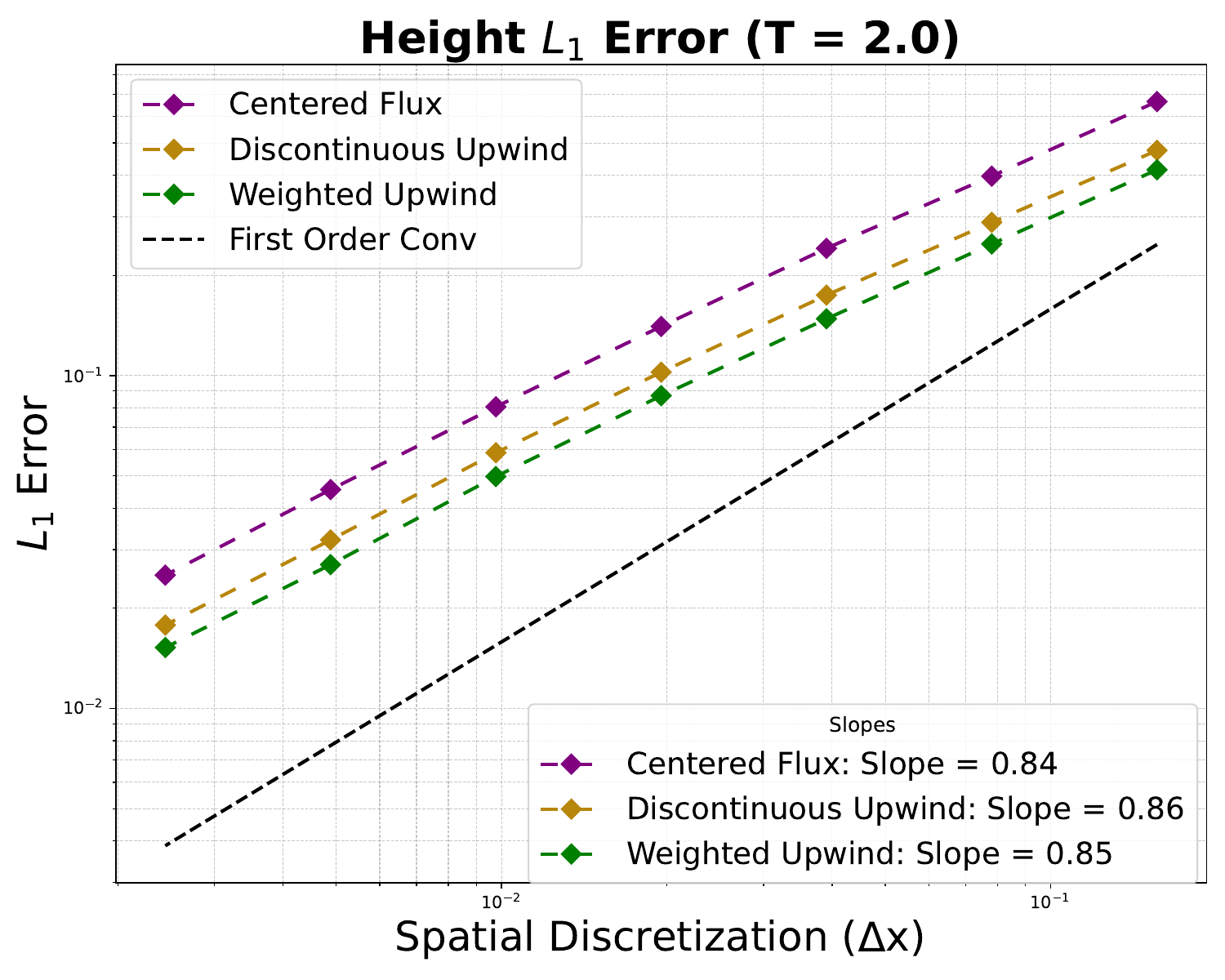}
    \caption{Height}
    \label{fig: height-sw-standard-damn-break-conv}
  \end{subfigure}
  \hfill
  \begin{subfigure}[b]{0.30\textwidth}
    \includegraphics[width=\linewidth]{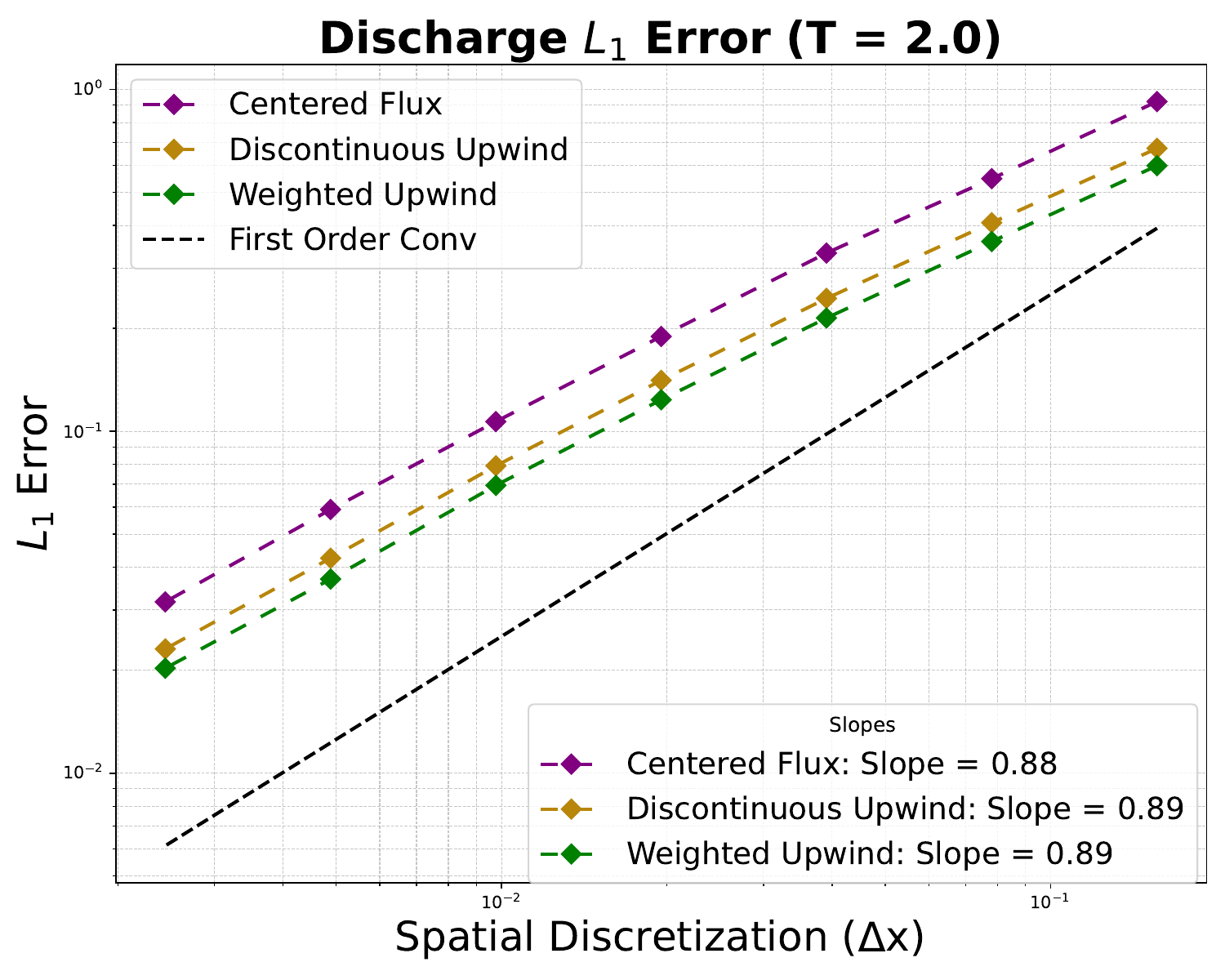}
    \caption{Discharge}
    \label{fig: discharge-sw-standard-damn-break-conv}
  \end{subfigure}
  \hfill
  \begin{subfigure}[b]{0.30\textwidth}
    \includegraphics[width=\linewidth]{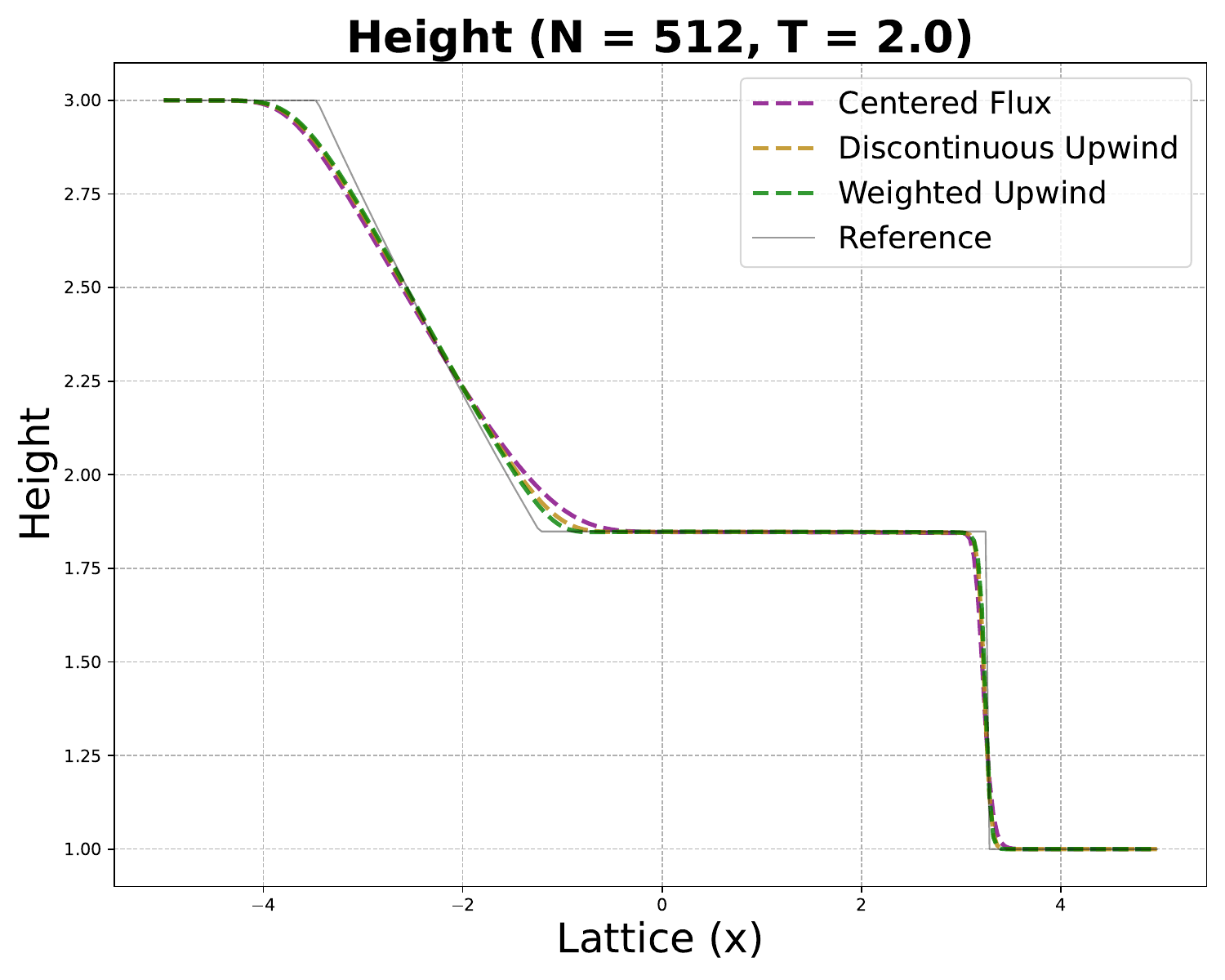}
    \caption{Density}
    \label{fig: height-sw-standard-damn-break-quantity}
  \end{subfigure}
  \caption{Standard dam-break problem (\autoref{tab: shallow_water_riemann_problems}) using discrete lattice points \(N=[64,128,256,\ldots,4096]\), \(\text{CFL}=0.90\), and \(\omega=1.0\).}
  \label{fig: standard-flow-shallow-water}
\end{figure}
\begin{figure}
  \centering
  \begin{subfigure}[b]{0.30\textwidth}
    \includegraphics[width=\linewidth]{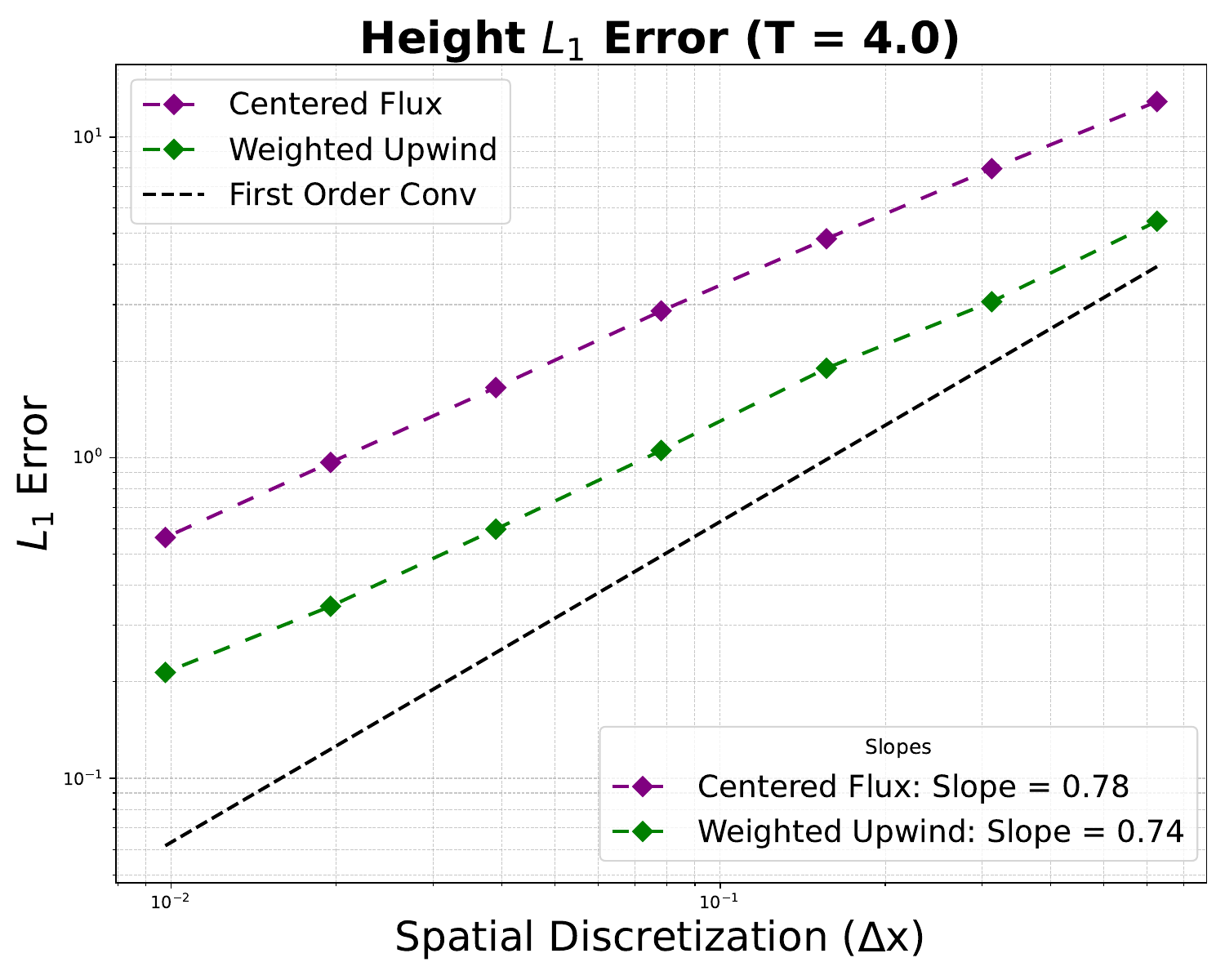}
    \caption{Height}
    \label{fig: height-sw-hydraulic-flow-conv}
  \end{subfigure}
  \hfill
  \begin{subfigure}[b]{0.30\textwidth}
    \includegraphics[width=\linewidth]{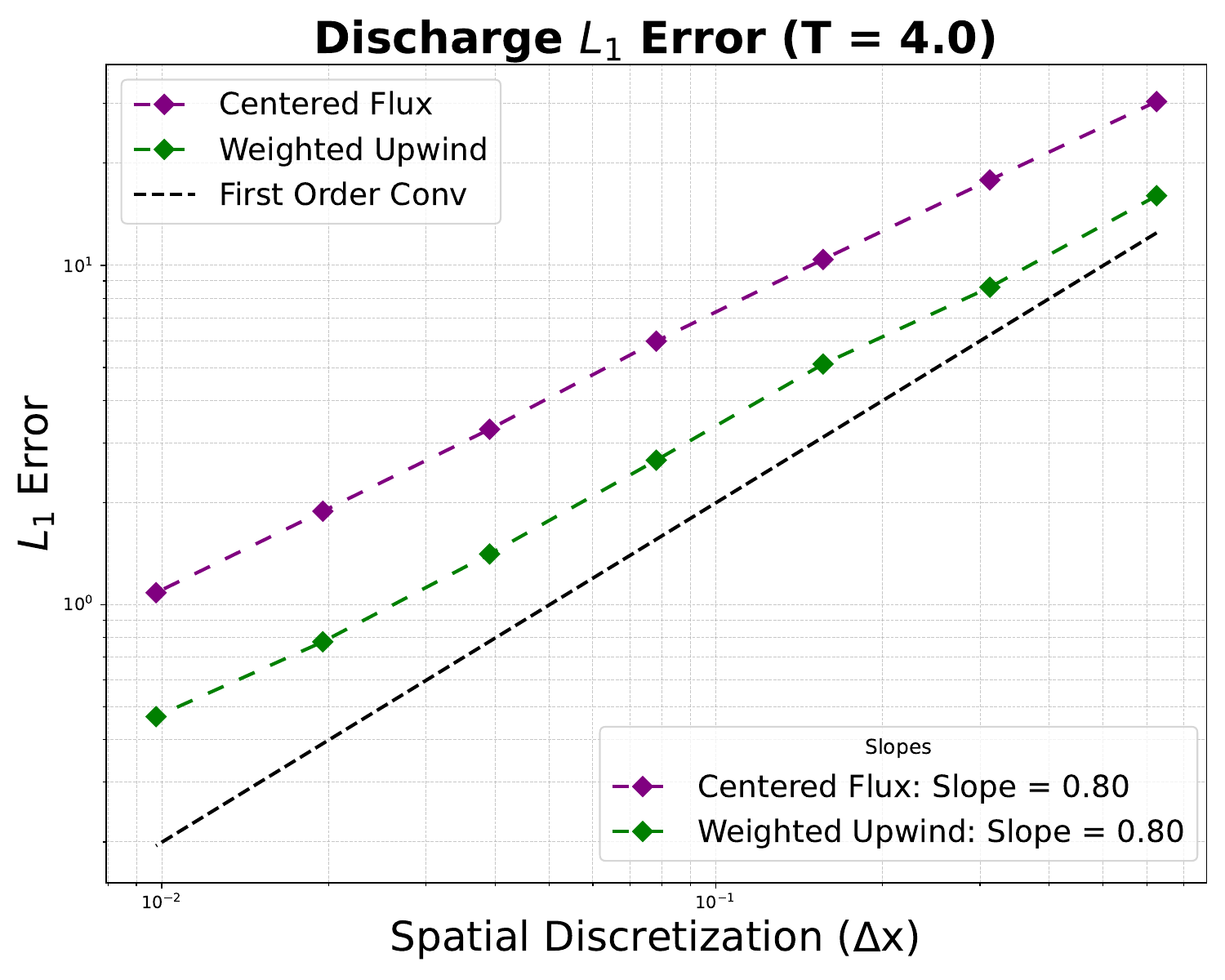}
    \caption{Discharge}
    \label{fig: discharge-sw-hydraulic-flow-conv}
  \end{subfigure}
  \hfill
  \begin{subfigure}[b]{0.30\textwidth}
    \includegraphics[width=\linewidth]{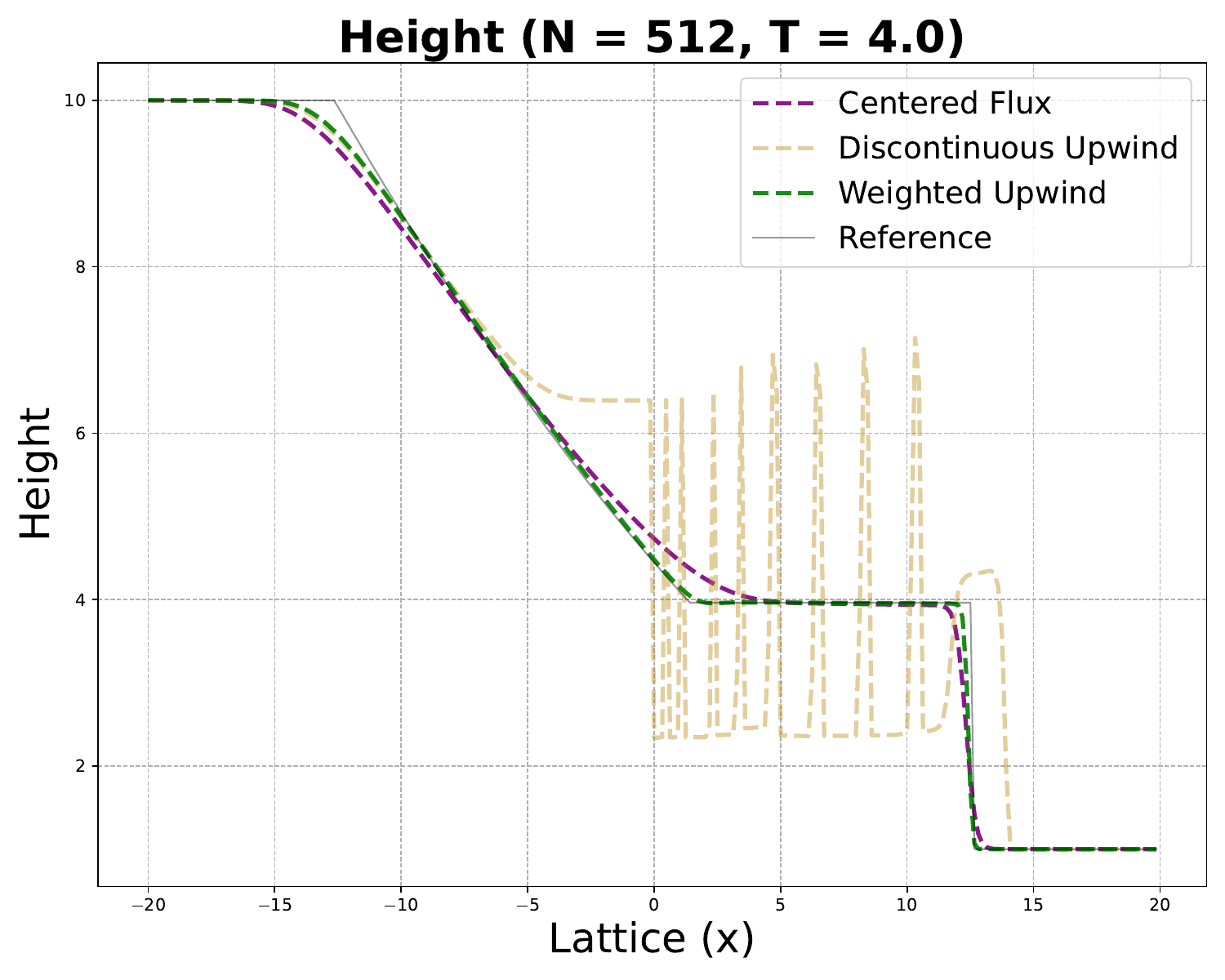}
    \caption{Height}
    \label{fig: height-sw-hydraulic-flow-quantity}
  \end{subfigure}
  \caption{Hydraulic flow (\autoref{tab: shallow_water_riemann_problems}) using discrete lattice points \(N=[64,128,256,\ldots,4096]\), \(\text{CFL}=0.80\), and \(\omega=1.0\).}
  \label{fig: hydraulic-flow-shallow-water}
\end{figure}
\begin{figure}
  \centering
  \begin{subfigure}[b]{0.30\textwidth}
    \includegraphics[width=\linewidth]{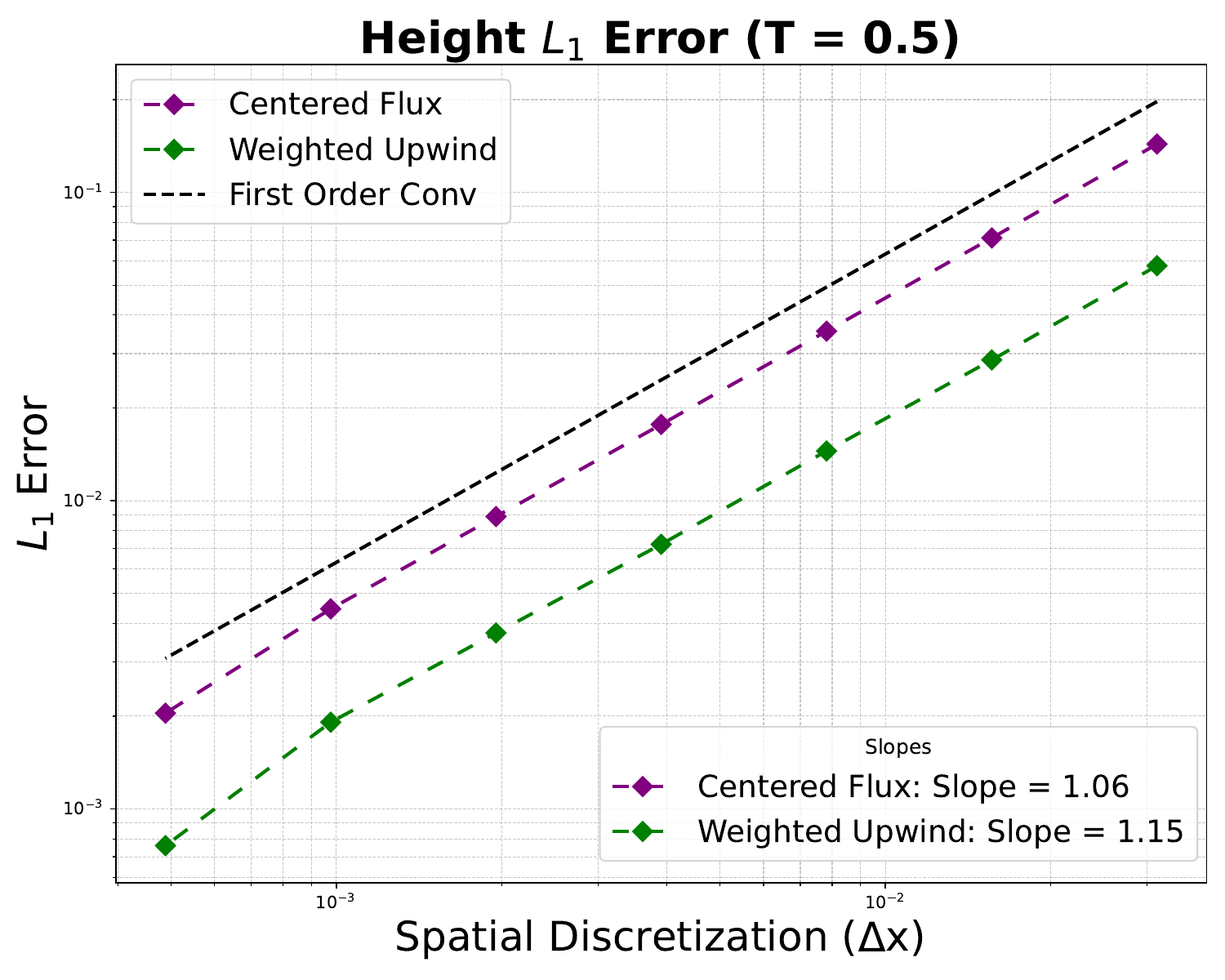}
    \caption{Height}
    \label{fig: height-sw-compression-flow-conv}
  \end{subfigure}
  \hfill
  \begin{subfigure}[b]{0.30\textwidth}
    \includegraphics[width=\linewidth]{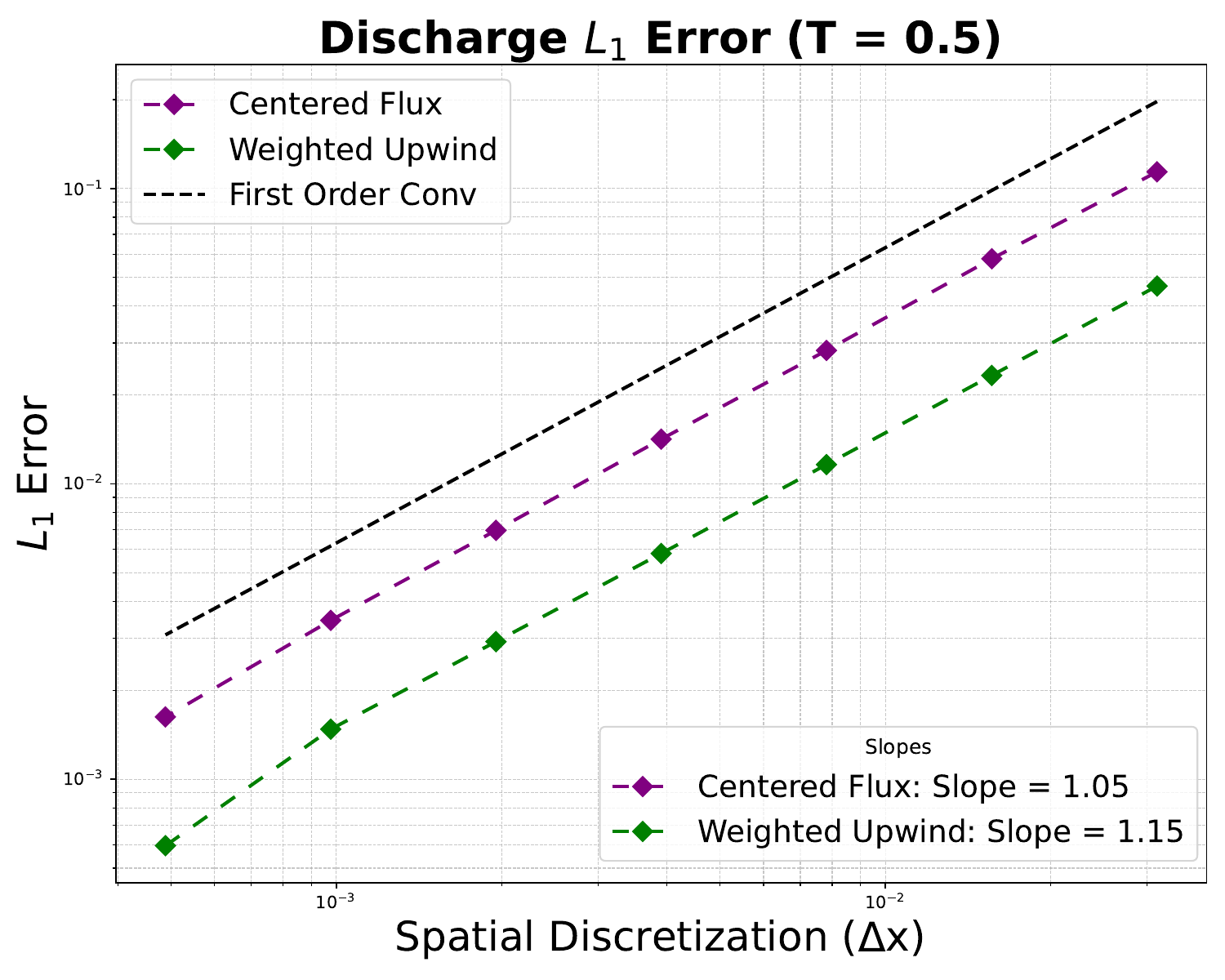}
    \caption{Discharge}
    \label{fig: discharge-sw-compression-flow-conv}
  \end{subfigure}
  \hfill
  \begin{subfigure}[b]{0.30\textwidth}
    \includegraphics[width=\linewidth]{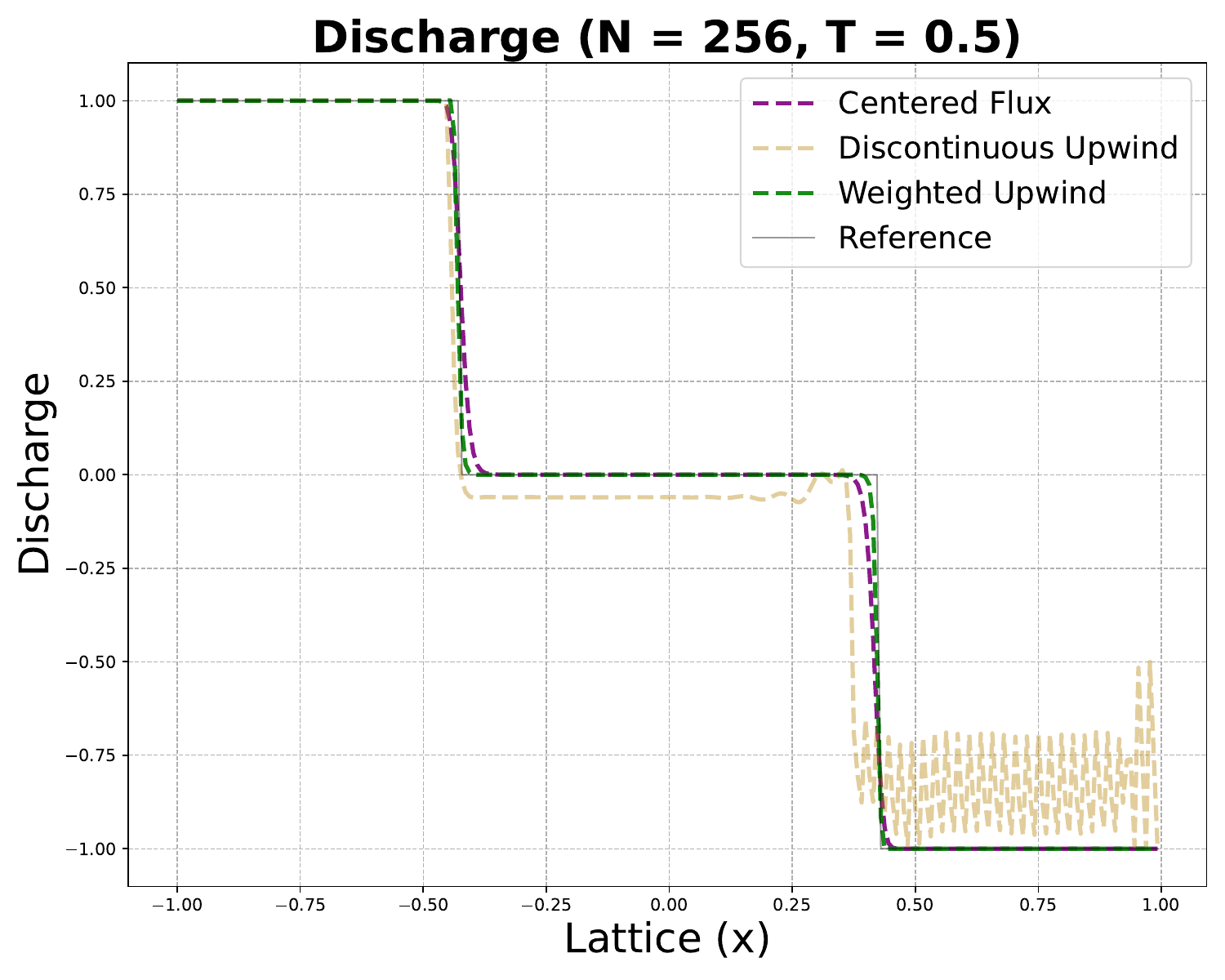}
    \caption{Discharge}
    \label{fig: discharge-sw-compression-flow-quantity}
  \end{subfigure}
  \caption{Compression flow (\autoref{tab: shallow_water_riemann_problems}) using discrete lattice points \(N=[64,128,256,\ldots,4096]\), \(\text{CFL}=0.90\), and \(\omega=1.0\).}
  \label{fig: compression-flow-shallow-water}
\end{figure}
\subsection{Euler Gas Dynamics}\label{ss: euler_gas_dynamics}
Having established that the weighted upwind distribution function set \eqref{eq: adjusted-f1}--\eqref{eq: adjusted-f3} is the most accurate of the three three equilibrium distribution function sets while maintaining equivalent stability for the shallow water system, we now investigate whether these advantages extend to the compressible Euler equations. These equations describe the conservation of mass, momentum, and total energy in an inviscid compressible fluid and provides a demanding set of test problems. In conservative form, the Euler system has components 
\begin{equation}\label{eq: euler_system_of_equations}
U=
\begin{pmatrix}
\rho\\
\rho \vb{u}\\
E
\end{pmatrix},
\qquad
\vb{F}(U)=
\begin{pmatrix}
\rho \vb{u}\\
\rho \vb{u}\otimes \vb{u} + p\vb{I}\\
\vb{u}(E+p)
\end{pmatrix}.
\end{equation}
The system is closed through the ideal gas equations of state,
\[
E=\frac{p}{\gamma-1}+\tfrac12\,\rho\,(\vb{u}\cdot \vb{u}),
\]
Here, $\rho$ is the density, $\vb{u}=(u,v)$ is the macroscopic fluid velocity, $E$ is the total energy, $p$ is pressure, and $\gamma$ is the adiabatic index. The Mach number is defined as,
\[
M=\frac{\lVert\mathbf{u}\rVert}{c}, \quad
c=\sqrt{\frac{\gamma p}{\rho}}.
\]
where $c$ is the speed of sound. When $M < 1$, the flow is subsonic while $M >1$ is a supersonic flow. A \textit{sonic point} occurs when $M = 1$. For the Euler system \eqref{eq: euler_system_of_equations}, we consider smooth problems in $2D$ to test for second order accuracy and Riemann problems in $1D$ to how well different discontinuities are resolved. 

The first smooth problem we consider is a \textit{sinusoidal entropy wave} \cite{Wissocq_2024_Positive_Preserving_VKLB}. The initial perturbation to the density is a transverse density wave aligned along \(x=y\). The initial conditions take the form,
\begin{equation}\label{eq: sinusoidal_solution_euler}
  \rho(x,y,0) = 1 + 0.1\,\sin\bigl(2\pi\,(x+y)\bigr),
  \quad
  u(x,y,0) = 1,
  \quad
  v(x,y,0) = 1,
  \quad
  p(x,y,0) = 1.
\end{equation}
The analytic solution at time \(t\) is then given by
\begin{equation}\label{eq: exact_solution_euler_sinusoidal_solution}
  \rho(x,y,t) = 1 + 0.1\,\sin\bigl(2\pi\,[\,x+y - t\,(u+v)\,]\bigr),
  \quad
  u(x,y,t)=1,
  \quad
  v(x,y,t)=1,
  \quad
  p(x,y,t)=1.
\end{equation}
One important aspect of \eqref{eq: exact_solution_euler_sinusoidal_solution} is that, along each coordinate direction, each eigenvalue of the Jacobian of \eqref{eq: euler_system_of_equations} remains a fixed sign through the domain for all time. There are no sonic points in this problem, as the minimum mach number $M \approx 1.13$ while the maximum is $M \approx 1.25$. \autoref{fig: smooth-solution-euler-convergence} illustrates the centered flux distribution function set \eqref{eq: centered-f1} -- \eqref{eq: centered-f3} and the weighted upwind distribution function set \eqref{eq: adjusted-f1}--\eqref{eq: adjusted-f3} recover second order convergence, with the weighted upwind set having nearly an order of magnitude higher accuracy. The discontinuous upwind distribution function set \eqref{eq: upwinding-f1} -- \eqref{eq: upwinding-f3} was unstable for this problem for all discretizations and CFL constraints, \autoref{ss: CFL}, ranging within \( \nu \in [0.05, 0.10, \ldots, 0.30] \). Hence, no convergence study was conducted for the discontinuous upwind set. 

The second smooth problem we consider is the \textit{vortex evolution problem} \cite{Chu_2025_Euler_Flux_Splitting}. This problem examines a smooth vortex whose center is given at \((x_c,y_c)\) (for us \((x_c, y_c) = (0,0)\)). Let
\[
  r^2 = (x-x_c)^2 + (y-y_c)^2,
  \quad
  T(x,y) \;=\;
  1 \;-\;\frac{(\gamma-1)\,\varepsilon^2}{8\,\gamma\,\pi^2}\,
    e^{\,1 - r^2}\,.
\]
The initial state is given by
\begin{equation}\label{eq: euler_smooth_vortex}
  \rho(x,y,0) = T^{1/(\gamma-1)}, 
  \;
  u(x,y,0) = 1 \;-\;\frac{\varepsilon}{2\pi}\,e^{\tfrac{1-r^2}{2}}\,\bigl(y-y_c\bigr),
  \;
  v(x,y,0) = 1 \;+\;\frac{\varepsilon}{2\pi}\,e^{\tfrac{1-r^2}{2}}\,\bigl(x-x_c\bigr),
  \;
  p(x,y,0) = \rho(x,y,0)^{\,\gamma}.
\end{equation}
The problem \eqref{eq: euler_smooth_vortex} admits an exact solution where the vortex is convected diagonally with velocity \((u,v)=(1,1)\). For each coordinate direction, an eigenvalue of the corresponding Jacobian \eqref{eq: euler_system_of_equations} change sign within the computational domain. The flow also undergoes a transition between subsonic and supersonic regimes. The minimum Mach number is approximately \(M \approx 0.54\) while the maximum is approximately \(M \approx 1.97\). \autoref{fig: smooth-vortex-euler-convergence} illustrates the centered flux distribution function set \eqref{eq: centered-f1}--\eqref{eq: centered-f3} and the weighted upwind distribution function set \eqref{eq: adjusted-f1}--\eqref{eq: adjusted-f3} recover second order convergence, with the weighted upwind set having a higher accuracy than the centered flux set. The discontinuous upwind set \eqref{eq: upwinding-f1}--\eqref{eq: upwinding-f3} was unstable for all discretization sizes $N$ and CFL \(\nu\) in the range \( \nu \in [0.05, 0.10, \ldots, 0.30] \). Hence, no convergence study was conducted for the discontinuous upwind set. 

Next, we examine the one-dimensional Riemann problems in \autoref{tab: euler-riemann-problems}. The Euler system \eqref{eq: euler_system_of_equations} is more complex than the shallow water system, as contact discontinuities can form along with shocks and rarefactions. Riemann problems for the one-dimension Euler system has analytical solutions \cite{Toro_2009_Riemann_Solvers} that allow convergence studies for all three distribution function sets.  

The Sod shock tube problem is a standard test problems for numerical schemes \cite{Toro_2009_Riemann_Solvers}. There are no change in signs of the eigenvalues within the computational domain for all times. \autoref{fig: sod-shock-tube} illustrates all three distribution function sets being stable, with the weighted upwind set \eqref{eq: adjusted-f1}--\eqref{eq: adjusted-f3} having the highest accuracy. 

The initial conditions for the Leblanc shock tube has an eight orders of magnitude difference in energy across the discontinuity (see caption of \autoref{tab: euler-riemann-problems}). In this problem, there is a change in sign of the eigenvalue along the rarefraction. \autoref{fig: leblanc-shock-tube} illustrates discontinuous upwind set \eqref{eq: upwinding-f1}--\eqref{eq: upwinding-f3} experiences an instability, where a discontinuity is forming at the eigenvalue transition. The upwind distribution function set \eqref{eq: adjusted-f1}--\eqref{eq: adjusted-f3} smoothened out this instability, and also had the highest accuracy of all three distribution function sets. 

The sonic fixed point problem includes a nonzero initial velocity, resulting in a change in eigenvalue sign along the rarefaction. \autoref{fig: fixed-sonic-point} illustrates the discontinuous upwind set \eqref{eq: upwinding-f1}--\eqref{eq: upwinding-f3} experiences an instability, where a discontinuity is forming at the eigenvalue transition. Like the previous Euler Riemann problems, the weighted upwind set \eqref{eq: adjusted-f1}--\eqref{eq: adjusted-f3} smoothened out the instabilities, and most accurately resolved the rarefaction, contact, and shock discontinuities. 
\begin{table}[htbp]
    \centering
    \renewcommand{\arraystretch}{1.15}
    \setlength{\tabcolsep}{8pt}
    \begin{tabular}{>{\raggedright}p{4.0cm}cccc}
        \toprule
        \textbf{Parameter} 
        & \textbf{Sod} 
        & \textbf{Leblanc} 
        & \textbf{Sonic}  \\
        \midrule
        Spatial domain 
        & $[-0.5,\,0.5]$ 
        & $[0,\,9.0]$ 
        & $[-1,\,1]$ \\

        Time interval [s] 
        & $[0.0,\,0.2]$ 
        & $[0.0,\,6.0]$ 
        & $[0.0,\,0.5]$ \\

        Initial discontinuity $x_0$ 
        & $0.0$ 
        & $3.0$ 
        & $-0.4$  \\

        Density $(\rho_l,\,\rho_r)$ 
        & $(1.0,\;0.125)$ 
        & $(1.0,\;10^{-3})$ 
        & $(1.0,\;0.125)$  \\

        Velocity $(u_l,\,u_r)$ 
        & $(0.0,\;0.0)$ 
        & $(0.0,\;0.0)$ 
        & $(0.75,\;0.75)$  \\

        Pressure $(p_l,\,p_r)$ 
        & $(1.0,\;0.1)$ 
        & $(6.667\times10^{-2},\;6.667\times10^{-9})$
        & $(1.0,\;0.1)$ \\

        Adiabatic index $\gamma$ 
        & $1.4$ 
        & $\tfrac{5}{3}$ 
        & $1.4$  \\
        \bottomrule
    \end{tabular}
    \caption{Initial conditions for standard 1D Euler Riemann problems.}
    \label{tab: euler-riemann-problems}
\end{table} 
\begin{figure}
  \centering
  \begin{subfigure}[b]{0.30\textwidth}
    \includegraphics[width=\linewidth]{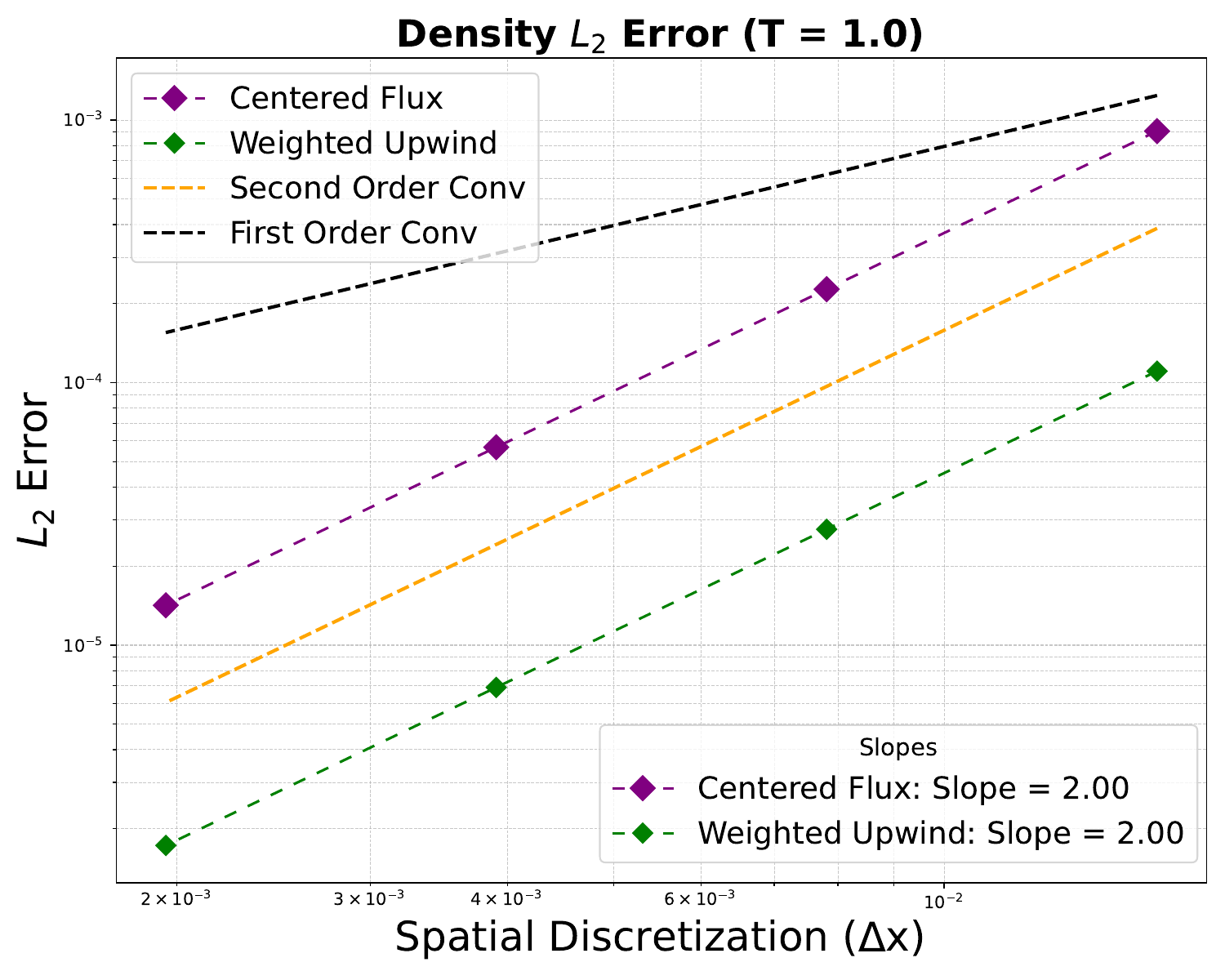}
    \caption{Density L2-error}
    \label{fig: density-euler-smooth-solution}
  \end{subfigure}
  \hfill
  \begin{subfigure}[b]{0.30\textwidth}
    \includegraphics[width=\linewidth]{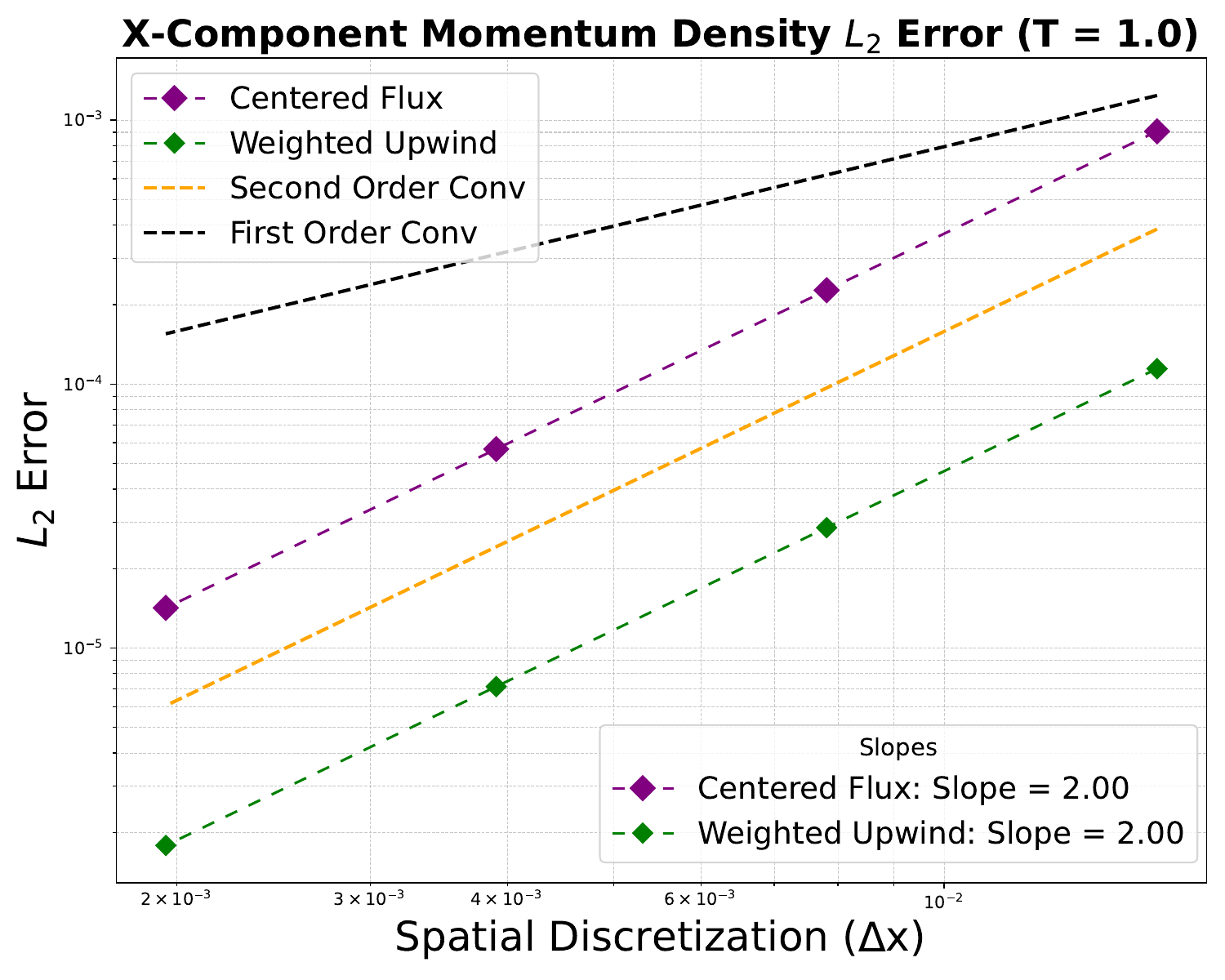}
    \caption{Momentum $\rho u_x$ L2-error}
    \label{fig: momentum-density-euler-smooth-solution}
  \end{subfigure}
  \hfill
  \begin{subfigure}[b]{0.30\textwidth}
    \includegraphics[width=\linewidth]{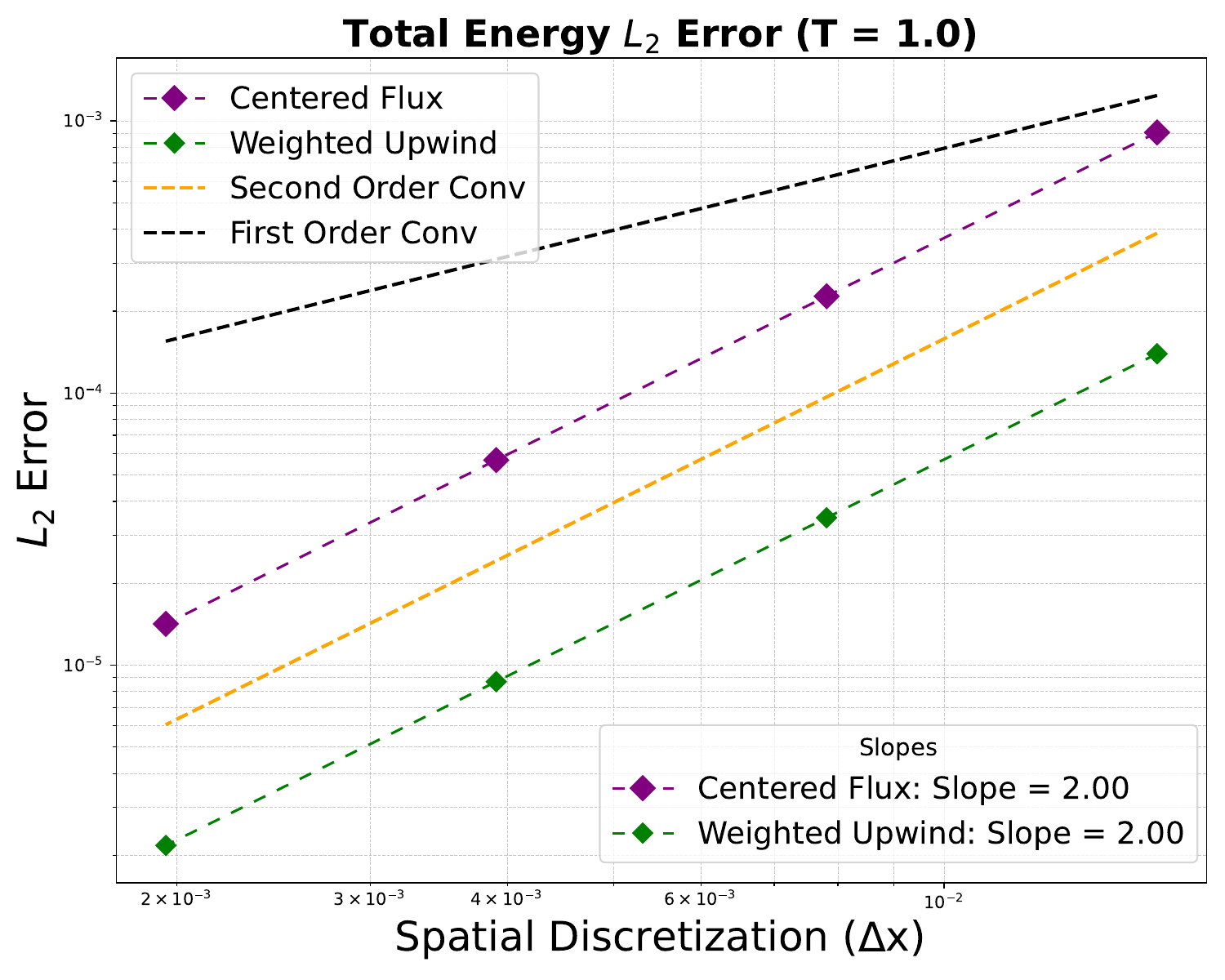}
    \caption{Total Energy $E$ L2-error}
    \label{fig: total-energy-euler-smooth-solution}
  \end{subfigure}
  \caption{Sinusoidal entropy wave \eqref{eq: sinusoidal_solution_euler} showing the \(L_2\) error using discrete lattice points \(N=[64,128,256,512]\), \(\text{CFL}=0.30\), \(\omega=2.0\), and final time \(T=1.0\).}
  \label{fig: smooth-solution-euler-convergence}
\end{figure}
\begin{figure}
  \centering
  \begin{subfigure}[b]{0.30\textwidth}
    \includegraphics[width=\linewidth]{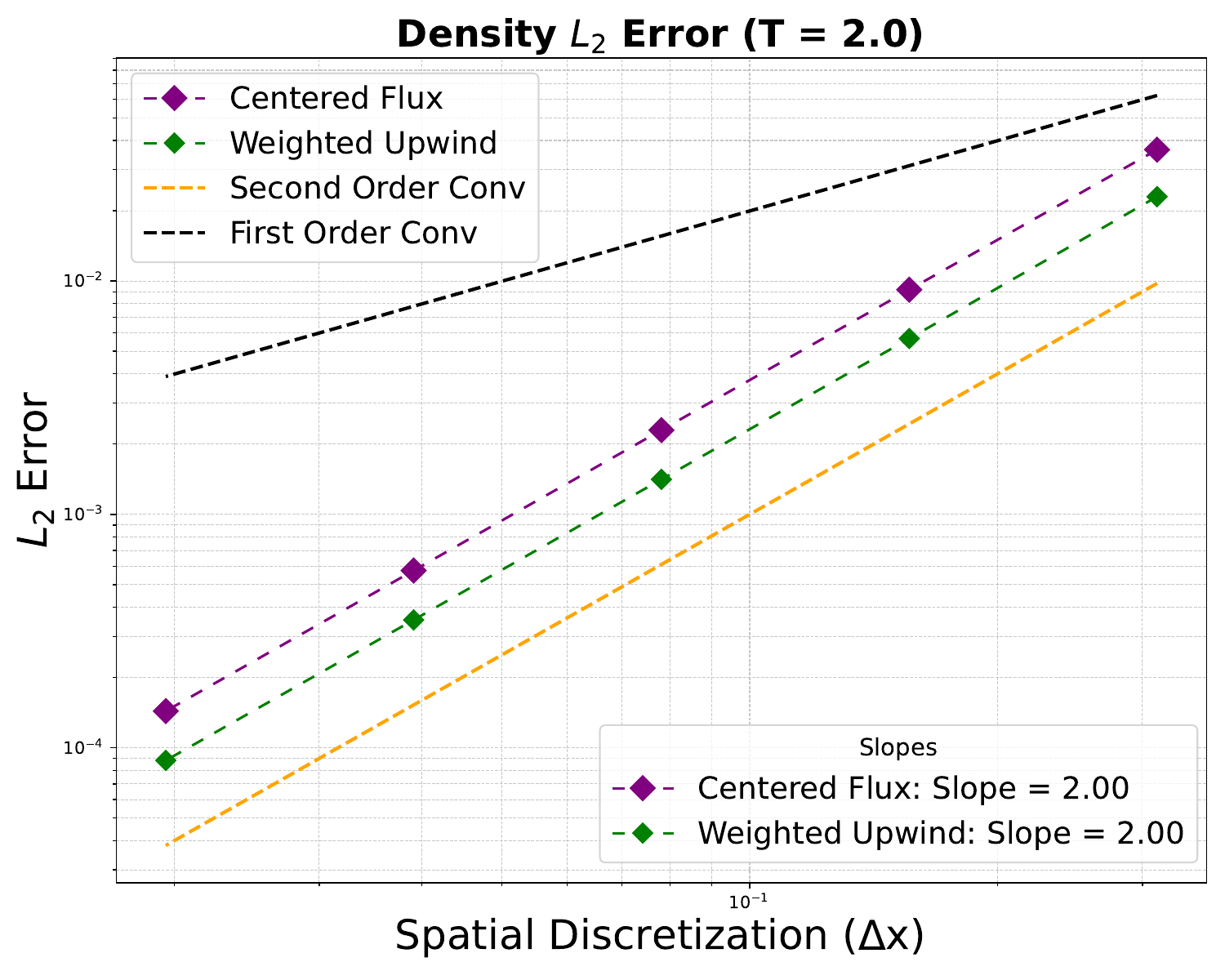}
    \caption{Density}
    \label{fig: density-euler-smooth-vortex}
  \end{subfigure}
  \hfill
  \begin{subfigure}[b]{0.30\textwidth}
    \includegraphics[width=\linewidth]{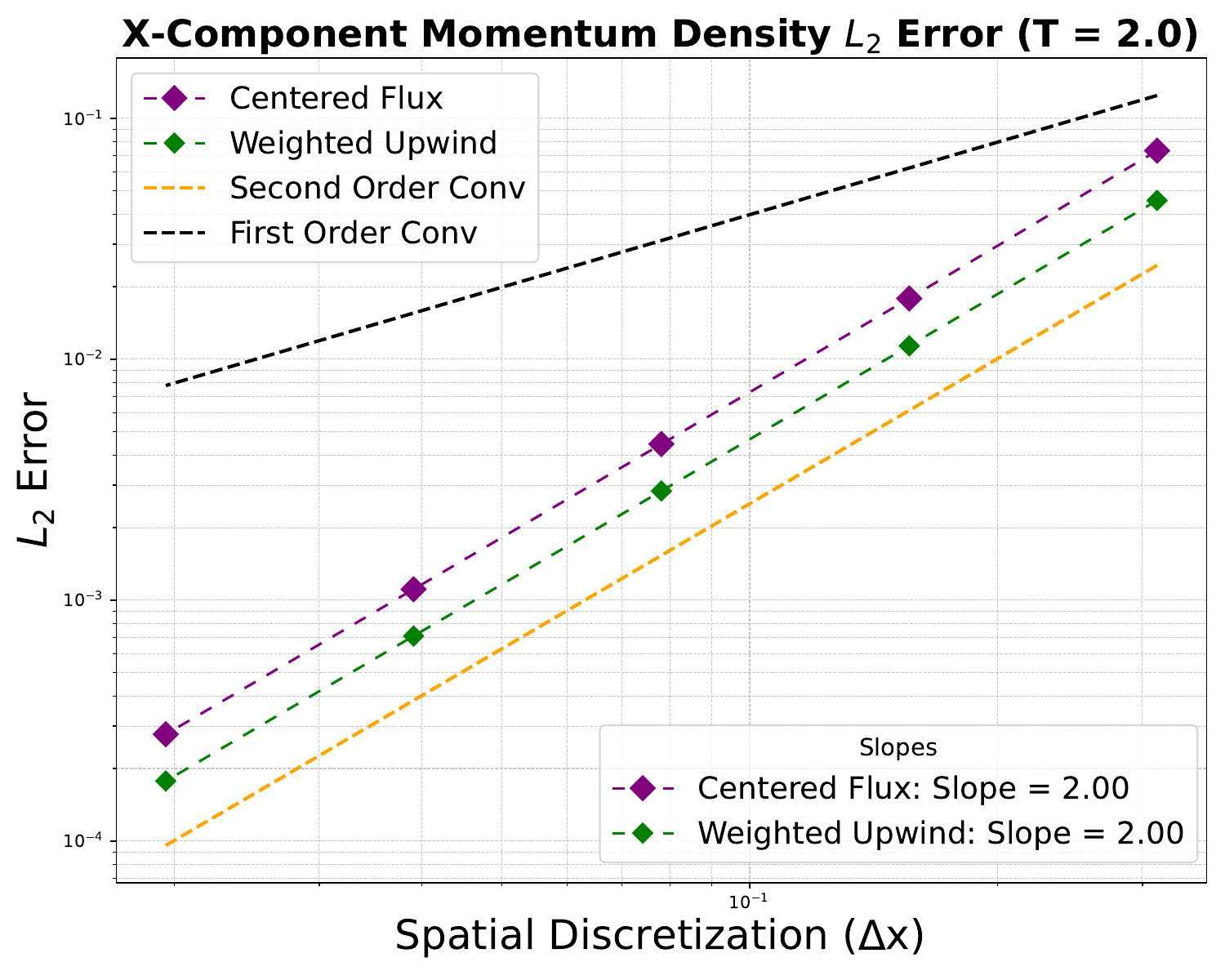}
    \caption{X-Component Momentum Density}
    \label{fig: momentum-density-euler-smooth-vortex}
  \end{subfigure}
  \hfill
  \begin{subfigure}[b]{0.30\textwidth}
    \includegraphics[width=\linewidth]{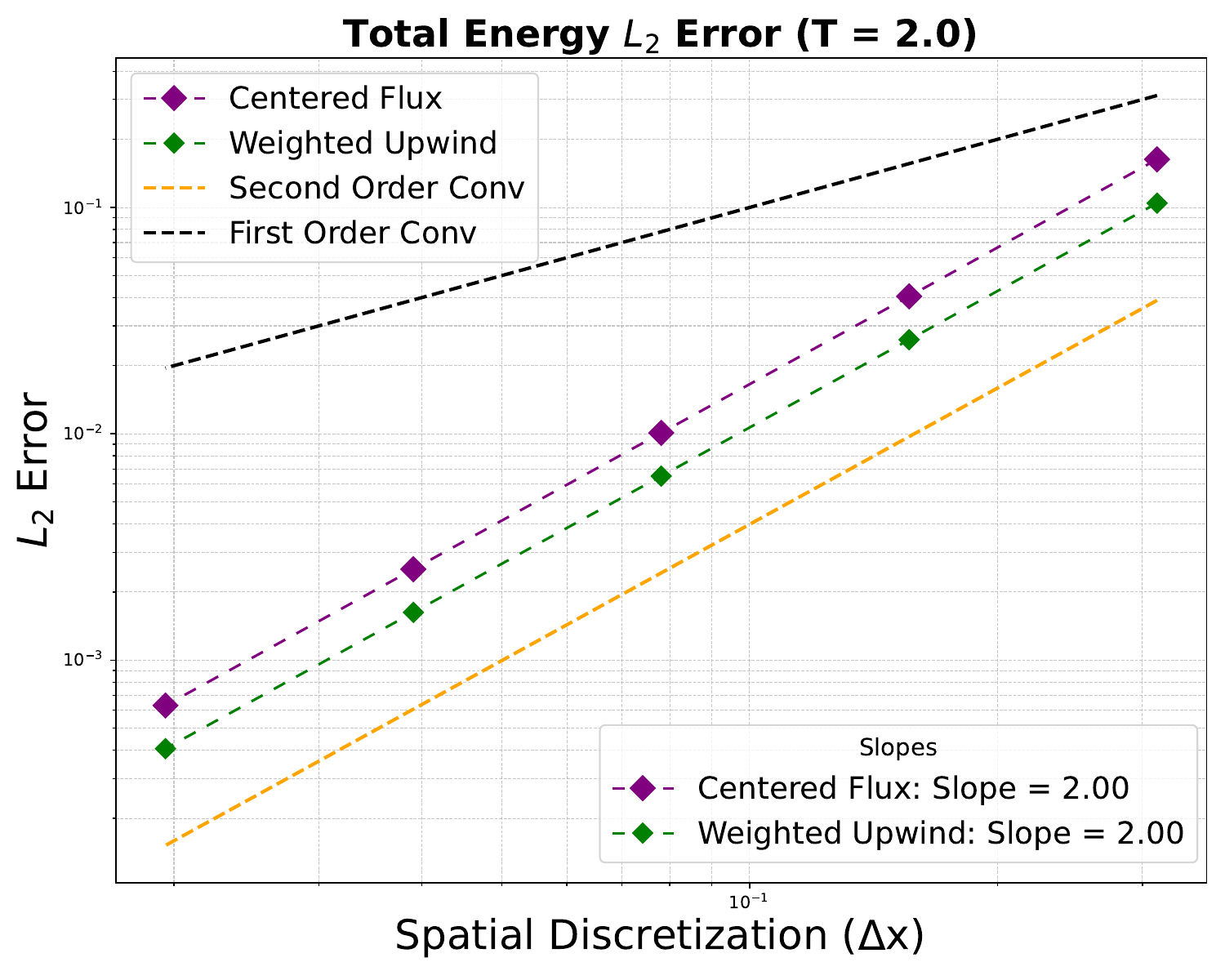}
    \caption{Total Energy}
    \label{fig: total-energy-euler-smooth-vortex}
  \end{subfigure}
  \caption{Euler smooth vortex \eqref{eq: euler_smooth_vortex} showing the \(L_2\) error using discrete lattice points \(N=[64,128,256,512,1024]\), \(\text{CFL}=0.30\), \(\omega=2.0\), and final time \(T=2.0\).}
  \label{fig: smooth-vortex-euler-convergence}
\end{figure}
\begin{figure}
  \centering
  \begin{subfigure}[b]{0.30\textwidth}
    \includegraphics[width=\linewidth]{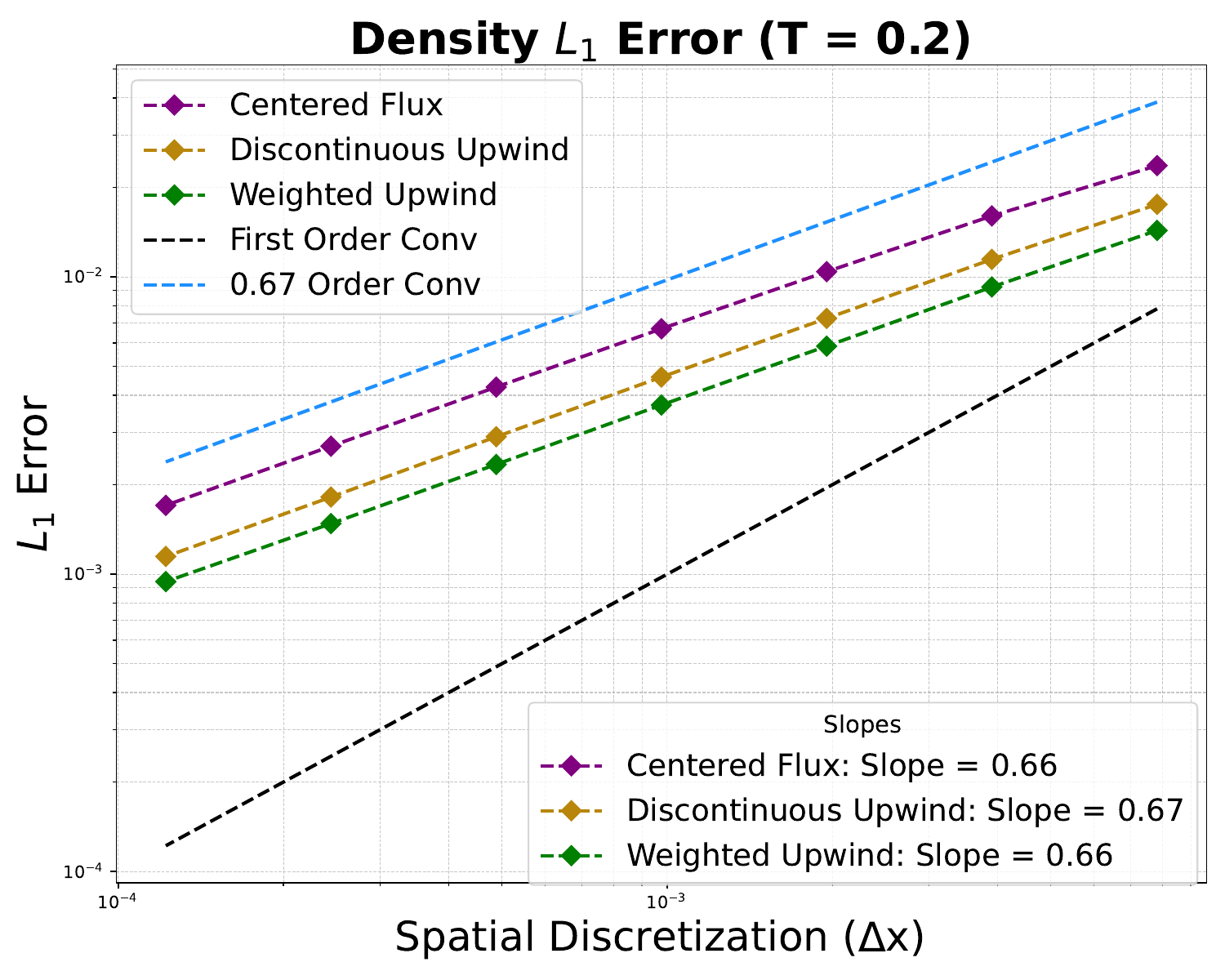}
    \caption{Density}
    \label{fig: density-sod-shock-tube-copnv}
  \end{subfigure}
  \hfill
  \begin{subfigure}[b]{0.30\textwidth}
    \includegraphics[width=\linewidth]{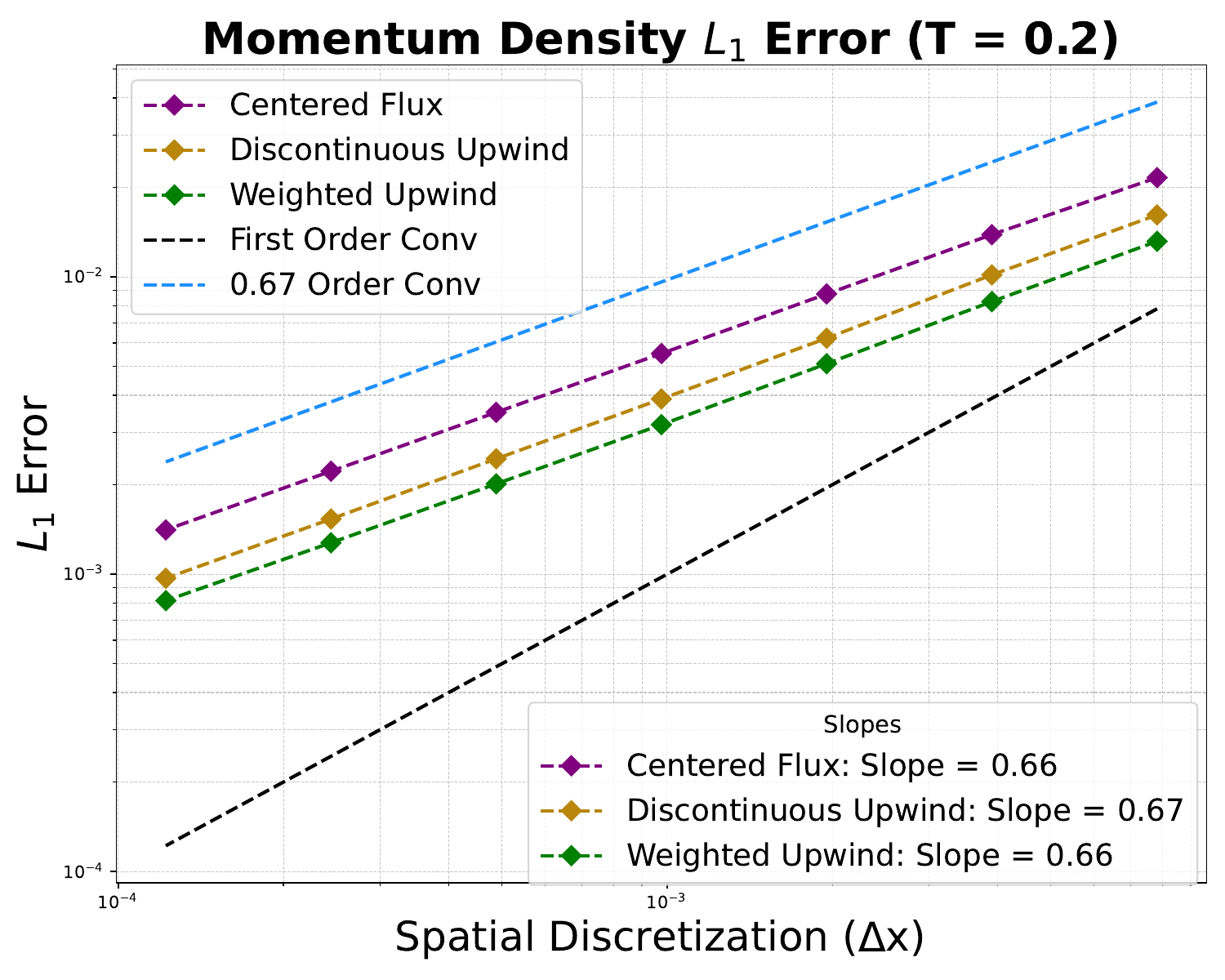}
    \caption{Momentum Density}
    \label{fig: energy-sod-shock-tube-conv}
  \end{subfigure}
  \hfill
  \begin{subfigure}[b]{0.30\textwidth}
    \includegraphics[width=\linewidth]{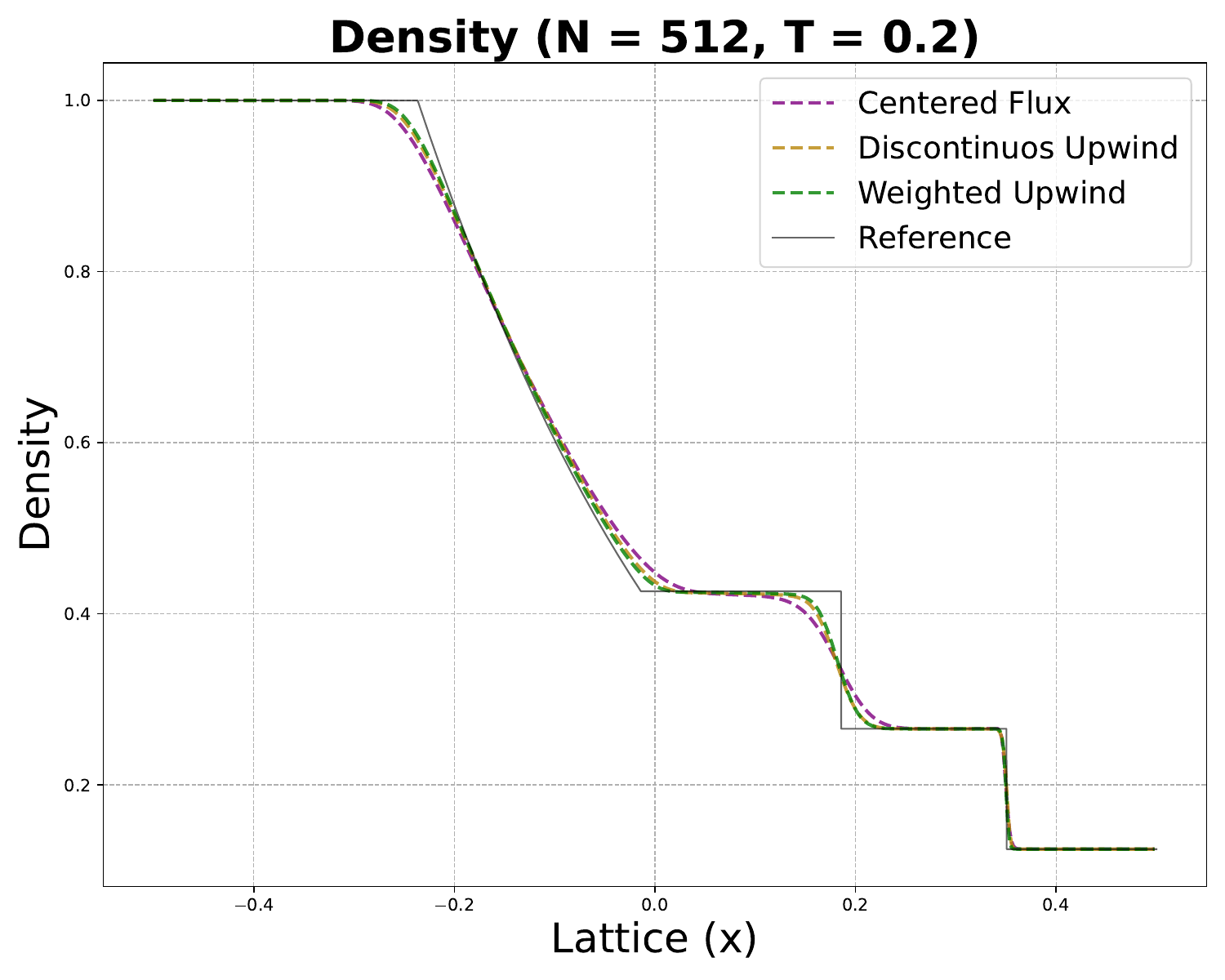}
    \caption{Density}
    \label{fig: density-sod-shock-tube-quantities}
  \end{subfigure}
  \caption{Sod shock tube problem (\autoref{tab: euler-riemann-problems}) using discrete lattice points \(N=[128,256,512,\ldots,8192]\) and \(\text{CFL}=0.90\).}
  \label{fig: sod-shock-tube}
\end{figure}
\begin{figure}
  \centering
  \begin{subfigure}[b]{0.30\textwidth}
    \includegraphics[width=\linewidth]{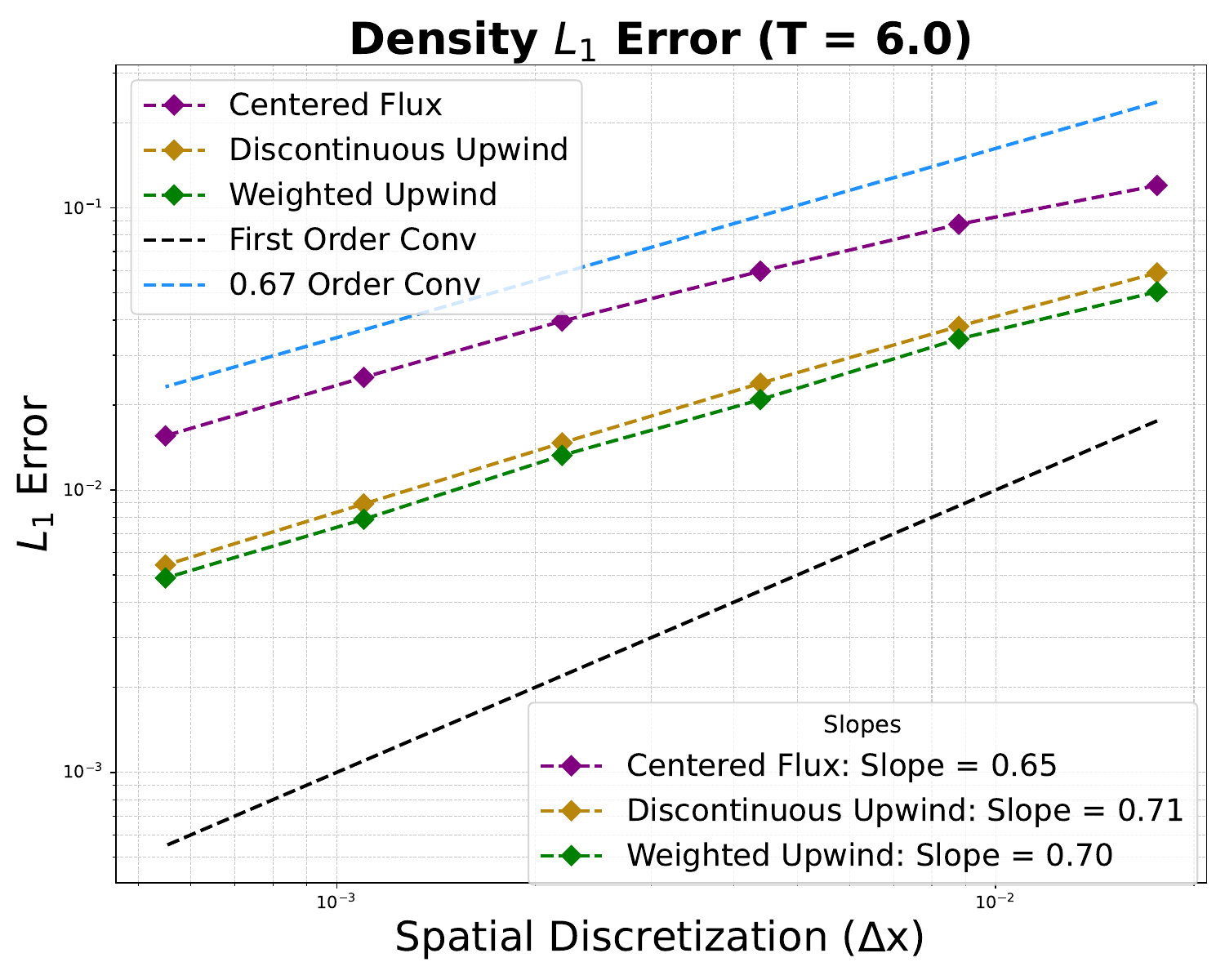}
    \caption{Density}
    \label{fig: density-leblanc-shock-tube-conv}
  \end{subfigure}
  \hfill
  \begin{subfigure}[b]{0.30\textwidth}
    \includegraphics[width=\linewidth]{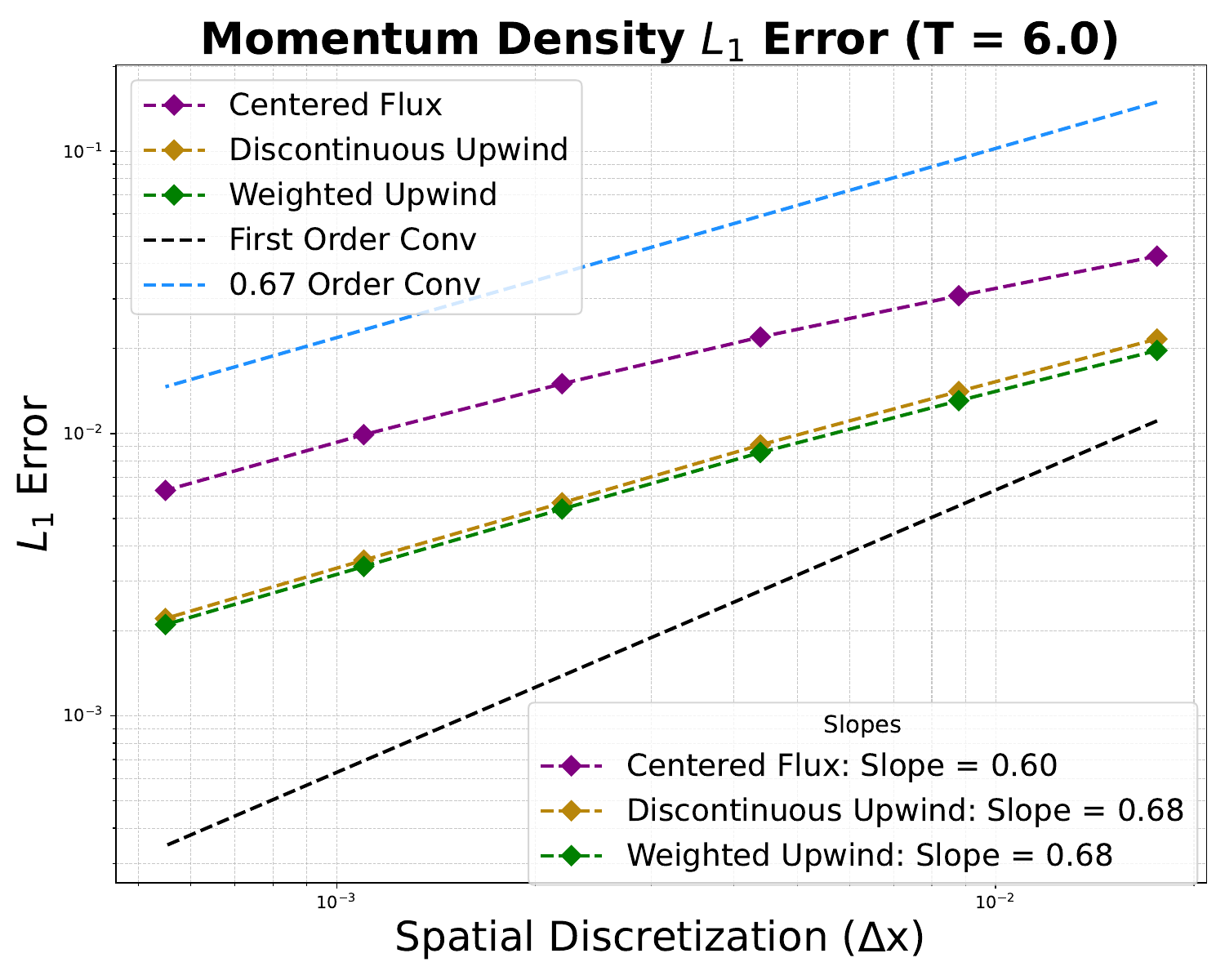}
    \caption{Momentum Density}
    \label{fig: energy-leblanc-shock-tube-conv}
  \end{subfigure}
  \hfill
  \begin{subfigure}[b]{0.30\textwidth}
    \includegraphics[width=\linewidth]{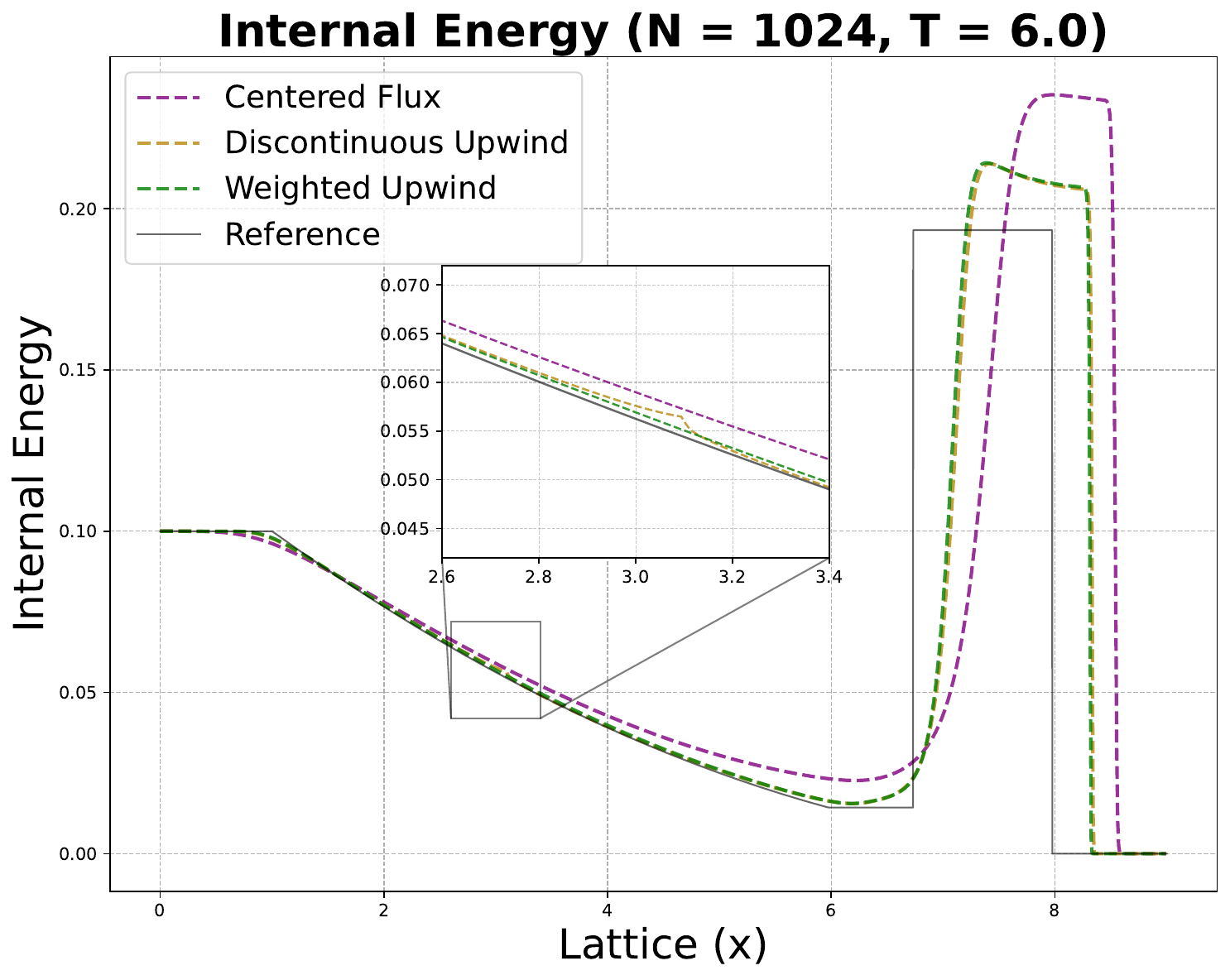}
    \caption{Internal Energy}
    \label{fig: internal-energy-leblanc-shock-tube-quantity}
  \end{subfigure}
  \caption{Leblanc shock tube problem (\autoref{tab: euler-riemann-problems}) using discrete lattice points \(N=[512,1024,\ldots,16384]\) and \(\text{CFL}=0.90\).} 
  \label{fig: leblanc-shock-tube}
\end{figure}
\begin{figure}
  \centering
  \begin{subfigure}[b]{0.30\textwidth}
    \includegraphics[width=\linewidth]{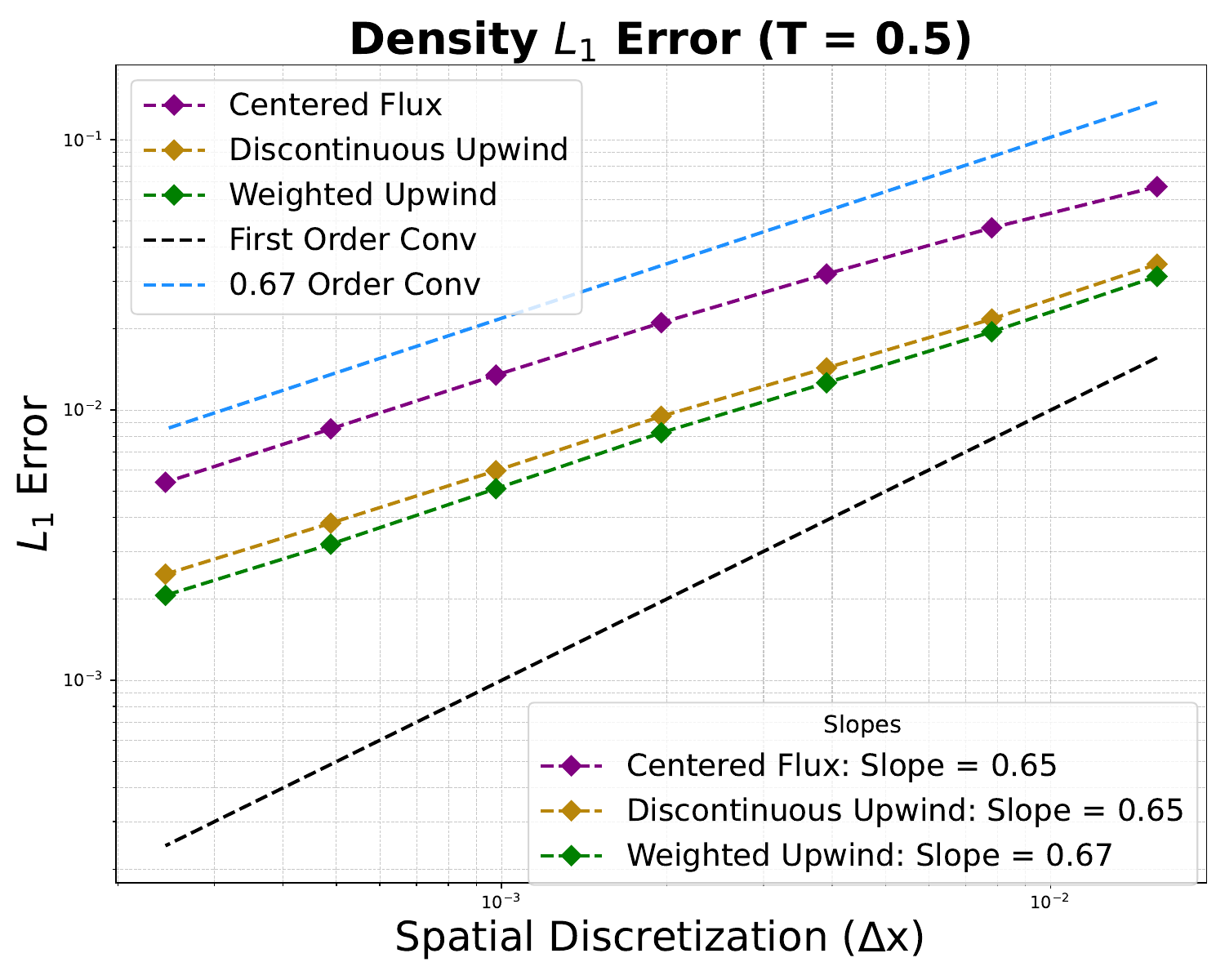}
    \caption{Density}
    \label{fig: density-fixed-sonic-point-conv}
  \end{subfigure}
  \hfill
  \begin{subfigure}[b]{0.30\textwidth}
    \includegraphics[width=\linewidth]{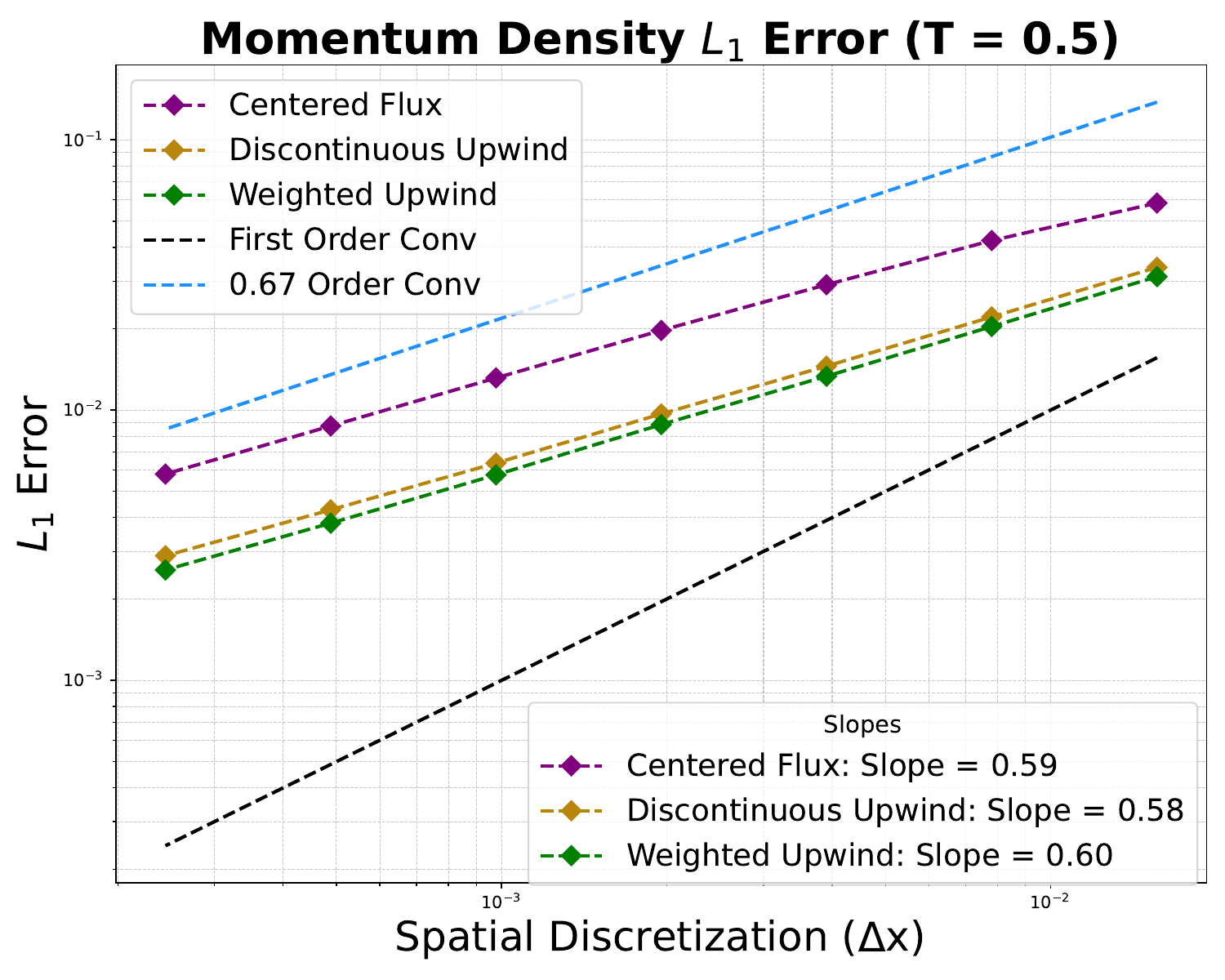}
    \caption{Momentum Density}
    \label{fig: total-energy-fixed-sonic-point-conv}
  \end{subfigure}
  \hfill
  \begin{subfigure}[b]{0.30\textwidth}
    \includegraphics[width=\linewidth]{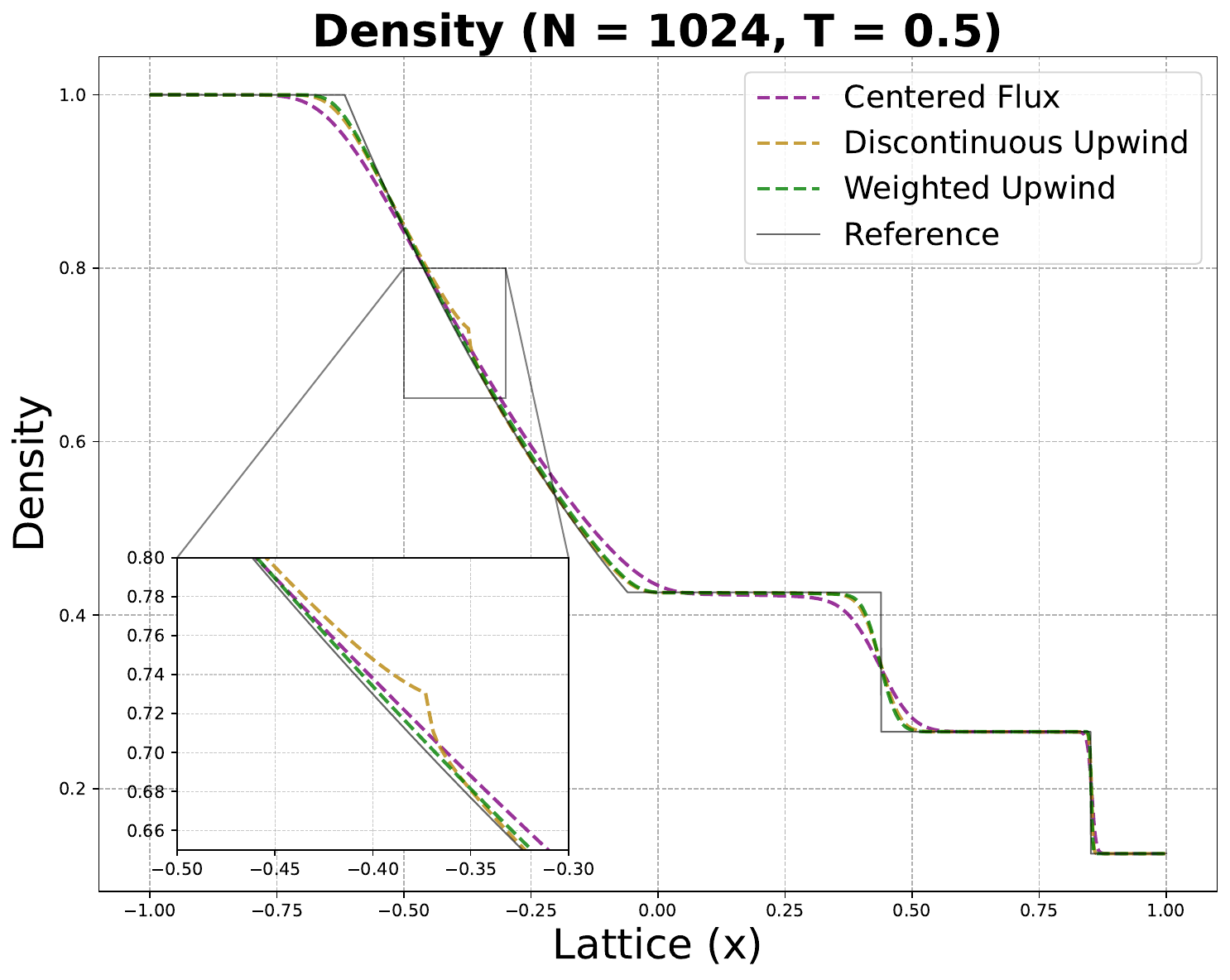}
    \caption{Density}
    \label{fig: density-fixed-sonic-point-quantities}
  \end{subfigure}
  \caption{Fixed sonic point shock tube problem (\autoref{tab: euler-riemann-problems}) using discrete lattice points \(N=[128,256,512,\ldots,8192]\) and \(\text{CFL}=0.90\).}
  \label{fig: fixed-sonic-point}
\end{figure}
\subsection{Ideal Magnetohydrodynamics}\label{ss: ideal_MHD}
Having established the performance of the weighted upwind distribution function set \eqref{eq: adjusted-f1}--\eqref{eq: adjusted-f3} for both the shallow water and compressible Euler equations, we now consider ideal magnetohydrodynamics (MHD). 
The MHD equations are a coupled system of equations for modeling conducting fluids in a magnetic field. They are derived from multifluid plasma equations where hydrodynamics equations for different species (e.g. electrons and ions) are transformed into a center-of-mass reference frame and coupled with a reduced form of Maxwell’s equations. Detailed derivations can be found in \cite{Goedbloed_Poedts_2004}. Ideal inviscid compressible MHD equations are then obtained in the limit when dissipative effects are negligible so that the resistive and viscous terms can be ignored. The applications of MHD type equations include astrophysical and fusion energy plasma simulations \cite{Goedbloed_Poedts_2004}. 

In our formulation of the VKLB ideal MHD system we use a hyperbolic divergence cleaning technique to control \(\nabla \cdot \vb{B} = 0\) 
%\cite{Dedner_2002_MHD_Hyperbolic_Divergence_Cleaning, %Baty_2022_MHD_VKLB}. 
This combines the ideal MHD system with a scalar field \(\psi\), coupling the divergence errors to a hyperbolic PDE that propagates the errors out of the domain with a wave-speed of $c_h$ \cite{Dedner_2002_MHD_Hyperbolic_Divergence_Cleaning}. In our studies we take $c_h = 6*c_f$ where $c_f$ is the fast magnetosonic wave speed \cite{Goedbloed_Poedts_2004,Dedner_2002_MHD_Hyperbolic_Divergence_Cleaning,baty2023robust}. The resulting MHD-DC system has the form of \eqref{eq: hyperbolic-system}, with conserved variables
\begin{equation}\label{eq: ideal_mhd_system}
U = \begin{bmatrix}
    \rho \\
    \rho \mathbf{u}\\
    E\\
    \mathbf{B}\\
    \psi 
\end{bmatrix},
\quad 
\mathbf{F}(U) = \begin{bmatrix}
    \rho \mathbf{u}\\
    \rho \mathbf{u} \otimes \mathbf{u} + \left(p + \frac{1}{2}\mathbf{B} \cdot \mathbf{B} \right) \mathbf{I} - \mathbf{B} \otimes \mathbf{B}\\
    \left(E + p + \frac{1}{2}\mathbf{B}\cdot \mathbf{B} \right)\mathbf{u} - \left(\mathbf{B}\cdot\mathbf{u} \right)\mathbf{B}\\
    \mathbf{u}\otimes \mathbf{B} - \mathbf{B} \otimes \mathbf{u} + \psi \mathbf{I}\\
    c_h^2 \mathbf{B}
\end{bmatrix}
\end{equation}
The system is closed through the ideal gas equation of state,
\[ E = \frac{p}{\gamma - 1} + \frac{1}{2}\rho\left(\mathbf{u} \cdot \mathbf{u}\right) + \frac{1}{2}\left(\mathbf{B}\cdot \mathbf{B}\right)
\]
Here, \(\rho\) is the density, \(p\) is the pressure,  \(\mathbf{u} = (u, v, w) \) is the fluid velocity, \(\gamma\) is the adiabatic index, \(\mathbf{B} = (B_1, B_2, B_3) \) is the magnetic field, and \( F^{(1)}(U) = \mathbf{F} \cdot \mathbf{e}_1, \; F^{(2)}(U) = \mathbf{F} \cdot \mathbf{e}_2 \) are the components of the flux along the x and y coordinate directions respectively. 

The first test considered is the one-dimensional Brio--Wu shock tube, a standard benchmark for ideal MHD \cite{Mabuza_2020_Linearity_Nodal_Variation}. Let $\Omega = [-0.5, 0.5]$. The initial conditions are
\begin{equation}\label{eq: brio_wu_ic}
(\rho, p, v_x, v_y, v_z, B_x, B_y, B_z)(x,0)=
\begin{cases}
(1.0,\;1.0,\;0,\;0,\;0,\;0.75,\;1.0,\;0.0), & x<0.0,\\[4pt]
(0.125,\;0.1,\;0,\;0,\;0,\;0.75,\;-1.0,\;0.0), & x\ge 0.0,
\end{cases}
\qquad \gamma=2.
\end{equation}
The initial conditions \eqref{eq: brio_wu_ic} are that of the Sod shock tube in \autoref{tab: euler-riemann-problems}, with the addition of a nonzero magnetic field \(\mathbf{B}\) that has a discontinuity in the initial condition for ${\bf B}_y$. The reference solution for \eqref{eq: brio_wu_ic} is from \cite{Athena_CPP}, with a high resolution solution using $N = 8000$ grid points. The discontinuous upwind set \eqref{eq: upwinding-f1}--\eqref{eq: upwinding-f3} required the CFL value $\nu \leq 0.45$, otherwise numerical instabilities arose as the pressure became negative. The centered flux distribution function set \eqref{eq: centered-f1}--\eqref{eq: centered-f3} and weighted upwind distribution function set \eqref{eq: adjusted-f1}--\eqref{eq: adjusted-f1} ran at a much higher CFL $\nu = 0.90$. \autoref{fig: brio-wu-shock-tube} illustrates the weighted upwind distribution function set more accurately captured the various discontinuities when compared to the the centered flux distribution set \eqref{eq: centered-f1}--\eqref{eq: centered-f3} and the discontinuous upwind distribution function set \eqref{eq: upwinding-f1}--\eqref{eq: upwinding-f3}. 

The second test is the Orszag--Tang problem, a widely used two-dimensional benchmark for ideal MHD \cite{Mabuza_2020_Linearity_Nodal_Variation}. The initial conditions are
\begin{equation}\label{eq: orzsag_tang}
\Omega = [0,1]^2, \ \gamma=\tfrac{5}{3},\ 
\rho=\tfrac{25}{36\pi},\ 
p=\tfrac{5}{12\pi},\ 
\mathbf{u}=(-\sin(2\pi y),\ \sin(2\pi x),\ 0), \,\mathbf{B}=(-\sin(2\pi y),\ \sin(4\pi x),\ 0)
\end{equation}
This problem evolves from smooth initial data and rapidly develops a complex pattern of interacting shocks, thereby testing both the accuracy and robustness of the scheme. The high resolution reference solution is from Athena \cite{Athena_CPP}, where we used a \(512 \times 512\) grid spacing, and owe the high resolution to the higher-order Godunov discretization. Having already established the weighted upwind set is more accurate than the discontinuous upwind set on a variety of challenging problems, for the Orzsag--Tang problem \eqref{eq: orzsag_tang} we only compared the centered flux distribution function set \eqref{eq: centered-f1} -- \eqref{eq: centered-f3} and the weighted upwind set \eqref{eq: adjusted-f1} -- \eqref{eq: adjusted-f3}. For the weighted upwind distribution function set, we used two different $c$ values in \eqref{eq: 2d_weighted_upwind} to examine the effects of including and excluding this diffusive coefficient (see \autoref{ss: properties}). \autoref{fig: orszag-tang} shows including $c = 0.05$ leads to a more dissipative solution than with $c = 0.0$, which is consistent with the interpretation of $c$ in \autoref{ss: properties}. Using either $c = 0.0$ and $c = 0.05$, \autoref{fig: orszag-tang} shows both WU-VKLB solutions more accurately capture the reference profiles than the centered flux distribution function set \eqref{eq: centered-f1}--\eqref{eq: centered-f3}.
\begin{figure}
  \centering
  \begin{subfigure}[b]{0.30\textwidth}
    \includegraphics[width=\linewidth]{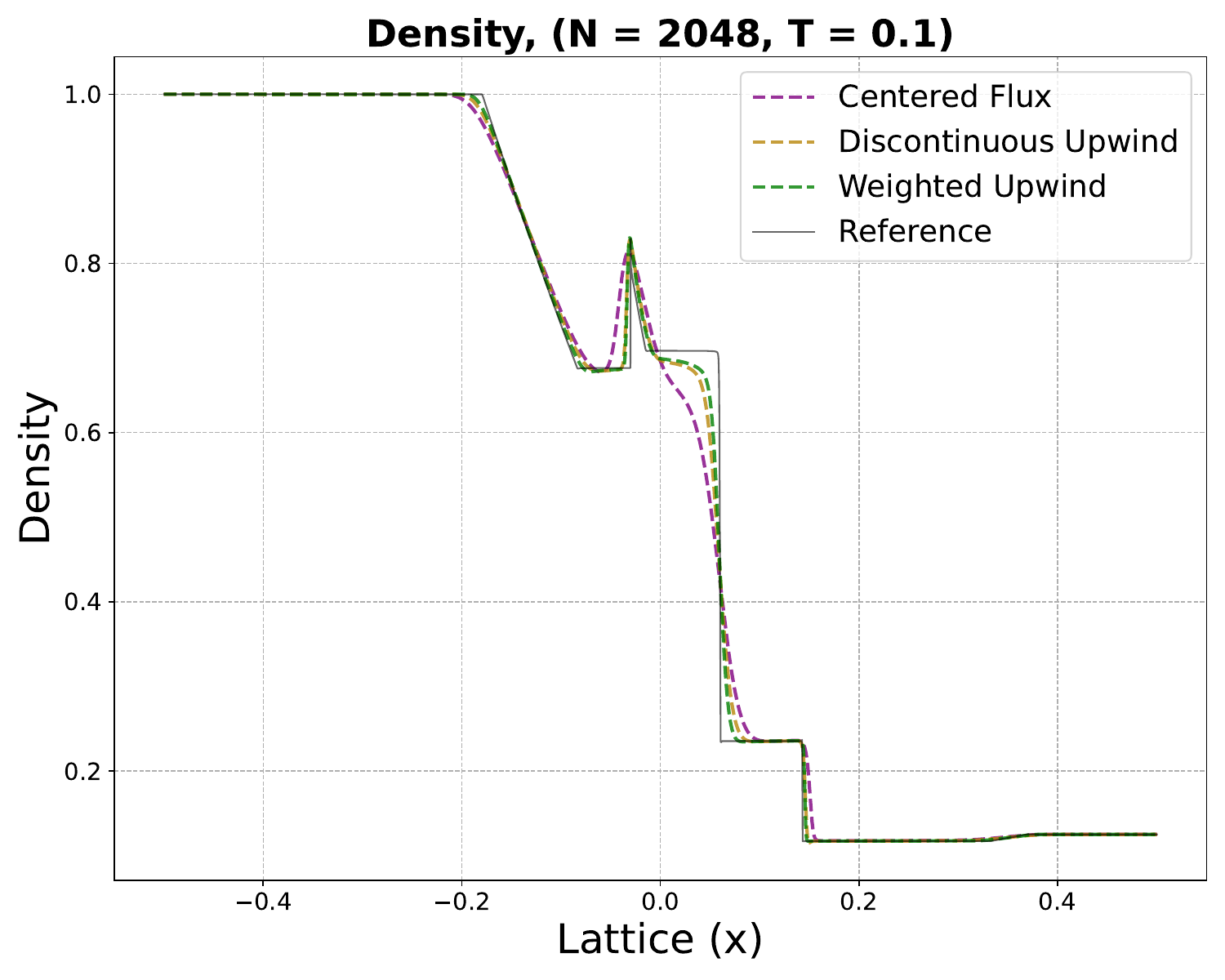}
    \caption{Density}
    \label{fig: density-brio-wu}
  \end{subfigure}
  \hfill
  \begin{subfigure}[b]{0.30\textwidth}
    \includegraphics[width=\linewidth]{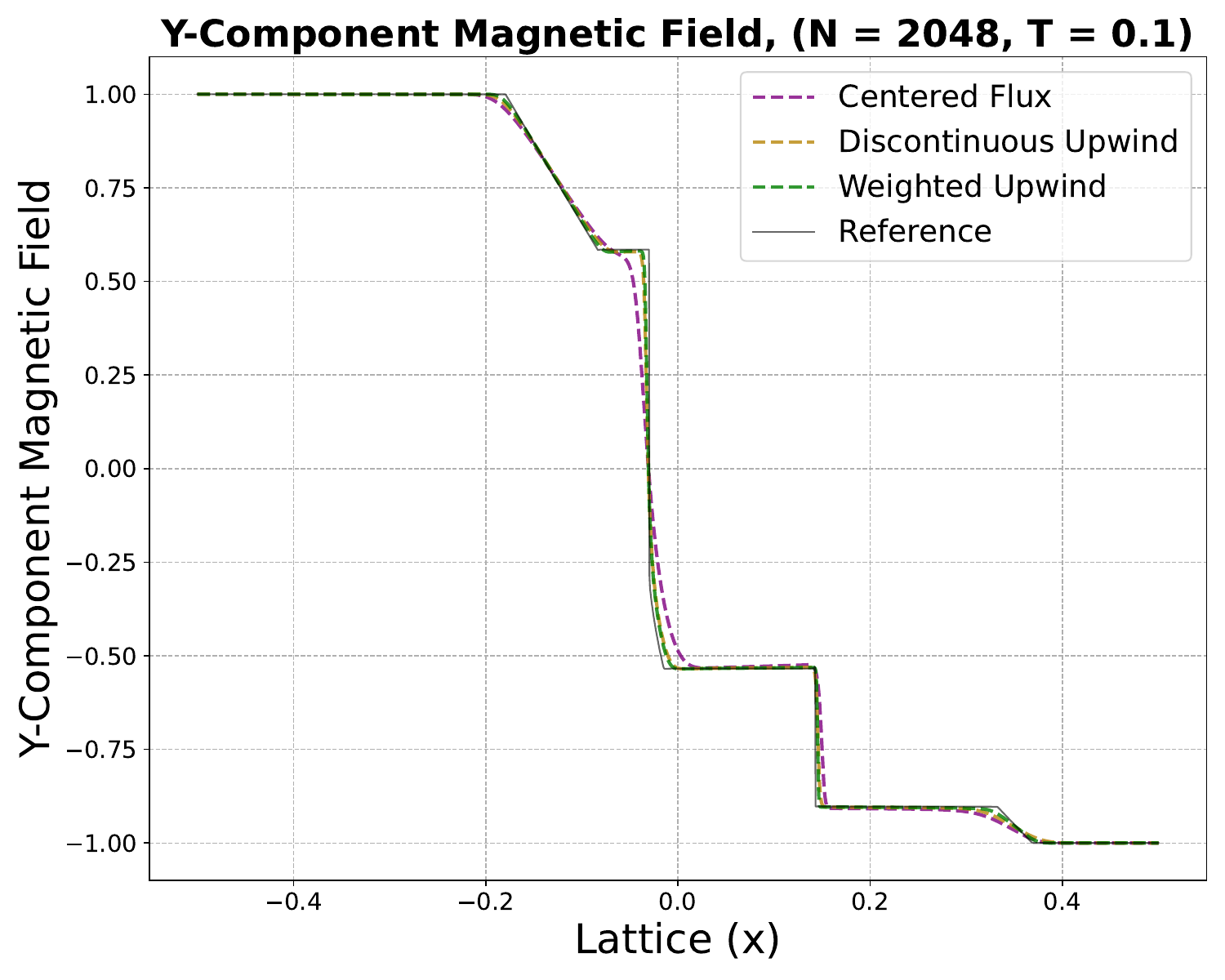}
    \caption{Y-Component Magnetic Field}
    \label{fig: magnetic-field-brio-wu}
  \end{subfigure}
  \hfill
  \begin{subfigure}[b]{0.30\textwidth}
    \includegraphics[width=\linewidth]{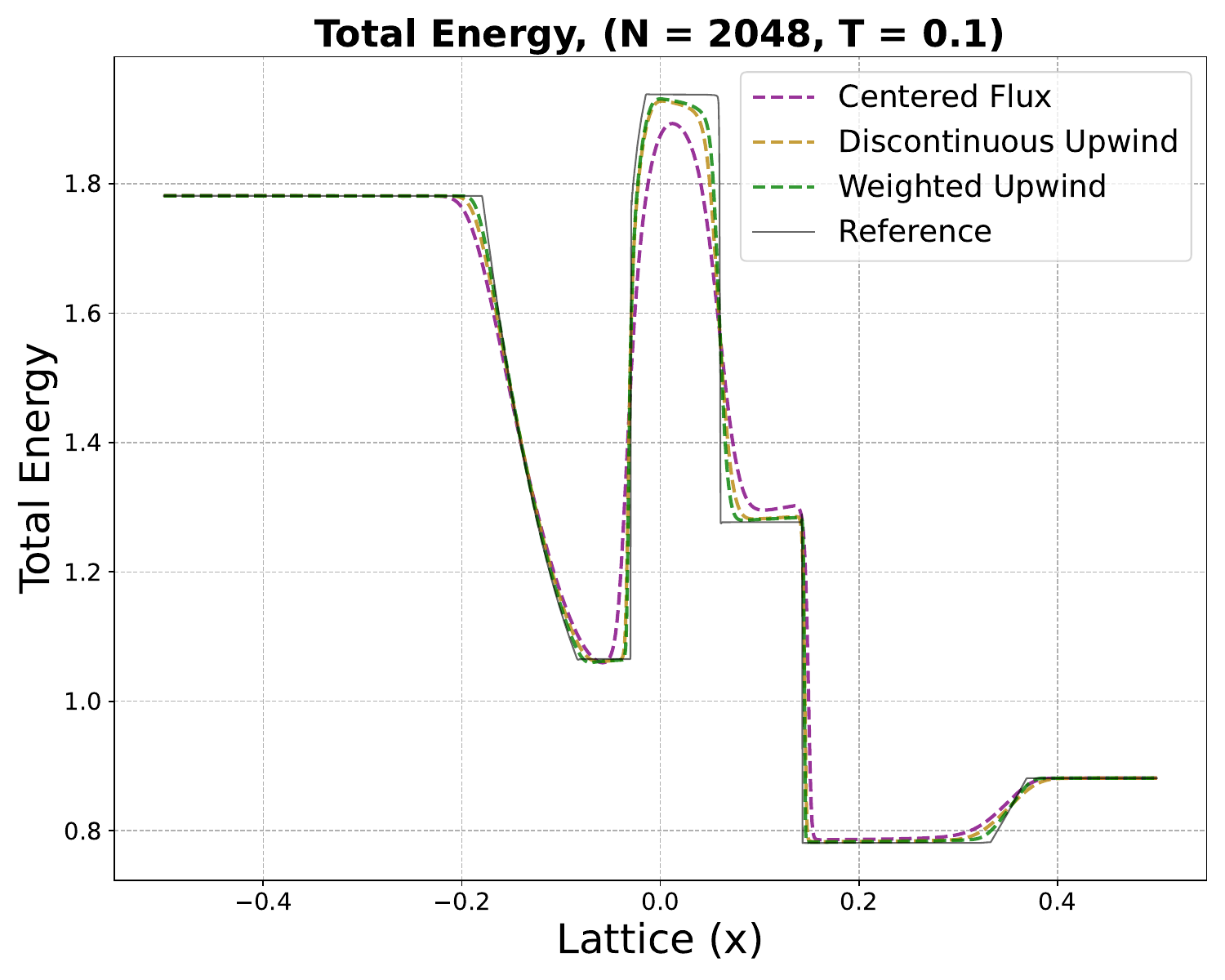}
    \caption{Total Energy}
    \label{fig: total-energy-brio-wu}
  \end{subfigure}
  \caption{Brio--Wu shock tube problem \eqref{eq: brio_wu_ic}. Discrete lattice points $N = 2048$. CFL constraint for the centered flux distribution function set \eqref{eq: centered-f1}--\eqref{eq: centered-f3} and the weighted upwind set \eqref{eq: adjusted-f1}--\eqref{eq: adjusted-f3} is \(CFL = 0.90\). The discontinuous set \eqref{eq: upwinding-f1}--\eqref{eq: upwinding-f3} requires \(CFL =0.45\) for stability.}
  \label{fig: brio-wu-shock-tube}
\end{figure}
\begin{figure}
  \centering
  \begin{subfigure}[b]{0.30\textwidth}
    \includegraphics[width=\linewidth]{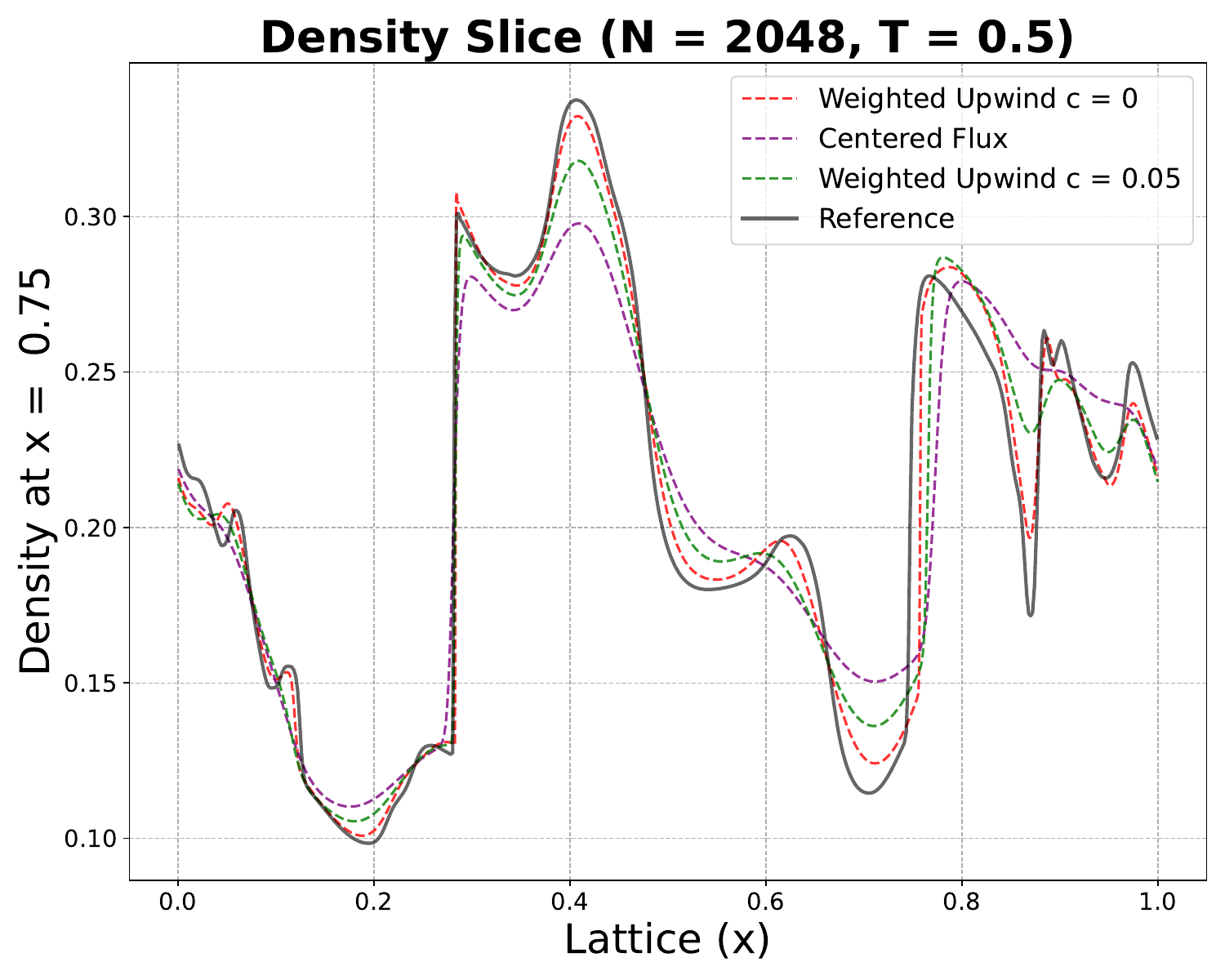}
    \caption{Density Slice at $x = 0.75$}
    \label{fig: density-orzsag-tang}
  \end{subfigure}
  \hfill
  \begin{subfigure}[b]{0.30\textwidth}
    \includegraphics[width=\linewidth]{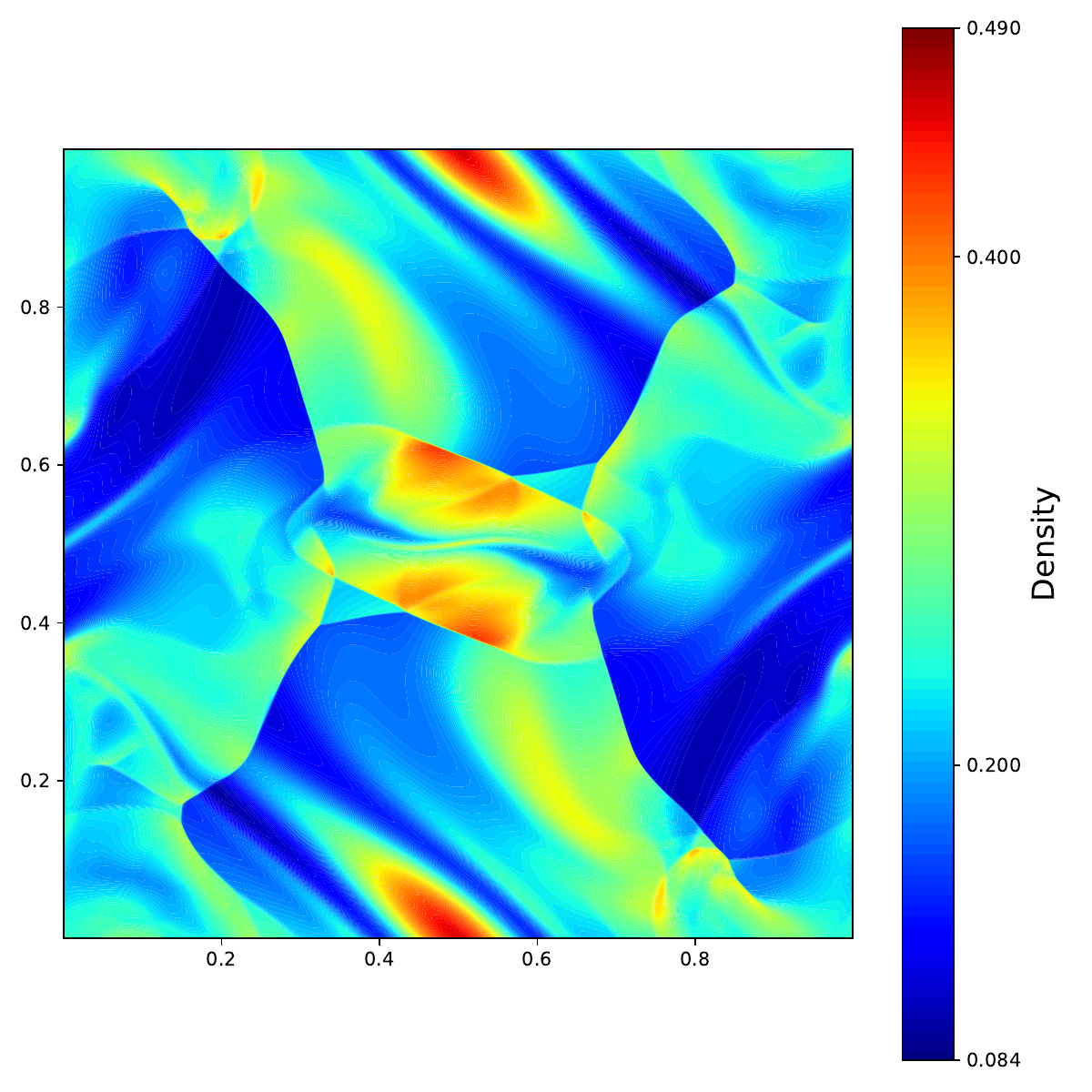}
    \caption{Density with $c = 0.0$}
    \label{fig: pressure-orzsag-tang}
  \end{subfigure}
  \hfill
  \begin{subfigure}[b]{0.30\textwidth}
    \includegraphics[width=\linewidth]{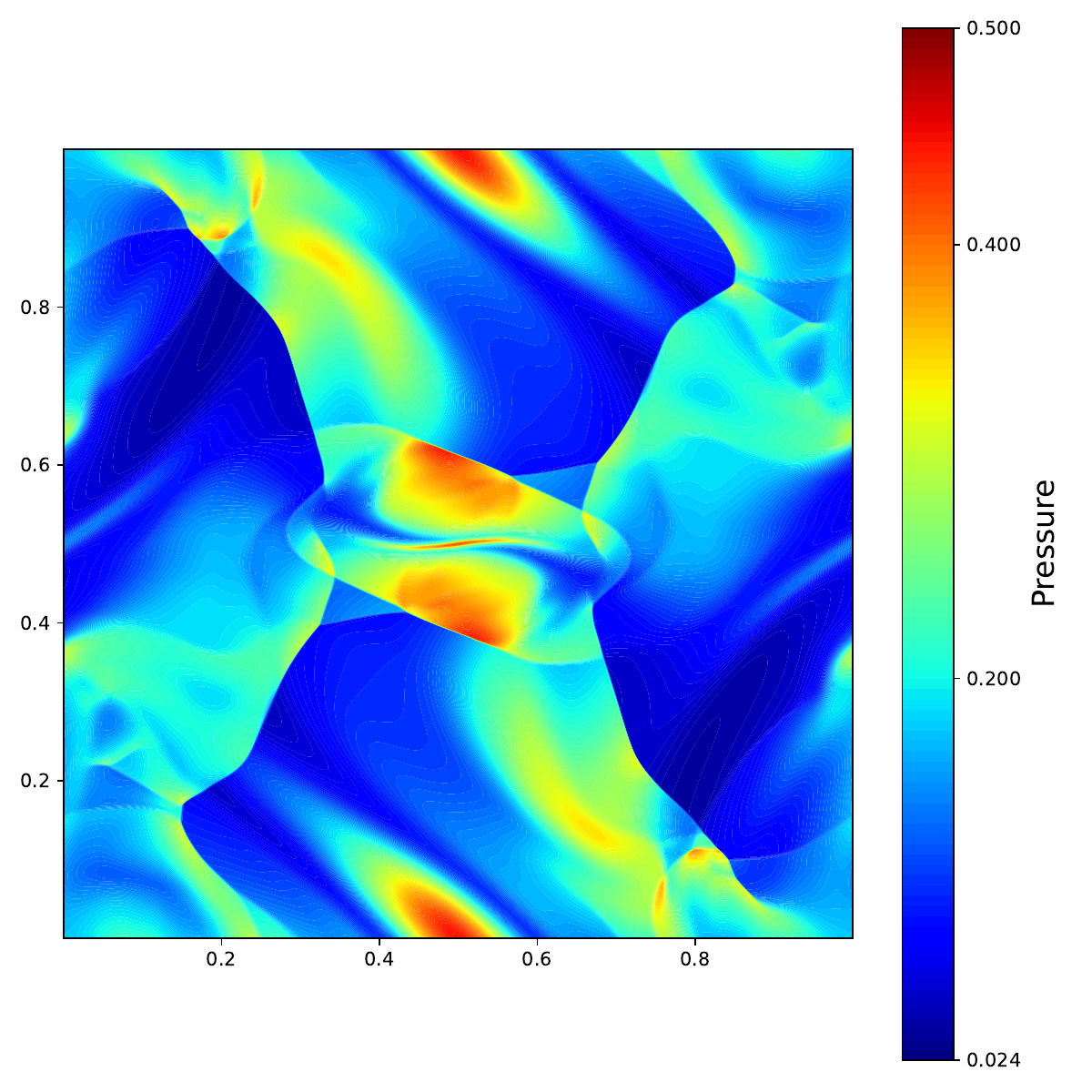}
    \caption{Pressure with $c = 0.0$}
    \label{fig: density-orzsag-tang-athena-comparison}
  \end{subfigure}
  \caption{Orszag--Tang 2D problem \eqref{eq: orzsag_tang} on a \(2048\times2048\) mesh to final time \(T=0.5\). We compare the centered flux set \eqref{eq: centered-f1}--\eqref{eq: centered-f3} and the weighted upwind set \eqref{eq: adjusted-f1}--\eqref{eq: adjusted-f3}, using \(c=0.0\) and \(c=0.05\) to see the affect of additional diffusion (see \autoref{thm:equiv_fd_omega_1}). }
  \label{fig: orszag-tang}
\end{figure}

\subsection{Effect of \boldmath $\omega$ on the solution}
Finally, we examine the effect of increasing the relaxation parameter beyond \(\omega=1\). As discussed in \autoref{ss: properties}, the sub-characteristic condition \eqref{subcharacteristic_condition} restricts the relaxation parameter to \(\omega \in (0,2]\). Some of the theoretical results in \autoref{sec:equivalence_fd_fv} required choosing $\omega=1$ (see \autoref{thm:equiv_fd_omega_1}     and \autoref{lem: FV}). Moreover, the choice  \(\omega \in (0,1]\) results in additional stability properties of the VKLB scheme~\cite{Anandan_2024_VKLB_Upwinding_Source_Term}.
%However, \cite{Anandan_2024_VKLB_Upwinding_Source_Term} showed that the desirable H-theorem property is obtained only under the further restriction \(\omega \in (0,1]\). 
Motivated by these results, we chose \(\omega = 1\) for the shock problems considered previously. We now examine the consequences of choosing \(1 < \omega < 2\). We examine the Euler sod shock tube (see \autoref{tab: euler-riemann-problems}) and Brio-Wu (see \eqref{eq: brio_wu_ic}). For the centered flux distribution function set \eqref{eq: centered-f1}--\eqref{eq: centered-f3}, \autoref{fig: centered-flux-sod-multiple-omega} and \autoref{fig: brio-wu-multiple-omega} illustrates larger values of \(\omega\) can improve accuracy, but this comes at the cost of problem-dependent tuning. For example, instabilities occur at \(\omega=1.4\) for the Euler Sod shock tube problem whereas the Brio--Wu shock tube problem remains stable until \(\omega=1.7\), where after this value the pressure drops below zero. In contrast, for the weighted upwind distribution function set \eqref{eq: upwinding-f1}--\eqref{eq: upwinding-f3}, the Euler Sod and Brio--Wu shock tube problems become unstable at \(\omega=1.4\). However, \autoref{fig: weighted-upwind-sod-mutliple-omega} and \autoref{fig: brio-wu-multiple-omega} illustrates the corresponding solution profiles are very close for the stable values of \(\omega\), indicating that choosing \(\omega=1\) is a sufficient choice. Moreover, for the Brio--Wu problem, the weighted upwind distribution function set with \(\omega=1\) is more accurate than the centered flux distribution function set with \(\omega=1.6\). These result illustrate the potential loss of stability that can occur for a choice of $\omega > 1$ and illustrates the difficulty of pursuing an apriori choice of $\omega > 1$ in attempt to obtain a marginal increase in resolution. In the conclusions that follow we briefly comment further on this issue.    
\begin{figure}
  \centering
  \begin{subfigure}[b]{0.30\textwidth}
    \includegraphics[width=\linewidth]{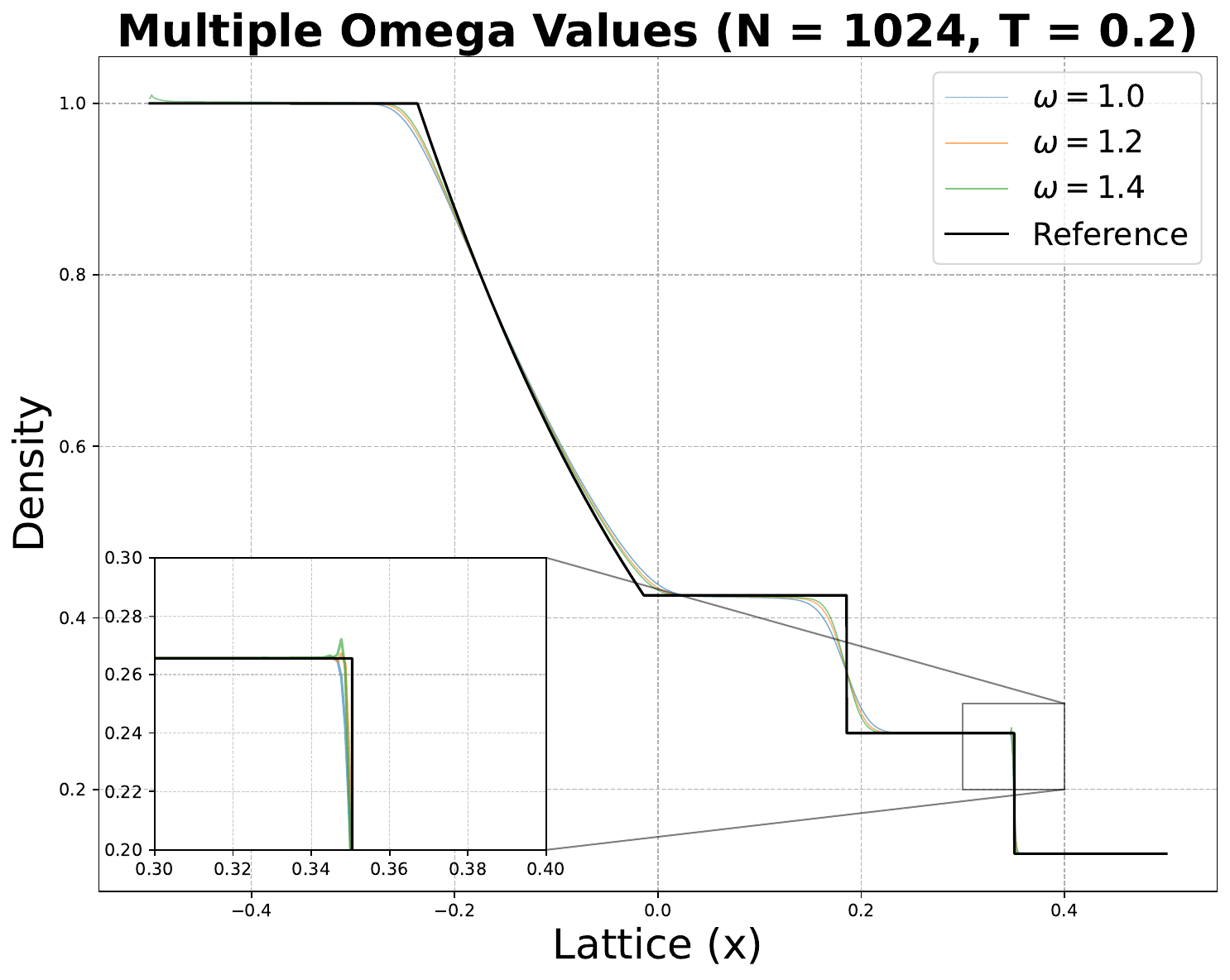}
    \caption{Centered flux, Sod problem.}
    \label{fig: centered-flux-sod-multiple-omega}
  \end{subfigure}
  \hfill
  \begin{subfigure}[b]{0.30\textwidth}
    \includegraphics[width=\linewidth]{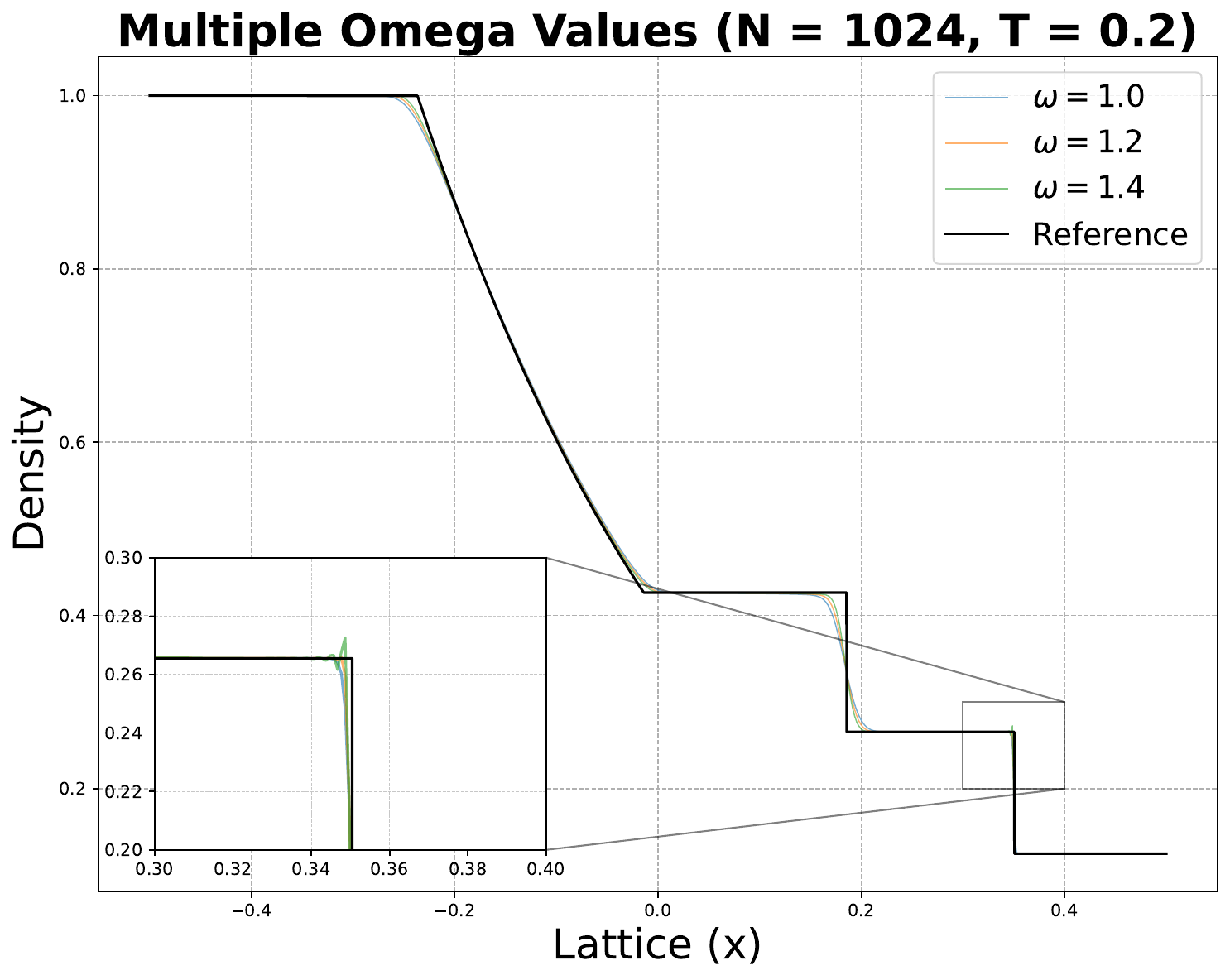}
    \caption{Weighted upwind, Sod problem.}
    \label{fig: weighted-upwind-sod-mutliple-omega}
  \end{subfigure}
  \hfill
  \begin{subfigure}[b]{0.30\textwidth}
    \includegraphics[width=\linewidth]{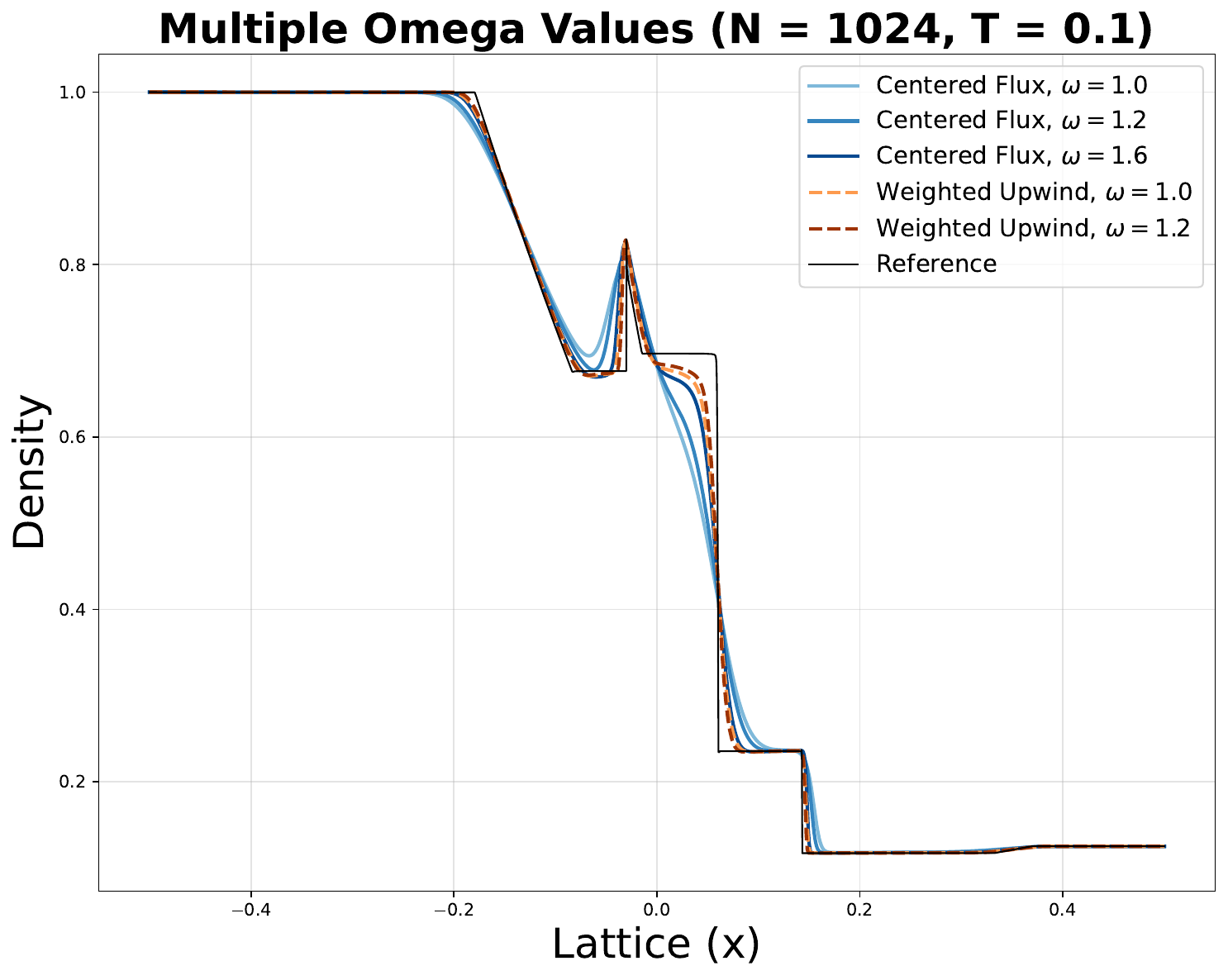}
    \caption{Brio--Wu comparison.}
    \label{fig: brio-wu-multiple-omega}
  \end{subfigure}
  \caption{Effect of increasing \(\omega>1\) on density stability for the centered flux
set \eqref{eq: centered-f1}--\eqref{eq: centered-f3} and weighted upwind set
\eqref{eq: adjusted-f1}--\eqref{eq: adjusted-f3}. Panels show the Sod shock
tube, \autoref{tab: euler-riemann-problems}, and Brio--Wu problem,
\eqref{eq: brio_wu_ic}. For Sod, both formulations become unstable at \(\omega=1.4\). For Brio--Wu, negative pressure occurs for the centered flux set when \(\omega>1.6\) and for the weighted upwind set when \(\omega>1.2\).
  }
  \label{fig: multiple-omega-values}
\end{figure}

\section{Conclusion }\label{s:7}
This work introduced a weighted upwind vector kinetic lattice Boltzmann (WU-VKLB) method for multidimensional hyperbolic conservation laws. Building on flux splitting equilibrium distribution functions as presented in \cite{Anandan_2024_VKLB_Upwinding_Source_Term} and the flux splitting \eqref{eq: eigenvalue-split-discrete}, we proposed a continuous, eigenvalue parameterized transition \eqref{fig:alpha-sigmoid} between the flux structure of the centered flux distribution function set \eqref{eq: centered-f1} -- \eqref{eq: centered-f3} and and the discontinuous upwind distribution function set \eqref{eq: upwinding-f1}--\eqref{eq: upwinding-f3}. This construction preserves the stream and collide simplicity of the VKLB framework \autoref{s: algorithm_VKLB} while also having the desired stability and accuracy of each limiting sets. Using an equivalent finite difference viewpoint, we showed the addition of the state vector yields an additional dissipative term acting directly on the conserved variables, providing a clear explanation for the observed increase in stability \autoref{s:4}. In addition, for flux functions which are homogeneous of degree one, an upper bound was found for the parameter $c$ when forcing the weighted upwind distribution function set to be monotone non-decreasing. This bound, which guarantees \eqref{eq: discrete-kinetic-equation} admits an H-inequality if \eqref{eq: hyperbolic-system} has an convex entropy function \cite{Wissocq_2024_Positive_Preserving_VKLB}, was used in all of our numerical results, showing promising numerical stability across the benchmark problems. 

Numerical experiments on challenging benchmark problems in shallow water, compressible hydrodynamics, and ideal magnetohydrodynamics \autoref{s:5} demonstrated that the weighted upwind distribution function set improves stability relative to discontinuous upwind distribution functions \eqref{eq: upwinding-f1}--\eqref{eq: upwinding-f3} while improving accuracy compared to both the centered flux \eqref{eq: centered-f1}--\eqref{eq: centered-f3} and discontinuous upwind distribution function \eqref{eq: upwinding-f1}--\eqref{eq: upwinding-f3} sets. 

The clearest stability gains occurred in both strong shock and smooth problems where one or more characteristic speeds changed signs. In these cases, the discontinuous upwind distribution function set \eqref{eq: upwinding-f1}--\eqref{eq: upwinding-f3} developed numerical instabilities associated with the transition in characteristic direction. The weighted upwind distribution function set regularized this behavior and suppressed the resulting spurious oscillations. 
 
The ideal MHD Brio--Wu problem illustrated the robustness of the weighted upwind set for more complex wave interactions. The discontinuous upwind distribution function set required a CFL restriction approximately half that of the centered flux and weighted upwind sets; attempts to use the larger CFL led to negative pressure and loss of stability. By contrast, the weighted upwind distribution function set remained stable under the same CFL restriction as the centered flux set, indicating that the added state variable in \eqref{eq: adjusted-f1}--\eqref{eq: adjusted-f3} has does not compromise time-step limits or accuracy.

The weighted upwind distribution function set also provided improved resolution of important flow features. In the shock dominated one-dimensional tests, it more accurately captured shocks, rarefactions, and contact discontinuities than the other two sets. The smooth Euler problems, \eqref{eq: sinusoidal_solution_euler} and \eqref{eq: euler_smooth_vortex}, show that this improved robustness is not obtained at the expense of formal accuracy: the weighted upwind scheme retained second-order convergence while remaining more accurate than the centered flux distribution function set. In the two-dimensional ideal MHD Orszag--Tang problem, the weighted upwind distribution function set converged under mesh refinement toward the highly resolved Athena reference solution \cite{Athena_CPP}. These results demonstrated the the parameter \(c\) adds numerical diffusion into the system, where varying the bounding parameter \(c\) from \(0.05\) to \(0\) improved accuracy for this problem, consistent with reduced artificial dissipation. However, in other test cases the choice \(c = 0\) produced oscillations near shocks for \(\nu = 0.90\). This suggests that while \(c=0.05\) provided a useful balance of stability and accuracy across the full test suite, it may be overly diffusive for some multidimensional MHD problems. 

Overall, the proposed weighted upwind VKLB scheme provides a framework for flux splitting given any hyperbolic system, creating a pathway for constructing stable and accurate upwind VKLB methods for a broad class of hyperbolic systems.
For future work, we will consider analysis approaches that might allow the generalization of the upper bound on $c$ to systems that do not have a flux of homogeneous order one. Having improved the underlying accuracy of WU-VKLB method over the centered flux VKLB, and both the accuracy and stability of the method over the discontinuous upwind VKLB method, we will also seek the development of a high-resolution method that blends the first and second order WU-VKLB methods. While carrying out our development and considering future high-resolution approaches we became aware of the very nice work in \cite{Wissocq_2024_Positive_Preserving_VKLB} that developed one such approach that in addition enforces positivity- and bounds-preserving properties. This and other approaches will be considered that would further increase the accuracy of the proposed WU-VKLB method. 
\\

\noindent \textbf{CRediT authorship contribution statement}\\

\textbf{Michael W. Brown}: Writing – original draft, Writing – review \& editing, Visualization, Validation, Formal analysis, Conceptualization, Software. \textbf{Jehanzeb Chaudhry}: Writing – review \& editing, Validation, Formal analysis, Conceptualization, Funding acquisition. \textbf{John N. Shadid}: Writing – review \& editing, Validation, Formal analysis, Conceptualization, Funding acquisition.\\

\noindent \textbf{Declaration of competing interest}\\

The authors declare that they have no known competing financial interests or personal relationships that could have appeared to influence the work reported in this paper.\\

\noindent \textbf{Data availability}\\

Data will be made available on request.\\

\noindent \textbf{Acknowledgments}\\

This work was supported  by the U.S. Department of Energy, Office of Science (SC), Office of Advanced
Scientific Computing Research’s Applied Mathematics Competitive Portfolios program. Sandia National Laboratories is a multi-mission laboratory managed and operated by National Technology \& Engineering Solutions of Sandia, LLC (NTESS), a wholly owned subsidiary of Honeywell International Inc., for the U.S. Department of Energy’s National Nuclear Security Administration (DOE/NNSA) under contract DE-NA0003525. This written work is authored by an employee of NTESS. The employee, not NTESS, owns the right, title and interest in and to the written work and is responsible for its contents. Any subjective views or opinions that might be expressed in the written work do not necessarily represent the views of the U.S. Government. The publisher acknowledges that the U.S. Government retains a non-exclusive, paid-up, irrevocable, world-wide license to publish or reproduce the published form of this written work or allow others to do so, for U.S. Government purposes. The DOE will provide public access to results of federally sponsored research in accordance with the DOE Public Access Plan.

\printcredits

\bibliographystyle{cas-model2-names}

\bibliography{weighted_upwind_ref}



\end{document}